\documentclass[11pt]{article}
\usepackage{cmap}
\usepackage[T1]{fontenc}
\usepackage{lmodern}
\usepackage{amsmath,amssymb,amsthm}
\usepackage{mathtools}
\usepackage{enumitem}
\usepackage{array}
\usepackage{booktabs}
\usepackage{graphicx}
\usepackage[margin=1.15in]{geometry}
\usepackage{microtype}
\usepackage[section]{placeins}
\usepackage[font=small,labelfont=bf,labelsep=period]{caption}
\usepackage[dvipsnames]{xcolor}
\usepackage[
  backend=biber,
  style=alphabetic,
  sorting=nyt,
  maxbibnames=99,
  maxalphanames=4,
  natbib=true
]{biblatex}
\appto{\bibsetup}{\emergencystretch=1em}
\usepackage{hyperref}
\usepackage[nameinlink,noabbrev,capitalise]{cleveref}
\graphicspath{{figures/}}
\numberwithin{equation}{section}

\newtheorem{theorem}{Theorem}[section]
\makeatletter
\IfFormatAtLeastTF{2026-06-01}{%
  \newcommand*{\newsharedtheorem}[2]{\newtheorem{#1}[theorem]{#2}}%
}{%
  \newcommand*{\newsharedtheorem}[2]{%
    \global\expandafter\let\csname c@#1\endcsname\c@theorem
    \expandafter\gdef\csname theH#1\endcsname{\theHtheorem}%
    \expandafter\gdef\csname p@#1\endcsname{\p@theorem}%
    \expandafter\gdef\csname cl@#1\endcsname{\cl@theorem}%
    \newtheorem{#1}[#1]{#2}%
    \expandafter\gdef\csname the#1\endcsname{\thetheorem}}%
}
\makeatother
\newsharedtheorem{proposition}{Proposition}
\newsharedtheorem{lemma}{Lemma}
\newsharedtheorem{corollary}{Corollary}
\newsharedtheorem{conjecture}{Conjecture}
\theoremstyle{definition}
\newsharedtheorem{definition}{Definition}
\newsharedtheorem{assumption}{Assumption}
\newsharedtheorem{example}{Example}
\theoremstyle{remark}
\newsharedtheorem{remark}{Remark}
\crefname{theorem}{\mbox{Theorem}}{\mbox{Theorems}}
\Crefname{theorem}{\mbox{Theorem}}{\mbox{Theorems}}
\crefname{proposition}{\mbox{Proposition}}{\mbox{Propositions}}
\Crefname{proposition}{\mbox{Proposition}}{\mbox{Propositions}}
\crefname{lemma}{\mbox{Lemma}}{\mbox{Lemmas}}
\Crefname{lemma}{\mbox{Lemma}}{\mbox{Lemmas}}
\crefname{corollary}{\mbox{Corollary}}{\mbox{Corollaries}}
\Crefname{corollary}{\mbox{Corollary}}{\mbox{Corollaries}}
\crefname{conjecture}{\mbox{Conjecture}}{\mbox{Conjectures}}
\Crefname{conjecture}{\mbox{Conjecture}}{\mbox{Conjectures}}
\crefname{definition}{\mbox{Definition}}{\mbox{Definitions}}
\Crefname{definition}{\mbox{Definition}}{\mbox{Definitions}}
\crefname{assumption}{\mbox{Assumption}}{\mbox{Assumptions}}
\Crefname{assumption}{\mbox{Assumption}}{\mbox{Assumptions}}
\crefname{example}{\mbox{Example}}{\mbox{Examples}}
\Crefname{example}{\mbox{Example}}{\mbox{Examples}}
\crefname{remark}{\mbox{Remark}}{\mbox{Remarks}}
\Crefname{remark}{\mbox{Remark}}{\mbox{Remarks}}
\crefname{section}{\mbox{Section}}{\mbox{Sections}}
\Crefname{section}{\mbox{Section}}{\mbox{Sections}}
\crefname{subsection}{\mbox{Section}}{\mbox{Sections}}
\Crefname{subsection}{\mbox{Section}}{\mbox{Sections}}
\crefname{appendix}{\mbox{Appendix}}{\mbox{Appendices}}
\Crefname{appendix}{\mbox{Appendix}}{\mbox{Appendices}}
\crefname{subappendix}{\mbox{Appendix}}{\mbox{Appendices}}
\Crefname{subappendix}{\mbox{Appendix}}{\mbox{Appendices}}
\crefname{figure}{\mbox{Figure}}{\mbox{Figures}}
\Crefname{figure}{\mbox{Figure}}{\mbox{Figures}}
\crefname{table}{\mbox{Table}}{\mbox{Tables}}
\Crefname{table}{\mbox{Table}}{\mbox{Tables}}

\newcommand{\R}{\mathbb R}
\newcommand{\Z}{\mathbb Z}
\newcommand{\T}{\mathbb T}
\newcommand{\RCD}{\mathrm{RCD}}
\newcommand{\Isom}{\mathrm{Isom}}
\newcommand{\Stab}{\mathrm{Stab}}
\newcommand{\Tor}{\operatorname{Tor}}
\newcommand{\pr}{\mathrm{pr}}
\newcommand{\rank}{\operatorname{rank}}
\newcommand{\diam}{\operatorname{diam}}

\def\Xint#1{\mathchoice{\XXint\displaystyle\textstyle{#1}}{\XXint\textstyle\scriptstyle{#1}}{\XXint\scriptstyle\scriptscriptstyle{#1}}{\XXint\scriptscriptstyle\scriptscriptstyle{#1}}\!\int}
\def\XXint#1#2#3{{\setbox0=\hbox{$#1{#2#3}{\int}$}\vcenter{\hbox{$#2#3$}}\kern-.5\wd0}}
\def\fint{\Xint-}

\title{Maximal First Betti Number Drop and Collapsing RCD Spaces}
\author{%
Shaosai Huang\thanks{Kspectra Research Inc., 4789 Yonge Street, Toronto, Ontario, M2N 0G3,
Canada.\newline Email: \texttt{arthur.foxie.huang@kspectra.ai}.}
\and
Xin Peng\thanks{University of Science and Technology of China, No.~96 Jinzhai Road, Hefei,
Anhui Province, 230026, China.\newline Email: \texttt{px2333@mail.ustc.edu.cn}.}}
\date{November 21, 2026}
\hypersetup{
  colorlinks=true,
  linkcolor=MidnightBlue,
  citecolor=OliveGreen,
  urlcolor=BrickRed,
  pdftitle={Maximal First Betti Number Drop and Collapsing RCD Spaces},
  pdfauthor={Shaosai Huang and Xin Peng},
  pdfsubject={Collapsing RCD(K,N) spaces at the maximal drop of the first Betti number:
    the limit symmetry, the orbifold good set, and local and global Seifert torus
    fibrations},
  pdfkeywords={RCD spaces, collapsing, first Betti number, Seifert fibrations, torus
    bundles, orbifolds, Gromov--Hausdorff convergence},
  pdflang={en-US}
}

\begin{document}
\maketitle

\begin{abstract}
Let $(X_i,d_i,\mathfrak m_i)$ be compact $\RCD(K,N)$ spaces converging to a space $X$ of rectifiable dimension $m$. Their first Betti numbers can drop by at most $N-m$. When equality holds, the maximal abelian covers have a non-collapsed limit $Y$ carrying a free isometric $\R^{N-m}$-action. We turn this limiting symmetry into fibrations of the approximating spaces. More precisely, after passing to a subsequence, there is an open full-measure set $G\subset X$, containing every regular point, on which $X$ is a topological orbifold and the $X_i$ admit local Seifert fibrations with $(N-m)$-torus fibres. Each finite local group acts on the fibre by translations. If the $X_i$ have no boundary, then $X\setminus G$ has Hausdorff codimension at least two. If $Y$ has no bubbling and $X$ is a smooth closed Riemannian orbifold, the local fibrations may be chosen as restrictions of a single global Seifert fibration. A global orbit coordinate on the smooth manifold cover shows, in addition, that this fibration is the orbit map of a continuous torus action and becomes a product with $\T^{N-m}$ on finite covers of $X_i$ and $X$ of a fixed degree, as Peng, Wang and Wang proved for smooth torus bundles at the Betti number equality. When the base is a closed Riemannian manifold, the maps are principal torus bundles; this confirms a conjecture of Zamora and Zhu for possibly singular $\RCD$ total spaces. The new ingredients are a pointwise linearization of the collapsing action at regular orbits, an invariant harmonic transverse coordinate compatible with finite isotropy, and an exactly equivariant orbit coordinate built from an elementary rounding-and-doubling argument in the collapsing deck group.
\end{abstract}



\section{Introduction}\label{sec:intro}

The interplay between the first Betti number and the topology of a closed Riemannian manifold with Ricci curvature bounded below has a rich history. The number is bounded above by the curvature and the dimension: Bochner's theorem gives $b_1\le n$ when $\mathrm{Ric}\ge0$ \cite{Boc46}, and Gromov's short-generator argument gives the same bound under the pinching condition $\diam(M)^2\,\mathrm{Ric}_M\ge-\delta(n)$ \cite{Gro81}. Attaining the bound brings rigidity. Equality in Bochner's theorem forces a flat torus \cite{Boc46}, and Gromov's conjecture that $b_1=n$ under the pinching condition forces a manifold diffeomorphic to $\T^n$ was proved by Colding \cite{Col97} up to homeomorphism and by Cheeger--Colding \cite[Appendix~A]{CC97} up to diffeomorphism. See \cite{HRW20} for a survey of these themes.

For collapsing sequences the same pattern appears, with the topology of the limit subtracted. Huang and Wang proved that if closed Riemannian $N$-manifolds $\{X_i\}$ with $\mathrm{Ric}\ge-(N-1)$ converge in the Gromov--Hausdorff sense to a closed $m$-manifold $X$ of bounded geometry, $m<N$, then
\begin{equation}\label{eq:b1ineq}
b_1(X_i)-b_1(X)\ \le\ N-m\qquad\text{for large }i,
\end{equation}
and that at equality the $X_i$ are torus bundles over $X$ \cite[Theorem~1.1]{HW26}, an analogue under lower Ricci bounds of Fukaya's fibration theorem for bounded sectional curvature \cite{Fuk87}. The bundles arising this way are themselves rigid: Peng--Wang--Wang showed that at the same Betti number equality a smooth torus bundle with affine structure group over a closed manifold is principal, and that finite covers of base and total space turn it into a product \cite[Theorem~1.1]{PWW26}.

It is natural to ask whether the above interplay between the first Betti number and lower Ricci curvature bounds generalizes to rougher spaces. The canonical candidates are $\RCD(K,N)$ spaces: metric measure spaces with Ricci curvature bounded below by $K$ and dimension bounded above by $N$ in a synthetic sense. This class is stable under measured Gromov--Hausdorff convergence and therefore contains all Ricci limit spaces. Every $\RCD(K,N)$ space $(Z,d,\mathfrak m)$ has a \emph{rectifiable} (or essential) \emph{dimension} $n\le N$: at $\mathfrak m$-almost every point, the tangent cone is uniquely $\R^n$ \cite{BS20}. We call
\begin{equation*}
\mathcal R(Z):=\{z\in Z:\ \text{every tangent cone of $Z$ at $z$ is $\R^n$}\}
\end{equation*}
the \emph{regular set} of $Z$, and its points \emph{regular}; its complement $\mathcal S(Z):=Z\setminus\mathcal R(Z)$ is the \emph{singular set}. The regular set is dense and has full measure, although it need not be open.

Some of the results above already generalize to $\RCD$ spaces. Under the pinching condition $\diam(Z)^2\,K\ge-\delta(N)$, a compact $\RCD(K,N)$ space $Z$ has $b_1(Z)\le N$, and at equality it is bi-H\"older homeomorphic to a flat $N$-torus \cite{MMP22,ZZ26}. The collapsing bound \eqref{eq:b1ineq} also persists: it holds whenever compact $\RCD(K,N)$ spaces $\{(X_i,d_i,\mathfrak m_i)\}$ converge in the measured Gromov--Hausdorff sense to a space $(X,d,\mathfrak m)$ of rectifiable dimension $m$. This is \cite[Theorem~5]{SRZ23}.

This paper studies the equality case for $\RCD$ spaces. Parts~\textup{(i)}--\textup{(iv)} below require no regularity assumption on $X$ and produce the limit-cover symmetry, the orbifold good set, and local Seifert fibrations. The global-fibration and virtual-product conclusions in part~\textup{(v)}, which extend the corresponding results of \cite{HW26,PWW26}, require no bubbling and a smooth closed Riemannian orbifold base. As a special case of our main result, we confirm Conjecture~2.9 of \cite{ZZ26} that at equality of \eqref{eq:b1ineq} the $X_i$ are torus bundles over $X$ whenever $X$ is a closed Riemannian manifold; see \cref{cor:conj}.

\subsection{Setting and Main Result}
Our main object is the limit of the maximal abelian covers of the $X_i$; in the equality case, these covers do not collapse. Throughout the paper we assume:
\begin{itemize}
\item[$(\ast)$] $(X_i,d_i,\mathfrak m_i)$ are compact $\RCD(K,N)$ spaces with $N\ge2$ and $\diam X_i\le D$, converging in the measured Gromov--Hausdorff sense to $(X,d,\mathfrak m)$ of rectifiable dimension $m$, and
\begin{equation*}
b_1(X_i)-b_1(X)=N-m=:k\ge1\quad\text{for all }i .
\end{equation*}
\end{itemize}
The case $N<2$ is elementary, while $k=0$ is the non-collapsed case; \cref{cor:conj} treats both for manifold limits. Let $\hat X_i\to X_i$ be the maximal abelian covers, with deck groups $H_i=H_1(X_i;\Z)$. After passing to a subsequence, let $(Y,H)$ be the equivariant pointed measured Gromov--Hausdorff limit of $(\hat X_i,H_i)$. We write $P\colon Y\to X$ for the projection and $f_i\colon X_i\to X$ for the Gromov--Hausdorff approximations. By part~(i) below, $Y$ has rectifiable dimension $N$. Thus $\mathcal R(Y)$ consists of the points at which every tangent cone is $\R^N$. It is $H$-invariant but need not be open (\cref{rem:sharp}). We say that $Y$ has \emph{no bubbling}, or that (R) holds, when $\mathcal R(Y)=Y$.

Finite isotropy in the limit is unavoidable, already for surfaces collapsing to an interval (\cref{ex:klein}); in the smooth setting, Ricci-flow smoothing likewise produces local infranil fibrations over controlled Riemannian orbifold bases \cite[Theorem~1.4]{HW22}. The fibrations we produce are therefore Seifert fibrations in the following sense, with the orbifold group of the base point acting on the torus fibre by translations.

\begin{definition}\label{def:seifert}
Let $U\subset X$ be open. A \emph{Seifert $\T^k$-fibration over $U$} is a continuous map $\sigma\colon E\to U$, defined on an open subset $E\subset X_i$, such that every $u\in U$ has a \emph{model neighbourhood} $V_u\subset U$ for which $\sigma^{-1}(V_u)\to V_u$ is equivalent to $(\mathbb D\times\T^k)/C_u\to\mathbb D/C_u$, with $u$ corresponding to the centre of the ball $\mathbb D\subset\R^m$. The finite abelian \emph{local group} $C_u$ acts diagonally on $\mathbb D\times\T^k$: linearly and faithfully on $\mathbb D$, and freely by translations on $\T^k$. Thus the product action is free and the total-space model is a manifold.
\end{definition}

\begin{theorem}[Main Theorem]\label{thm:master}
Assume $(\ast)$.
\begin{enumerate}[label=\textup{(\roman*)}]
\item \emph{(Limit cover and symmetry)} For large $i$, $\mathfrak m_i=c_i\mathcal H^N$, the covers $\hat X_i$ are uniformly non-collapsed, and $Y$ is a non-collapsed $\RCD(K,N)$ space. Moreover,
$H\cong\R^k\times\Z^b\times C$, where $b=b_1(X)$ and $C$ is finite abelian, and the identity component $H_0\cong\R^k$ acts freely on $Y$. Finally, $X=W/\Gamma$, where $W:=Y/H_0$ and $\Gamma:=H/H_0\cong\Z^b\times C$. Every stabilizer $\Gamma_w$ is finite and contained in $C$. For every sufficiently small $r>0$, depending on $w$, a neighbourhood of the image of $w$ in $X$ is isometric to $B_W(w,r)/\Gamma_w$. If this image lies in $\mathcal R(X)$, then $\Gamma_w$ is trivial.
\end{enumerate}
There is an open set $G\subset X$, the \emph{good set}, containing $P(\mathcal R(Y))$ with the following properties.
\begin{enumerate}[label=\textup{(\roman*)},resume]
\item \emph{(Orbifold base)} $G$ is a topological orbifold. Its local group at $x$ is $C_x:=\Gamma_w\le C$, for any $w\in W$ over $x$; these finite abelian groups act linearly and faithfully.
\item \emph{(Size of the good set)} $G\supset\mathcal R(X)$; hence $G$ is dense and of full measure. Moreover $\dim_{\mathcal H}X=m$, and if the $X_i$ have no boundary in the sense of \cite{DPG18}, then ${\dim_{\mathcal H}(X\setminus G)\le m-2}$.
\item \emph{(Local Seifert fibrations)} Every $x\in G$ has neighbourhoods $U_x^0\subset U_x\subset G$ such that, for large $i$, there are Seifert $\T^k$-fibrations $\sigma_i\colon E_i\to U_x$. The point $x$ has a model neighbourhood containing $U_x^0$, with local group $C_x$. Moreover, $\sup_{E_i}d(\sigma_i,f_i)\to0$, and $f_i^{-1}(L)\subset E_i$ for every compact $L\subset U_x$ and all sufficiently large $i$. Near every point with $C_x=1$---in particular, near every point of $\mathcal R(X)$---the maps $\sigma_i$ are $\T^k$-bundles.
\item \emph{(Globalization and virtual product)} If (R) holds, then $G=X$. If, in addition, $X$ is a smooth closed Riemannian orbifold, then for large $i$ the local Seifert fibrations in part~\textup{(iv)} may be taken to be the restrictions of a single Seifert $\T^k$-fibration $\sigma_i\colon X_i\to X$ that is an $o(1)$-GH approximation, with the local group at $x$ conjugate in $O(m)$ to the Riemannian orbifold group $\Gamma_x^{\mathrm{orb}}$. Moreover, on finite covers $\sigma_i$ becomes a product: there are a closed manifold $\widetilde W$, independent of $i$, with an action of $C$ such that $X=\widetilde W/C$ as orbifolds, and regular coverings $\widetilde X_i\to X_i$ with deck group $C$, such that $\sigma_i$ is induced by a $C$-equivariant map $\widetilde X_i\to\widetilde W$ that is the projection of a homeomorphism
\begin{equation*}
\widetilde X_i\cong\widetilde W\times\T^k .
\end{equation*}
Moreover, $\sigma_i$ is the orbit map of a continuous effective $\T^k$-action on $X_i$ whose isotropy group at points over $x$ is isomorphic to $C_x$. If $X$ is a closed Riemannian manifold, then (R) holds and every $C_x$ is trivial; $\sigma_i$ is then a principal $\T^k$-bundle, and $\widetilde W\to X$ is a regular covering with deck group $C$.
\end{enumerate}
\end{theorem}

Part~(i) is proved in \cref{sec:limit}, parts~(ii) and~(iii) in \cref{ssec:setG}, part~(iv) in \cref{ssec:proofiv}, and part~(v) in \cref{sec:global}. All limit-dependent objects---including $Y,H,\Gamma,C,W,W',G$, the quotient maps $\Pi_i\colon H_i\to\Gamma$ of \cref{sec:groups}, and $X_i'$---and every ``for large $i$'' in the construction refer to the chosen subsequence. The covers $W'$ and $X_i'$ also depend on the complement $\Gamma_{\mathrm f}$ fixed in \cref{ssec:setup}. Conclusions involving only $X_i$ and $X$ pass to the full sequence by the usual sub-subsequence argument, provided their hypotheses (such as (R), a condition on $Y$) hold along every subsequence.

The restriction to $G$ is genuine: bubbling can make $G$ strictly smaller than $X$ (\cref{rem:sharp}).

\subsection{Contribution and Proof Spine}
We first say what is imported, then which difficulties the equality case presents and which new techniques overcome them, and then how the four stages fit together.

\emph{Known Inputs.} The analytic inputs are published: the splitting maps and the transformation theorem of Bru\`e, Naber and Semola \cite{BNS22}, and the canonical Reifenberg estimates of Honda and Peng \cite{HP23}. We also use two results of \cite{HHWZ26}. The rank inequality \cite[Lemma~6.1]{HHWZ26}, which enters the chain of inequalities behind \eqref{eq:b1ineq} in the proof of \cite[Theorem~6.12]{HHWZ26}, keeps the maximal abelian covers from collapsing (\cref{prop:chain}); the equivariant regularization over a smooth Riemannian target \cite[Theorem~1.11]{HHWZ26} gives the base factor of the global product (\cref{ssec:regularization}). Other steps have counterparts in recent work, which we credit where they arise but do not use as inputs. The extension of pseudo-homomorphisms \cite[Lemma~6.6]{HHWZ26} is derived from the published work of Santos-Rodr\'iguez and Zamora-Barrera \cite{SRZ23} (\cref{lem:det}), and the convergence of quotients by converging normal subgroups \cite[Lemma~3.1]{Wan23} is proved directly (\cref{lem:quotients}). The affine infranil fibration theorem under locally bounded covering geometry \cite[Theorem~B]{Wan26} is not needed: a global orbit coordinate shows directly that the $C$-cover is a product (\cref{ssec:Kproduct}). Finally, the product conclusion of part~(v) is the topological counterpart of \cite[Theorem~1.1]{PWW26}, quoted above, which we do not import.

\emph{Difficulties and New Techniques.} Maximal drop produces a symmetry only in the limit, and the central theme of this paper is to convert this \emph{asymptotic symmetry on the limit cover} into \emph{exact product coordinates on the approximating spaces}. Three difficulties stand in the way. First, no smooth structure is available: the action of $H_0$ on $Y$ cannot be differentiated, and Ricci-flow smoothing, as in \cite{HW26}, does not apply. We replace the derivative by a \emph{pointwise linearization}. In a harmonic chart $u$ the cocycles $u\circ\varphi_v-u$ are harmonic; a gradient estimate and a doubling iteration extract their linear part at the orbit point (\cref{prop:lip,lem:pointlinear}), and the transformation theorem of \cite{BNS22} carries this linear model to every smaller scale (\cref{lem:linearmodel,lem:fixedlinear}). Second, the action of $H_0$ is non-compact, so it cannot be averaged, while the limit may carry finite isotropy; for a singular limit, none of the results above provides both a slice for this action and a fibration compatible with finite isotropy. We pass to the quotient by a fine lattice in $H_0$, where the residual action is by a compact torus, and average there over the torus and the finite isotropy. The result is the \emph{transverse coordinate} (\cref{prop:U2}): a harmonic map, invariant under $H_0$ and equivariant under finite isotropy, whose local level sets are exactly the $H_0$-orbits. Third, on the approximating spaces the collapsed directions appear only through deck transformations of $\hat X_i$ that act almost by translations, and the local bundles that the transverse coordinate produces (\cref{lem:bundlecrit}) have, a priori, an unknown fibre. An elementary rounding-and-doubling argument maps the collapsing deck group, modulo torsion, onto a lattice in $\R^k$ by an exact homomorphism (\cref{prop:lattice}), and a partition of unity then gives the \emph{orbit coordinate} $U_i$ (\cref{prop:orbit}), exactly equivariant under the deck transformations stabilizing a small tube over a point of $X$. Together with lifted approximants $\hat\Phi_i$ of the transverse coordinate, it forms
\begin{equation*}
\mathcal Z_i=(U_i,\hat\Phi_i)
   \colon \tilde T_i\longrightarrow \R^k\times\R^m,
\end{equation*}
the local product chart on a tube $\tilde T_i\subset\hat X_i$ from which we read off the Seifert fibration; a degree-one argument shows that its fibres are tori (\cref{thm:chart}). For the global statement we glue orbit coordinates, which take values in $\R^k$ and can be averaged, rather than local fibrations (\cref{prop:globalorbit}).

\emph{Proof Spine.} The argument has four stages.
\begin{enumerate}[label=\textup{\arabic*.}]
\item \emph{Symmetry from Equality} (\cref{sec:limit}). Equality produces a non-collapsed limit $Y$ of the maximal abelian covers and a free action $H_0\cong\R^k$; a rank-and-homogeneity argument at a regular-point blow-up supplies the dimension chain.
\item \emph{The Transverse Coordinate} (\cref{sec:norot,sec:eqmaps}). The pointwise linearization gives one linear model of the action at every small scale near a regular orbit. Harmonic replacement on the lattice quotient, followed by averaging, produces the invariant slice and the orbifold charts on $G$.
\item \emph{The Orbit Coordinate} (\cref{sec:groups,sec:local}). Compatible quotient maps $\Pi_i\colon H_i\to\Gamma$ split the deck groups, and equivariant transverse maps yield local bundles. The lattice homomorphism is defined on $\ker\Pi_i$ without passing to a further subsequence, and the orbit coordinate is built from it. Pairing the two coordinates gives the local Seifert chart, which also shows that $\ker\Pi_i\cong\Z^k$.
\item \emph{Globalization} (\cref{sec:global}). Under no bubbling the charts cover $X$. For a smooth orbifold limit, unfolding makes $W'$ a smooth manifold. A global orbit coordinate, exactly equivariant for the whole deck group, pairs with a $C$-equivariant regularization of the quotient approximations \cite[Theorem~1.11]{HHWZ26} into a $C$-equivariant one-sheeted covering of $W'\times\T^k$, with $C$ acting on $\T^k$ by translations. So the $C$-cover $X_i'$ is a product, the global Seifert fibration is its quotient by $C$, and translation of the torus factor makes this fibration the orbit map of a torus action.
\end{enumerate}

\Cref{sec:prelim} fixes notation and records the analytic packages used throughout; a reader familiar with splitting maps may read \cref{ssec:setup,ssec:eqconv}, begin at \cref{sec:limit}, and return to \cref{ssec:packages} whenever a package is invoked. \Cref{app:eqcm} records an averaging over the elements of small displacement in a local group; nothing else in the paper uses it. \Cref{app:notation} indexes the notation and explains the order in which the constants are fixed.

\section{Background and Reusable Tools}\label{sec:prelim}

This section fixes notation and collects the facts about $\RCD$ spaces, isometric actions, and equivariant convergence that are used repeatedly, together with the local analytic packages on which the later constructions rest and an orbit-counting estimate.

\subsection{Setup}\label{ssec:setup}
This subsection fixes the covers and limits used throughout, names the objects that the proof moves between, and states our conventions. Under $(\ast)$, $N=m+k$ is an integer. The maximal abelian cover $\hat X_i\to X_i$ exists because $X_i$ is semi-locally simply connected (see (P3) below). We equip $\hat X_i$ with the lifted measure, normalized to give mass $1$ to the unit ball about a base point $\hat p_i$. After passing to a subsequence, $(\hat X_i,\hat p_i,H_i)\to(Y,\hat p,H)$ in the equivariant pointed measured Gromov--Hausdorff (pmGH) sense.

\emph{The Objects to Keep Apart.} The argument moves between the covers, the limit, and the approximating spaces. The following names are fixed once and for all. \emph{Upstairs} are the maximal abelian covers $\hat X_i\to X_i$ with deck groups $H_i$, and their equivariant limit $(Y,H)$, which does not collapse. The \emph{limit quotients} are the identity component $H_0\cong\R^k$ of $H$; the transverse quotient $W=Y/H_0$, with quotient map $P_W\colon Y\to W$; the intermediate quotient $W'=W/\Gamma_{\mathrm f}$ for a fixed free complement $\Gamma_{\mathrm f}\cong\Z^b$ to $C$ in $\Gamma$; and $X=W'/C$ with $C$ finite abelian. For the \emph{isotropy}, let $y\in Y$ have images $w\in W$ and $w'\in W'$: the stabilizer $C_y:=\Stab_H(y)$ is a finite subgroup of $C$, and $C_y=\Gamma_w=C_{w'}:=\Stab_C(w')$ (\cref{lem:localgroups}); thus $C_x=C_y$ for $x=P(y)$ (\cref{thm:master}(ii)). \emph{Downstairs} are the rank-$k$ fibre groups $H_i':=\ker\Pi_i\le H_i$ of the compatible quotient maps $\Pi_i\colon H_i\to\Gamma$, the finite covers $X_i'=\hat X_i/\mathsf K_i$ on which $C$ acts freely, and the subgroups $\Lambda_i\le H_i$ stabilizing a small tube over a point of $X$ (\cref{sec:groups}). The \emph{coordinates} are the $H_0$-invariant harmonic transverse map $\hat\psi$ and its approximants $\Phi_i$ (\cref{sec:eqmaps,ssec:eqsplit}), the exactly $\Lambda_i$-equivariant orbit map $U_i$, and the product chart $\mathcal Z_i=(U_i,\hat\Phi_i)$ on a tube $\tilde T_i\subset\hat X_i$, where $\hat\Phi_i$ is the lift of $\Phi_i$ to $\hat X_i$ (\cref{ssec:chart}). Finally, the \emph{good set} is the open, full-measure set $G\subset X$ over which the local Seifert fibrations are constructed (\cref{ssec:setG}).

The covers and their limits fit into
\begin{equation*}
\begin{array}{ccccc}
\hat X_i & \longrightarrow & X_i'=\hat X_i/\mathsf K_i & \longrightarrow & X_i=X_i'/C\\[4pt]
\downarrow & & \downarrow & & \downarrow\\[4pt]
Y & \longrightarrow & W'=Y/\mathsf K & \longrightarrow & X=W'/C
\end{array}
\end{equation*}
where the horizontal arrows are quotient maps and the vertical ones denote equivariant measured Gromov--Hausdorff convergence as $i\to\infty$. The groups $\mathsf K_i$ and $\mathsf K$, and the covers $X_i'$, are constructed in \cref{sec:groups}.

\emph{Conventions.} The symbol $\Psi(\epsilon_1,\dots,\epsilon_l\,|\,c_1,\dots,c_p)$ denotes a non-negative function tending to $0$ with $\epsilon_1,\dots,\epsilon_l$, while $c_1,\dots,c_p$ remain fixed. Its value may change from line to line, and arguments are suppressed when clear. Unless stated otherwise, scalar occurrences of $C,c>0$ denote constants depending only on $N$; when used as a group, $C$ always denotes the finite factor in \cref{thm:master}(i). An $\epsilon$-GH approximation has distortion at most $\epsilon$ and $\epsilon$-dense image. A ball $B(z,s)$ is \emph{$\epsilon s$-close to $B^n(s)$} when $(B(z,s),s^{-1}d)$ is $\epsilon$-GH close to $B^n(1)$; \emph{measured} closeness also compares the measures after normalizing both balls to mass $1$. A point is \emph{$n$-regular} if all its tangent cones are $\R^n$. Following \cite[Definition~2.4]{HHWZ26}, rescaled to radius $r$, a map $u\colon B(x,r)\to\R^j$ in the domain of the local Laplacian is a \emph{$(j,\delta)$-splitting map} if
\begin{equation*}
\operatorname{Lip}u\le C(N),\qquad|\Delta u|\le\delta/r,\qquad\fint_{B(x,r)}|\langle\nabla u_a,\nabla u_b\rangle-\delta_{ab}|\le\delta\quad(1\le a,b\le j).
\end{equation*}
We omit $j$ when it is clear. When a Lipschitz bound $L$ is specified, it replaces $C(N)$ in this definition, and constants may then depend on $L$ and $j$. We call $u$ \emph{harmonic} when $\Delta u=0$. Our order of the parameters $j$ and $\delta$ is the reverse of that in \cite{HHWZ26}. The notation of the paper, with the dependence of every named constant, is indexed in \cref{app:notation}.

\subsection{RCD Spaces}
All spaces in the argument are $\RCD$ spaces, their covers and quotients, or their limits. We use the following facts about an $\RCD(K,N)$ space $(Z,d,\mathfrak m)$ with $N\in(1,\infty)$.
\begin{enumerate}[label=(P\arabic*)]
\item \emph{Rectifiable dimension.} We write $\dim Z$ for the rectifiable dimension $n$ of $Z$ (see the Introduction), which exists by \cite{BS20}. If $n=N$, then $N\in\mathbb N$ and $\mathfrak m=c\,\mathcal H^N$ \cite{BGHZ23}. Under pmGH convergence, $\dim$ is lower semicontinuous \cite[Theorem~1.5]{Kit19}.
\item \emph{Rescaling and non-collapsed spaces.} For $\lambda,c>0$, $(Z,\lambda d,c\,\mathfrak m)$ is $\RCD(\lambda^{-2}K,N)$, because Ricci curvature scales as the inverse square of distance. So a hypothesis $|K|r^2\le\delta$ on a ball of radius $r$ is the scale-invariant form of an almost non-negative lower Ricci bound: rescaling cannot achieve it, but, as $K$ is fixed, it holds on every sufficiently small ball. The $\RCD$ condition is stable under pmGH convergence \cite{GMS15}, and being $\RCD(-\epsilon,N)$ for every $\epsilon>0$ implies $\RCD(0,N)$. Consequently, pmGH limits of normalized rescalings $(Z,s_j^{-1}d,z_j)$ with $s_j\to0$ are $\RCD(0,N)$ spaces, since the lower bounds $s_j^2K$ tend to $0$. In particular, every tangent cone is $\RCD(0,N)$. An $\RCD(K,N)$ space of the form $(Z,d,c\,\mathcal H^N)$, $c>0$, is called non-collapsed. Replacing $c\,\mathcal H^N$ by $\mathcal H^N$ keeps the $\RCD(K,N)$ condition, by the scaling above, and gives a non-collapsed space in the sense of \cite[Definition~1.1]{DPG18}; we apply the results of \cite{DPG18} in this normalization. Its tangent cones are metric cones with measure $\mathcal H^N$, and $\mathcal H^N$ is continuous under GH convergence within this class (volume convergence).
\item \emph{Isometries and covers.} $\Isom(Z)$ is a Lie group \cite{Sos18,GSR19}. $Z$ is semi-locally simply connected, in the local form: for every $z$ and $R>0$ there is $r>0$ such that every loop in $B(z,r)$ is contractible in $B(z,R)$ \cite{Wan24}. So every open ball has a universal cover. Covering spaces of $Z$, and quotients of $Z$ by free, properly discontinuous groups of measure-preserving isometries, are $\RCD(K,N)$, with the lifted or quotient measure. Indeed, the projection is a local isomorphism of metric measure spaces; $\RCD(K,N)$ spaces are non-branching \cite{Den25}, so the curvature-dimension condition globalizes \cite{CM21,Li24}.
\item \emph{No small subgroups} \cite[Corollary~3.10]{ZZ26}. Let $(Z_i,d_i,\mathfrak m_i,p_i)$ be pointed $\RCD(K,N)$ spaces converging in the pmGH sense to a space of rectifiable dimension $N$. Then every sequence of subgroups $\mathsf G_i\le\Isom(Z_i)$ with $\sup_{g\in\mathsf G_i,\,x\in B(p_i,R)}d(gx,x)\to0$ for every $R$ is eventually trivial.
\end{enumerate}

The groups that act in this paper are closed groups of isometries of proper spaces, and we use the following standard facts about such actions.

\begin{lemma}[Proper Actions]\label{lem:proper}
Let $Z$ be a proper metric space and $\mathsf G\le\Isom(Z)$ a closed subgroup.
\begin{enumerate}[label=\textup{(\alph*)}]
\item The action is proper: $\{g:gL\cap L\ne\emptyset\}$ is compact for every compact $L\subset Z$. Stabilizers are compact, orbit maps $g\mapsto gz$ are proper and closed, and orbits are closed.
\item $d_{Z/\mathsf G}(\mathsf Gx,\mathsf Gy):=\inf_{g\in\mathsf G}d(x,gy)$ is a metric on $Z/\mathsf G$ inducing the quotient topology, the infimum is attained, $Z/\mathsf G$ is proper, and the projection $\pi$ is $1$-Lipschitz and maps $\bar B(x,r)$ onto $\bar B(\pi x,r)$. If $Z$ is geodesic, so is $Z/\mathsf G$.
\item If $\mathsf G'\lhd\mathsf G$ is closed, then $\mathsf G/\mathsf G'$ acts on $Z/\mathsf G'$ by isometries, continuously and properly; if $\mathsf G/\mathsf G'$ is discrete, the action is properly discontinuous.
\end{enumerate}
\end{lemma}

\begin{proof}
As $Z$ is proper, the compact-open topology on $\Isom(Z)$ is metrizable, and compactness may be tested with sequences.

(a) For compact $L$ and $g_k\in T_L:=\{g\in\mathsf G:gL\cap L\ne\emptyset\}$, the points $g_k^{\pm1}x$ stay in bounded sets, so by Arzel\`a--Ascoli a subsequence has $g_k\to g$ and $g_k^{-1}\to g'$ uniformly on compact sets, with $gg'=g'g=\mathrm{id}$; and $g\in T_L$, as $\mathsf G$ and $L$ are closed. So $T_L$ is compact. The other assertions follow, as $\Stab_{\mathsf G}(z)=T_{\{z\}}$ and preimages of compact sets under orbit maps are closed in sets $T_{L\cup\{z\}}$.

(b) By (a), the infimum is attained, so $\bar d$ is a metric. Attained distances show that $\pi$ maps closed balls onto closed balls of the same radius, which gives the quotient topology and properness, and projections of minimizing geodesics are geodesics.

(c) Normality makes the action well defined, and isometric; continuity follows from the triangle inequality. For compact $L'\subset Z/\mathsf G'$, (b) gives a compact $L$ with $\pi'(L)=L'$, and the elements of $\mathsf G/\mathsf G'$ moving $L'$ to meet itself form the image of $T_L$, which is compact, and finite if $\mathsf G/\mathsf G'$ is discrete.
\end{proof}

Balls that are close to Euclidean balls are the setting of every local construction below, and we record three facts about them. Metric closeness to $\R^m$ on a sufficiently large ball forces measured closeness on smaller balls; by contrast, closeness on only one fixed ball is insufficient: $(x_1+100)^{N-m}\mathcal L^m$ is an $\RCD(0,N)$ measure on $B^m(40)$. On a ball close to Euclidean, harmonic splitting maps exist, and every splitting map is close to a linear map composed with any almost isometry. These are consequences of the work of Bru\`e--Naber--Semola \cite{BNS22} and Ambrosio--Honda \cite{AH18}.

\begin{lemma}[Almost Euclidean Balls]\label{lem:measured}\label{lem:eucsplit}
\begin{enumerate}[label=\textup{(\alph*)}]
\item For every $\psi>0$ and $R\ge1$ there is $\delta>0$, depending only on $\psi$, $R$ and $N$, with the following property. Let $(Z,d,\mathfrak m,z)$ be a pointed $\RCD(-\delta,N)$ space such that $B(z,1/\delta)$ is $\delta$-GH close to $B^m(1/\delta)$. Then, on balls of radius $R$ and in the pointed measured GH sense,
\begin{equation*}
\bigl(Z,d,\mathfrak m/\mathfrak m(B(z,1)),z\bigr)\ \text{is $\psi$-close to}\ \bigl(\R^m,|\cdot|,\omega_m^{-1}\mathcal L^m,0\bigr).
\end{equation*}
\end{enumerate}
In (b) and (c), $(Z,d,\mathfrak m)$ is an $\RCD(K,N)$ space, $q\in Z$, $r>0$, and $1\le j\le n\le N$ are integers; closeness to a Euclidean ball is meant in the measured sense, except that the metric sense suffices when $n=N$ and $Z$ is non-collapsed. For every $\psi\in(0,1]$ and $L\ge1$ there is $\delta=\delta(\psi,L,N)>0$ such that the following hold.
\begin{enumerate}[label=\textup{(\alph*)},resume]
\item If $|K|(8r)^2\le\delta$ and $B(q,8r)$ is $8\delta r$-close to $B^n(8r)$, there is a harmonic map $u$ from $B(q,4r)$ to $\R^n$ with $|\nabla u_a|\le C(N)$ almost everywhere on $B(q,4r)$,
\begin{equation*}
\begin{aligned}
&\fint_{B(q,4r)}|\langle\nabla u_a,\nabla u_b\rangle-\delta_{ab}|\le\psi&&(1\le a,b\le n),\\
&|u(x)-u(x')|\le(1+\psi)\,d(x,x')&&(x,x'\in B(q,r)).
\end{aligned}
\end{equation*}
In particular, $u$ is a harmonic $(n,C(N)\psi)$-splitting map on $B(q,r)$ with Lipschitz bound $2$.
\item Suppose that $|K|r^2\le\delta$, that $B(q,r)$ is $\delta r$-close to $B^n(r)$, and that $\mathcal E\colon B(q,r)\to\R^n$ satisfies $\mathcal E(q)=0$ and $\bigl||\mathcal E(x)-\mathcal E(x')|-d(x,x')\bigr|\le\delta r$ for $x,x'\in B(q,r)$. If $u\colon B(q,r)\to\R^j$ is a $(j,\delta)$-splitting map with Lipschitz bound $L$, then there is a linear map $A\colon\R^n\to\R^j$ with $AA^{\mathsf T}=I_j$ such that $|u-u(q)-A\circ\mathcal E|\le\psi r$ on $B(q,r/2)$.
\end{enumerate}
\end{lemma}

\begin{proof}
\emph{(a)} Otherwise there are counterexamples with $\delta_j\to0$ and normalized measures. By Bishop--Gromov and \cite{GMS15}, a subsequence converges in the pmGH sense to an $\RCD(0,N)$ space (P2) whose pointed metric space is $(\R^m,|\cdot|,0)$. The splitting theorem \cite{Gig26}, applied in $m$ independent directions, forces the limit measure to be a multiple of $\mathcal L^m$, and the normalization makes it $\omega_m^{-1}\mathcal L^m$, a contradiction.

\emph{Reductions for (b) and (c).} Both assertions are invariant under rescaling (P2) and under multiplying $\mathfrak m$ by a constant. When $n=N$ and $Z$ is non-collapsed, metric closeness implies measured closeness, by the compactness argument of (a) and volume convergence (P2). So we may assume measured closeness.

\emph{(b)} Rescale so that $4r=1$. For small $\delta$, \cite[Theorem~3.8(i)]{BNS22}, with $k=n$ and a one-point factor, gives a $\delta_1$-splitting map $u\colon B(q,1)\to\R^n$ in the sense of \cite[Definition~3.1]{BNS22}: $u$ is harmonic, $|\nabla u_a|<C(N)$, and its Gram averages over $B(q,1)$ are below $\delta_1$. By \cite[Remark~3.3]{BNS22}, $u$ is $(1+C(N)\delta_1^{1/2})$-Lipschitz on $B(q,1/4)$. Choose $\delta_1$ with $\delta_1+C(N)\delta_1^{1/2}\le\psi$. The last assertion follows by doubling.

\emph{(c)} Rescale so that $r=1$. If (c) fails, there are counterexamples $(Z_k,q_k,\pi_k,u_k)$ with $\delta_k\to0$. Along a subsequence, realized in a common proper space as in \cite{AH18}, the balls $B(q_k,1)$ converge to $B^n(1)$ with a multiple of Lebesgue measure. The maps $\pi_k$ converge, first on a countable dense set and then everywhere, to a linear isometry, and the $L$-Lipschitz maps $u_k$ converge uniformly to some $u_\infty$. By \cite[Theorem~4.4, Corollary~4.3]{AH18}, $u_\infty$ is harmonic, the energies converge, and $\langle\nabla u_\infty^a,\nabla u_\infty^b\rangle=\delta_{ab}$ almost everywhere. By Weyl's lemma, $u_\infty$ is smooth, and the Bochner formula with $|\nabla u_\infty^a|\equiv1$ gives $\nabla^2u_\infty=0$. So $u_\infty$ is linear with orthonormal rows, hence of the form $A\circ\lim\pi_k$, contradicting the choice of the sequence.
\end{proof}

\subsection{Reusable Local Analytic Tools}\label{ssec:packages}
Several local analytic mechanisms recur throughout the proof. We record them here so that later arguments can focus on equivariance and topology rather than repeat the same estimates.

Throughout this subsection, the displayed radii are fixed multiples of one scale, and they may be replaced by any other fixed multiples in the same order: the conclusions then hold on a concentric ball whose radius is a fixed multiple of the scale, determined by the new multiples, with constants depending on them. In \cref{lem:orbitslice}, with $\alpha r$ and $\beta r$, $0<\alpha<\beta<1$, in place of $r/4$ and $r/2$, the condition on $\varkappa$ becomes $(1+\varkappa)\alpha<(1-\varkappa)\beta^{1+\varkappa}$.

\emph{Volume to Reifenberg.} Given $\epsilon>0$, suppose a ball $B(z,R)$ in a non-collapsed $\RCD(K,N)$ space satisfies $|K|R^2\le\delta$ and $\vartheta_R(z)\ge1-\delta$, where $\delta=\delta(\epsilon,N)>0$ is sufficiently small and
\begin{equation*}
 \vartheta_s(z):=\mathcal H^N(B(z,s))/V_{K,N}(s)
\end{equation*}
is the \emph{Bishop--Gromov ratio}, $V_{K,N}(s)$ being the volume of the radius-$s$ ball in the model space. This ratio is non-increasing in $s$, so the lower bound propagates to every smaller radius. Volume almost rigidity \cite[Theorem~1.6]{DPG18}, followed by volume convergence, then shows that every ball $B(z,s)$ with $s\le R/2$ is $\epsilon s$-close to $B^N(s)$ in the measured sense. If the same control holds at every centre and scale in a region, the metric Reifenberg theorem \cite[Theorem~A.1.1]{CC97} supplies the local topological-manifold structure, with bi-H\"older charts. We refer to this chain of implications as the \emph{volume--Reifenberg package}. In later applications, the only nonstandard step is to obtain the volume-ratio bound uniformly over the required family of centres.

The next lemma combines harmonic replacement with the comparison estimate used after averaging.

\begin{lemma}[Harmonic Replacement and Comparison]\label{lem:harmonicpackage}
Let $(Z,d,\mathfrak m)$ be $\RCD(K,N)$, and let $q\in Z$ and $r>0$ satisfy $|K|r^2\le1$. Assume that
\begin{equation*}
 Z\setminus B(q,11r)\ne\varnothing.
\end{equation*}
Let
\begin{equation*}
B(q,8r)\subset\Omega\subset B(q,9r)
\end{equation*}
be open. Suppose $v\colon\Omega\to\R^j$ satisfies
\begin{equation*}
\operatorname{Lip}v\le L,\qquad |\Delta v|\le\delta/r,
\end{equation*}
and let $h$ be its harmonic replacement on $\Omega$, characterized by
$h-v\in W^{1,2}_0(\Omega;\R^j)$. Then
\begin{equation*}
\|h-v\|_{L^\infty(\Omega)}\le C\delta r,\qquad
\fint_{B(q,3r)}|\nabla(h-v)|^2\le C\delta^2,\qquad
\operatorname{Lip}h\le C
\quad\hbox{on }B(q,6r),
\end{equation*}
where $C=C(N,L,j)$. Thus harmonic replacement preserves a splitting estimate on a smaller concentric ball, at the cost of replacing $\delta$ by $C\delta$. If a compact group of measure-preserving isometries preserves $\Omega$ and $v$ is equivariant for an orthogonal target action, then $h$ remains equivariant.

Moreover, even without the hypothesis $Z\setminus B(q,11r)\ne\varnothing$, if $v_1,v_2\colon B(q,9r)\to\R^j$ have Lipschitz constants at most $L$, Laplacians bounded by $\delta/r$, and $|v_1-v_2|\le\delta r$, then
\begin{equation*}
\fint_{B(q,5r)}|\nabla(v_1-v_2)|^2\le C\delta^2.
\end{equation*}
\end{lemma}

\begin{proof}
The exterior point makes the zero-boundary Dirichlet form on $B(q,10r)$ coercive \cite[(4.11)]{AH18}. So \cite[Lemma~4.7]{AH18} gives the torsion function $\omega\in W^{1,2}_0(B(q,10r))$ with $\Delta\omega=-1$. It is non-negative, and $\omega\le C(N)r^2$ by the Sobolev inequality, available as $\RCD(K,N)$ spaces are doubling and support a Poincar\'e inequality \cite{Raj12}, \cite[Theorem~5.1]{HK00}, and Stampacchia's lemma \cite[Lemma~B.1]{KS00}. Testing the positive parts of $\pm(h-v)-(\delta/r)\omega$, which have non-negative Laplacian, gives $|h-v|\le C\delta r$. Testing $\Delta(h-v)=-\Delta v$ with $h-v$ and using doubling gives the energy estimate.

For the Lipschitz bound, the maximum principle, obtained by testing $\Delta h_a=0$ with $(h_a-\sup v_a)_+$ and $(\inf v_a-h_a)_+$, keeps $h$ in the range of $v$, which has diameter at most $18Lr$. Jiang's gradient estimate \cite[Theorem~1.1]{Jia14} on the balls $B(y,2r)$, $y\in B(q,6r)$, then gives $|\nabla h|\le C(N)L$ on $B(q,7r)$. The Sobolev--Poincar\'e inequality \cite[(2.7)]{AH18} on chains of balls converts this into the Lipschitz bound on $B(q,6r)$. If $v$ is a $(j,\delta)$-splitting map on $B(q,3r)$, the energy estimate and the Cauchy--Schwarz inequality bound the Gram error of $h$ by $C\delta$.

Uniqueness for the coercive Dirichlet problem on $\Omega$ gives equivariance. Finally, \cite[Lemma~3.3]{Jia14}, applied to $v_1-v_2$ with a cut-off in $B(q,9r)$, gives the comparison estimate without the exterior point.
\end{proof}

The first consequence makes approximate equivariance exact without losing the splitting estimates.

\begin{corollary}[Equivariantization by Averaging]\label{lem:averaging}
Let a compact group $\mathsf G$ act by measure-preserving isometries on an $\RCD(K,N)$ space $(Z,d,\mathfrak m)$, let $q\in Z$ and $r>0$ satisfy $|K|r^2\le1$, and let $\rho\colon\mathsf G\to O(j)$ be a continuous homomorphism. Let $u$ satisfy the Lipschitz and Laplacian bounds of \cref{lem:harmonicpackage} on a $\mathsf G$-invariant open set containing $B(q,9r)$. Assume that $u$ is $(j,\delta)$-splitting on $B(q,5r)$ and
\begin{equation*}
 \sup_{g\in\mathsf G}\|\rho(g)^{-1}u\circ g-u\|_{L^\infty(B(q,9r))}\le\delta r.
\end{equation*}
With normalized Haar measure, put $u_{\mathsf G}:=\int_{\mathsf G}\rho(g)^{-1}u\circ g\,dg$.
Then $u_{\mathsf G}$ is exactly $\rho$-equivariant and $(j,C\delta)$-splitting on $B(q,5r)$, where $C=C(N,L,j)$, and
\begin{equation*}
 \operatorname{Lip}u_{\mathsf G}\le L,\qquad |\Delta u_{\mathsf G}|\le\delta/r,\qquad
 \|u_{\mathsf G}-u\|_{L^\infty(B(q,9r))}\le\delta r.
\end{equation*}
\end{corollary}

\begin{proof}
Write $u_g:=\rho(g)^{-1}u\circ g$. The substitution $k=gh$ and the invariance of Haar measure give $u_{\mathsf G}(hx)=\rho(h)u_{\mathsf G}(x)$. Each $u_g$ is $L$-Lipschitz with $|\Delta u_g|\le\delta/r$, because $g$ preserves $d$ and $\mathfrak m$ and $\rho(g)$ is orthogonal. By Fubini, both bounds pass to the average, the second in the weak sense, and averaging $|u_g-u|\le\delta r$ gives the $L^\infty$ bound. The comparison clause of \cref{lem:harmonicpackage}, applied to $u_{\mathsf G}$ and $u$, and the Cauchy--Schwarz inequality give the splitting estimate.
\end{proof}

The second consequence allows prescribed coordinates to replace the corresponding coordinates of a full splitting chart.

\begin{lemma}[Stability under Coordinate Replacement]\label{lem:replacement}
Let $(Z,d,\mathfrak m)$ be $\RCD(K,N)$, let $q\in Z$ and $r>0$ satisfy $|K|r^2\le1$, and let
\begin{equation*}
 \xi,F\colon B(q,2r)\longrightarrow\R^N
\end{equation*}
be in the domain of the local Laplacian. Suppose
\begin{equation*}
 \operatorname{Lip}\xi+\operatorname{Lip}F\le L,\qquad
 |\Delta\xi|+|\Delta F|\le\delta/r,\qquad
 \|F-\xi\|_{L^\infty(B(q,2r))}\le\delta r,
\end{equation*}
and $\xi$ is an $(N,\delta)$-splitting map on $B(q,r)$. Then
\begin{equation*}
 \fint_{B(q,r)}|\nabla(F-\xi)|^2\le C\delta^2,
\end{equation*}
and $F$ is an $(N,C\delta)$-splitting map on $B(q,r)$, where $C=C(N,L)$. Writing $\xi=(\xi',\xi'')$, this includes the case $F=(\xi',v)$, in which only the block $\xi''$ is replaced.

Suppose additionally that $Z$ is non-collapsed, $|K|r^2\le\delta$, and every ball $B(a,s)$ with $a\in B(q,r)$ and $s\le r$ is $\delta s$-close to $B^N(s)$. If $\delta=\delta(N,L)$ is sufficiently small, then $F$ is a bi-H\"older embedding on $B(q,r/2)$. More precisely, for $x,y\in B(q,r/2)$,
\begin{equation*}
(1-\varkappa)r^{-\varkappa}d(x,y)^{1+\varkappa}
 \le |F(x)-F(y)|\le(1+\varkappa)d(x,y),
\qquad \varkappa=\varkappa(\delta\mid N,L)\longrightarrow0.
\end{equation*}
Its image $F(B(q,r/2))$ is open.
\end{lemma}

\begin{proof}
The comparison clause of \cref{lem:harmonicpackage}, applied to $F$ and $\xi$, gives the energy estimate. The Gram estimate follows by Cauchy--Schwarz and the Lipschitz bounds.

For the bi-H\"older estimate, rescale so that $r=1$, give $Z$ the measure $\mathcal H^N$ (P2), and write $\Psi=\Psi(\delta|N,L)$. Lipschitz maps with bounded Laplacian are regular in the sense of \cite[Definition~3.1]{HP23}, by the chain rule \cite[Proposition~4.28, Theorem~4.29]{Gig15}, and $Z$ has essential dimension $N$.
\begin{itemize}
\item \emph{Large scales.} For a $\delta$-GH approximation $\mathcal E$ of $B(q,1)$ onto $B^N(1)$ with $\mathcal E(q)=0$, \cref{lem:eucsplit}(c), with $j=n=N$, gives a linear isometry $A$ with $|F-F(q)-A\mathcal E|\le\Psi$ on $B(q,1/2)$. So $F$ changes distances on $B(q,1/2)$ by at most $\Psi$.
\item \emph{The upper bound.} For $x\in B(q,1/2)$, compare $F$ with its harmonic replacement $h$ on $B(x,2/9)$ (\cref{lem:harmonicpackage}). By \cite[Lemma~4.3]{HP23}, $h$ is $(1+C\delta)$-Lipschitz on $B(x,1/48)$, and by \cite[Corollary~3.3]{HP23}, $F-h$ is $C\delta$-Lipschitz there.
\item \emph{The lower bound.} By compactness and volume convergence, the flatness hypothesis of \cite[Proposition~4.7]{HP23} holds on balls of a fixed radius $s_1=s_1(N,\epsilon_{\mathrm{HP}})$ centred in $B(q,1/2)$, once $\delta$ is small. That proposition then gives $|F(z)-F(z')|\ge(1-\Psi_{\mathrm{HP}})d(z,z')^{1+\Psi_{\mathrm{HP}}}$ on these balls, where $\Psi_{\mathrm{HP}}\to0$ as $\epsilon_{\mathrm{HP}}\to0$.
\end{itemize}
Combining the scales gives the estimate with $\varkappa:=\max\{\Psi_{\mathrm{HP}},\Psi/s_1,48\Psi,C\delta\}$; here $C=C(N,L)$, and $\epsilon_{\mathrm{HP}}$ is a function of $\delta$ that tends to $0$ so slowly that $\Psi/s_1(N,\epsilon_{\mathrm{HP}})\to0$. The metric Reifenberg theorem \cite[Theorem~A.1.1]{CC97} makes $B(q,r/2)$ a topological $N$-manifold, so the image is open by invariance of domain.
\end{proof}

The next elementary criterion packages the open-and-closed argument that
turns a full Reifenberg chart with invariant transverse coordinates into an
exact local orbit slice.

\begin{lemma}[Orbit--Slice Criterion]\label{lem:orbitslice}
Let a group $A\cong\R^k$ act freely and properly on a proper metric space
$Z$.  Suppose that $\mathfrak b$ is $A$-invariant, that is, $\mathfrak b(gx)=\mathfrak b(x)$ whenever $g\in A$ and $x,gx\in B(z,r)$, and that
\begin{equation*}
 \Upsilon=(\mathfrak a,\mathfrak b)\colon B(z,r)\longrightarrow\R^k\times\R^m
\end{equation*}
is a homeomorphism onto an open image.  Assume that, for
$x,y\in B(z,r)$,
\begin{equation*}
 (1-\varkappa)r^{-\varkappa}d(x,y)^{1+\varkappa}
 \le |\Upsilon(x)-\Upsilon(y)|
 \le(1+\varkappa)d(x,y),
\end{equation*}
where $(1+\varkappa)/4<(1-\varkappa)2^{-1-\varkappa}$.  Then
\begin{equation*}
 \mathfrak b^{-1}(\mathfrak b(z))\cap B(z,r/4)=Az\cap B(z,r/4).
\end{equation*}
\end{lemma}

\begin{proof}
Put $c':=(1-\varkappa)2^{-1-\varkappa}$.  Properness and openness of the embedding
give
\begin{equation*}
 \partial\bigl(\Upsilon(B(z,r/2))\bigr)
 \subset\Upsilon(\partial B(z,r/2)).
\end{equation*}
The lower estimate and radial boundary crossing therefore imply
\begin{equation*}
 B(\Upsilon(z),c'r)\subset\Upsilon(B(z,r/2)).
\end{equation*}
Let $L$ be the inverse image of
$B^k(\mathfrak a(z),c'r)\times\{\mathfrak b(z)\}$.  Then $L\subset B(z,r/2)$ and
$\mathfrak a|_L$ is a homeomorphism onto $B^k(\mathfrak a(z),c'r)$, so $L$ is connected.  If
$p\in B(z,r/4)$ and $\mathfrak b(p)=\mathfrak b(z)$, the upper estimate gives
$|\mathfrak a(p)-\mathfrak a(z)|<(1+\varkappa)r/4<c'r$, hence $p\in L$.

The set $Az\cap L$ is non-empty and closed in $L$, because proper orbits
are closed.  It is also open: near each of its points, the map from the
$A$-parameter to the first coordinate $\mathfrak a$ is a continuous injection
between open subsets of $\R^k$, so invariance of domain applies.  Thus
$L\subset Az$.  This proves one inclusion; the other follows from the
$A$-invariance of $\mathfrak b$.
\end{proof}

The following topological promotion turns the local transverse maps of \cref{sec:local} into bundles.

\begin{lemma}[From Submersions to Bundles]\label{lem:bundlepromotion}
Let $M$ be a topological $N$-manifold without boundary, let $B$ be a
connected topological $m$-manifold without boundary, and let
$p\colon M\to B$ be a non-empty proper topological submersion. Then $p$ is
surjective and is a fibre bundle whose fibre is a closed topological
$(N-m)$-manifold.
\end{lemma}

\begin{proof}
The image of $p$ is open by the local product definition of a topological
submersion and closed by properness, hence equals $B$. A proper topological
submersion is a fibre bundle \cite[Essay~II, Section~1]{KS77}. Its fibres are
compact by properness and are topological $(N-m)$-manifolds without boundary
in the local product charts.
\end{proof}

The last tool of this subsection compares the sizes, for a Haar measure, of the sets of group elements that move a point by at most two different distances; part~(c) bounds in the same way the elements that move the point into a given ball, which is the number of translates of a cutoff that are active at a point. The lemma bounds the overlap of the partitions of unity in \cref{ssec:chart}, where the group is discrete and the measure counts elements, and the weights of the averages in \cref{app:eqcm}.

\begin{lemma}[Orbit Counting]\label{lem:orbitcount}
Let $(\mathcal X,d,\mathfrak m)$ be a metric measure space, $p\in\mathcal X$, and $\bar R>0$, $C_{\mathrm{dbl}}\ge1$ constants with $0<\mathfrak m(B(p,2t))\le C_{\mathrm{dbl}}\,\mathfrak m(B(p,t))<\infty$ for $0<t\le\bar R$. Let a locally compact group $\mathsf G$, with a left Haar measure $\mu_{\mathsf G}$, act continuously on $\mathcal X$; if $\mathsf G$ is discrete, $\mu_{\mathsf G}$ is counting measure. Suppose that $\mathsf N(p,r):=\mu_{\mathsf G}(\{\gamma\in\mathsf G:d(\gamma p,p)\le r\})$ is finite for all $r$, and let $0<r\le R\le\bar R/4$.
\begin{enumerate}[label=\textup{(\alph*)}]
\item If $\mathsf G$ acts by measure-preserving isometries, then $\mathsf N(p,R)\le C_{\mathrm{dbl}}^{\lceil\log_2(16R/r)\rceil}\,\mathsf N(p,r)$.
\item Let $0\le\eta\le1/10$, and suppose that every $\gamma\in\mathsf G$ with $d(\gamma p,p)\le\bar R$ has the following properties: $\gamma$ and $\gamma^{-1}$ are $(1+\eta)$-Lipschitz on $B(p,\bar R)$, and $\mathfrak m(\gamma A)\ge(1-\eta)\,\mathfrak m(A)$ for every Borel set $A\subseteq B(p,\bar R)$. Then $\mathsf N(p,R)\le(1-\eta)^{-1}C_{\mathrm{dbl}}^{\lceil\log_2(18R/r)\rceil}\,\mathsf N(p,r)$.
\item In case (a), if $\mathsf G$ is unimodular, $q\in\mathcal X$ and $0<r_1\le r_2$ with $2r_2\le\bar R/4$, then
\begin{equation*}
 \mu_{\mathsf G}\bigl(\{\gamma\in\mathsf G:d(\gamma^{-1}p,q)<r_2\}\bigr)\le C_{\mathrm{dbl}}^{\lceil\log_2(32r_2/r_1)\rceil}\,\mathsf N(p,r_1).
\end{equation*}
\end{enumerate}
\end{lemma}

\begin{proof}
In case (a) put $\eta:=0$; isometries are $1$-Lipschitz. Let $S:=\{\gamma p:d(\gamma p,p)\le R\}$, and let $\gamma_1p,\dots,\gamma_Mp\in S$ have pairwise distances at least $r/3$. The balls $B(\gamma_lp,r/8)$ are pairwise disjoint and lie in $B(p,2R)$. In case (a) each has measure $\mathfrak m(B(p,r/8))$, and doubling at $p$ gives $M\,\mathfrak m(B(p,r/8))\le\mathfrak m(B(p,2R))\le C_{\mathrm{dbl}}^{\,j}\,\mathfrak m(B(p,r/8))$ with $j:=\lceil\log_2(16R/r)\rceil$. In case (b), $\gamma_l$ maps $B(p,r/9)$ into $B(\gamma_lp,r/8)$, so each ball has measure at least $(1-\eta)\,\mathfrak m(B(p,r/9))$, and the same argument with $j:=\lceil\log_2(18R/r)\rceil$ gives $M\le(1-\eta)^{-1}C_{\mathrm{dbl}}^{\,j}$. Hence a maximal family of this kind exists; fix one. If $d(\gamma p,p)\le R$, then $d(\gamma p,\gamma_lp)<r/3$ for some $l$, by maximality. Since $\gamma_l^{-1}$ is $(1+\eta)$-Lipschitz on $B(p,\bar R)$, which contains both points, $d(\gamma_l^{-1}\gamma p,p)\le(1+\eta)r/3<r$. Thus $\gamma$ lies in the $\gamma_l$-translate of $\{\gamma':d(\gamma'p,p)\le r\}$, and left invariance of $\mu_{\mathsf G}$ gives $\mathsf N(p,R)\le M\,\mathsf N(p,r)$.

(c) Let $A$ be the set on the left, and fix $\gamma_1\in A$; if $A$ is empty there is nothing to prove. For $\gamma\in A$ we have $d(\gamma_1\gamma^{-1}p,p)=d(\gamma^{-1}p,\gamma_1^{-1}p)<2r_2$, so $A$ lies in $B^{-1}\gamma_1$, where $B:=\{\gamma':d(\gamma'p,p)\le2r_2\}$. Inversion and right translation preserve $\mu_{\mathsf G}$, because $\mathsf G$ is unimodular, so $\mu_{\mathsf G}(A)\le\mathsf N(p,2r_2)$, and (a) with $R=2r_2$ and $r=r_1$ gives the claim.
\end{proof}

The averages of translates in \cref{sec:eqmaps,sec:local} and in \cref{app:eqcm} are exactly equivariant for one reason: moving the base point reindexes the family of translates, and the reindexing preserves the measure on the index set. We record this once.

\begin{lemma}[Equivariance by Reindexing]\label{lem:reindex}
Let a unimodular locally compact group $\mathsf G$, with Haar measure $\mu_{\mathsf G}$, act on a set $\mathcal X$ by a homomorphism $\rho$ into the bijections of $\mathcal X$, and on a Riemannian manifold $M$ by a homomorphism $\rho_M\colon\mathsf G\to\Isom(M)$. Let $\mathcal A\subseteq\mathsf G$ be a Borel set. For $a$ in a finite set $\mathcal I$, let $\omega_a\colon\mathcal X\to[0,\infty)$ and $\xi_a\colon\{\omega_a>0\}\to M$ be maps such that, for every $p\in\mathcal X$, the function $g\mapsto\omega_a(\rho(g)p)$ is Borel and integrable on $\mathcal A$ and $g\mapsto\xi_a(\rho(g)p)$ is Borel where it is defined. For $p\in\mathcal X$ let $\nu_p$ be the finite measure on $M$ with
\begin{equation*}
 \int b\,d\nu_p=\sum_{a\in\mathcal I}\int_{\mathcal A}\omega_a\bigl(\rho(g)p\bigr)\,b\bigl(\rho_M(g)^{-1}\xi_a(\rho(g)p)\bigr)\,d\mu_{\mathsf G}(g)
\end{equation*}
for bounded Borel functions $b$ on $M$. Let $p\in\mathcal X$ and $g_0\in\mathsf G$, and assume that $\omega_a(\rho(g)p)=0$ for every $a\in\mathcal I$ and every $g$ in the symmetric difference of $\mathcal A$ and $\mathcal Ag_0$; this holds for all $p$ and $g_0$ if $\mathcal A=\mathsf G$. Then $\nu_{\rho(g_0)p}=\rho_M(g_0)_*\nu_p$; in particular the two measures have the same mass. Consequently $\operatorname{cm}(\nu_{\rho(g_0)p})=\rho_M(g_0)\operatorname{cm}(\nu_p)$, where $\operatorname{cm}(\nu_q)$ denotes the centre of mass of the probability measure $\nu_q/\nu_q(M)$: its mean if $M$ is a Euclidean space, and the Riemannian centre of mass of \cref{eqcm:lem:cm}(a) otherwise, if that is defined at $q=p$ and at $q=\rho(g_0)p$.
\end{lemma}

\begin{proof}
For $g\in\mathcal A$ we have $\rho(g)\rho(g_0)p=\rho(gg_0)p$ and $\rho_M(g)^{-1}=\rho_M(g_0)\rho_M(gg_0)^{-1}$. The map $g\mapsto gg_0$ carries $\mathcal A$ onto $\mathcal Ag_0$ and preserves $\mu_{\mathsf G}$, because $\mathsf G$ is unimodular. Hence
\begin{equation*}
 \int b\,d\nu_{\rho(g_0)p}=\sum_a\int_{\mathcal Ag_0}\omega_a\bigl(\rho(h)p\bigr)\,b\bigl(\rho_M(g_0)\rho_M(h)^{-1}\xi_a(\rho(h)p)\bigr)\,d\mu_{\mathsf G}(h)=\int b\circ\rho_M(g_0)\,d\nu_p ,
\end{equation*}
where the last equality uses that the integrand vanishes on the symmetric difference of $\mathcal Ag_0$ and $\mathcal A$. The mean on a Euclidean space commutes with its isometries, and so does the centre of mass, by \cref{eqcm:lem:cm}(a).
\end{proof}

The lemma covers every exact equivariance of an average in this paper. \Cref{lem:averaging} and \cref{prop:eqsplit} are the cases of a compact group with $\mathcal A=\mathsf G$, one chart, a weight equal to $1$ on an invariant set, and a linear target. In \cref{prop:orbit}, $\mathsf G=\mathcal A=\Lambda_i$ acts by deck transformations, $M=\R^k$, and $\rho_M(\lambda)$ is the translation by the lattice coordinate $\ell_i(\lambda)$. \Cref{app:eqcm} uses a proper subset $\mathcal A$ and the centre of mass.

\subsection{Equivariant Convergence}\label{ssec:eqconv}
The deck groups are compared with their limit through equivariant convergence, whose conventions we fix here. We use equivariant Gromov--Hausdorff convergence in the sense of Fukaya--Yamaguchi \cite{FY92}, written $(Z_i,p_i,\mathsf G_i)\to(Z,p,\mathsf G)$, with $\epsilon_i$-approximations $(f_i,\phi_i,\psi_i)$, $\epsilon_i\to0$, as in \cite[Section~2.1]{HHWZ26}. With $\mathsf G_i(R):=\{g\in\mathsf G_i:d(gp_i,p_i)\le R\}$ and $\mathsf G(R):=\{g\in\mathsf G:d(gp,p)\le R\}$, these are maps $f_i\colon B(p_i,1/\epsilon_i)\to Z$, $\phi_i\colon\mathsf G_i(1/\epsilon_i)\to\mathsf G(1/\epsilon_i)$ and $\psi_i\colon\mathsf G(1/\epsilon_i)\to\mathsf G_i(1/\epsilon_i)$ such that, whenever all points involved lie in $B(p_i,1/\epsilon_i)$,
\begin{enumerate}[label=\textup{(E\arabic*)}]
\item\label{ax:E1} $|d(f_i(z),f_i(z'))-d(z,z')|<\epsilon_i$, and the $\epsilon_i$-neighbourhood of the image of $f_i$ contains $B(p,1/\epsilon_i)$;
\item\label{ax:E2} $d\bigl(f_i(gz),\phi_i(g)f_i(z)\bigr)<\epsilon_i$ for $g\in\mathsf G_i(1/\epsilon_i)$;
\item\label{ax:E3} $d\bigl(f_i(\psi_i(h)z),hf_i(z)\bigr)<\epsilon_i$ for $h\in\mathsf G(1/\epsilon_i)$.
\end{enumerate}
We may and do assume $f_i(p_i)=p$. When $Z_i$ and $Z$ are proper geodesic spaces and the groups are closed, subgroups and quotients pass to the limit as follows. For $\mathsf G_i'=\mathsf G_i$, part~(b) is the convergence $Z_i/\mathsf G_i\to Z/\mathsf G$ of Fukaya and Yamaguchi \cite{FY92}. The general case is \cite[Lemma~3.1]{Wan23}, whose proof uses a convergence of pairs of groups that is not defined there; we sketch a direct proof.

\begin{lemma}[Subgroups and Quotients]\label{lem:quotients}
\begin{enumerate}[label=\textup{(\alph*)}]
\item Closed subgroups $\mathsf G_i'\le\mathsf G_i$ subconverge to a closed subgroup $\mathsf G'\le\mathsf G$ \emph{together with} $\mathsf G_i\to\mathsf G$: after passing to a subsequence, there are $\epsilon_i'\to0$ and maps $\phi_i'\colon\mathsf G_i'(1/\epsilon_i')\to\mathsf G'(1/\epsilon_i')$ and $\psi_i'\colon\mathsf G'(1/\epsilon_i')\to\mathsf G_i'(1/\epsilon_i')$ such that $(f_i,\phi_i',\psi_i')$ satisfy \ref{ax:E1}--\ref{ax:E3} with $\epsilon_i'$.
\item If, in addition, the $\mathsf G_i'$ are normal, then $\mathsf G'\lhd\mathsf G$, and $(Z_i/\mathsf G_i',\mathsf G_i/\mathsf G_i')\to(Z/\mathsf G',\mathsf G/\mathsf G')$, where the groups act through their images in the isometry groups of the quotients. The approximations are induced by $(f_i,\phi_i,\psi_i)$: uniformly on bounded sets, the space maps send $[z]$ to $[f_i(z)]$ up to $o(1)$, and the group maps send the class of $\gamma\in\mathsf G_i$ to that of $\phi_i(\gamma)$, and the class of $g\in\mathsf G$ to that of $\psi_i(g)$, up to $o(1)$ in the topology of uniform convergence on compact sets.
\end{enumerate}
\end{lemma}

\begin{proof}[Sketch of proof]
(a) The sets $\mathsf G(R)$ are compact (\cref{lem:proper}(a)). By \ref{ax:E1} and \ref{ax:E2}, $\phi_i(g_i)\to g$ if and only if $f_i(g_iz_i)\to gz$ whenever $f_i(z_i)\to z$. By Blaschke's selection theorem and a diagonal argument, after passing to a subsequence the closures of the sets $\phi_i(\mathsf G_i'(R))$, $R\in\mathbb N$, converge in the Hausdorff sense; let $\mathsf G'$ be the union of the limits. By the criterion above, limits of products and inverses are products and inverses of limits, so $\mathsf G'$ is a closed subgroup. As $\mathsf G'$ has a countable base, the displacement $g\mapsto d(gp,p)$ jumps on $\mathsf G'$ at only countably many radii. For radii $r_k\to\infty$ outside this set, the closures of $\phi_i(\mathsf G_i'(r_k))$ converge to $\mathsf G'(r_k)$. Nearest-point choices between these Hausdorff-close sets, with $k=k_i\to\infty$ slowly, give $\phi_i'$ and $\psi_i'$.

(b) Normality passes to the limit: $\psi_i(h)^{-1}\psi_i'(g)\psi_i(h)\in\mathsf G_i'$ converges to $h^{-1}gh$ by \ref{ax:E2} and \ref{ax:E3}. Distances in the quotients are attained (\cref{lem:proper}(b)). Lifts of one point of $Z_i/\mathsf G_i'$ differ by an element of $\mathsf G_i'$, which $f_i$ intertwines, up to $\epsilon_i'$, with an element of $\mathsf G'$; so $[z]\mapsto[f_i(z)]$, evaluated at lifts $z$ of minimal distance to $p_i$, is well defined up to $\epsilon_i'$. Applied to minimizing elements of $\mathsf G_i'$ and of $\mathsf G'$, the same fact shows that this map changes distances by $O(\epsilon_i')$ on balls of radius $o(1/\epsilon_i')$ and has $O(\epsilon_i')$-dense image. Every element of either quotient group is represented by an element with the same displacement at the base point. On such representatives, the maps $[\gamma]\mapsto[\phi_i(\gamma)]$ and $[g]\mapsto[\psi_i(g)]$ satisfy \ref{ax:E2} and \ref{ax:E3} with errors $O(\epsilon_i')$, by \ref{ax:E2} and \ref{ax:E3} for $(f_i,\phi_i,\psi_i)$. The image of $\mathsf G$ in $\Isom(Z/\mathsf G')$ is closed, by \cref{lem:proper}(a), and it is the limit group.
\end{proof}

All approximations of quotients below, including $f_i\colon X_i\to X$ and $f_i'\colon X_i'\to W'$, are of this kind.

\section{The Non-Collapsed Limit of the Abelian Covers}\label{sec:limit}

Under $(\ast)$, every inequality used to prove \eqref{eq:b1ineq} is an equality. We extract two consequences that drive the rest of the paper. First, the cover limit $Y$ is non-collapsed (\cref{prop:chain}). Second, over regular points of $X$, every blow-up of the action is translational on an $\R^k$-factor (\cref{lem:blowupreg}). Together with the orbit-dimension bound of \cref{lem:orbitdim}, the second fact excludes a torus factor in $H_0$ and implies that the action is free.

The limit group $H$ is a closed abelian subgroup of the Lie group $\Isom(Y)$, and $X=Y/H$ \cite{FY92}. Because $Y$ is connected and $Y/H$ is compact, $H$ is compactly generated. Thus $H=\R^{k_1}\times\Z^{k_2}\times C'$ for some compact group $C'$.

\begin{proposition}[Dimension Chain]\label{prop:chain}
$k_1=k$, $k_2=b_1(X)=:b$, and $\dim Y=N$. Moreover $\mathfrak m_i=c_i\mathcal H^N$ for large $i$, the covers $\hat X_i$ are uniformly non-collapsed, and $Y$ is a non-collapsed $\RCD(K,N)$ space.
\end{proposition}

\begin{proof}
The proof of \cite[Theorem~6.12]{HHWZ26} uses the chain
\begin{equation*}
b_1(X)\ \ge\ k_2\ \ge\ b_1(X_i)-k_1\ \ge\ b_1(X_i)+m-\dim Y\ \ge\ b_1(X_i)+m-N .
\end{equation*}
The first inequality holds because the quotient of $Y$ by $\R^{k_1}\times C'$ carries a free, properly discontinuous action of $\Z^{k_2}$ with quotient $X$; hence $\pi_1(X)$ surjects onto $\Z^{k_2}$. The second is the rank inequality ${k_1+k_2\ge b_1(X_i)}$ \cite[Lemma~6.1]{HHWZ26}, applied after a fixed rescaling to the action of $H_i$ on $\hat X_i$, which is generated by elements moving $\hat p_i$ by at most $3D$, as $\diam X_i\le D$.

It remains to justify the third inequality, $\dim Y\ge m+k_1$. Choose $x\in\mathcal R(X)$ and a lift $y\in Y$, and let $(s_j^{-1}Y,y,H)\to(Y_\infty,o,\mathcal A)$ be a blow-up. The first three steps in the proof of \cref{lem:blowupreg} below give a metric-measure splitting
\begin{equation*}
Y_\infty=\R^m\times Z,
\end{equation*}
where $Z$ is an $\RCD(0,N-m)$ space, $\mathcal A$ acts trivially on the $\R^m$-factor and simply transitively on $Z$, and the orbit map $\mathcal A\to Z$ is a homeomorphism. As $Z$ is connected,
\begin{equation*}
\mathcal A\cong Z\cong\R^p\times\T^q
\end{equation*}
for some $p,q\ge0$.

Put $n_Z:=\dim Z$. Choose an $n_Z$-regular point of $Z$. Transitivity makes every point $n_Z$-regular and makes the Euclidean tangent convergence uniform in the centre. The metric Reifenberg theorem \cite[Theorem~A.1.1]{CC97} therefore makes $Z$ locally a topological $n_Z$-manifold. Since $Z$ is also homeomorphic to $\R^p\times\T^q$, invariance of dimension gives
\begin{equation}\label{eq:homogeneous-dimension}
n_Z=p+q.
\end{equation}

We next compare $p$ with $k_1$. The closed factor $\R^{k_1}\le H$ acts freely and properly on $Y$. Fix a basis $e_1,\dots,e_{k_1}$ and choose $\delta_j\downarrow0$ so that
\begin{equation*}
d_Y((\delta_je_\ell)y,y)\le s_j
\qquad(1\le\ell\le k_1).
\end{equation*}
The lattice $\delta_j\Z^{k_1}$ acts freely, properly discontinuously, and measure-preservingly on $(Y,s_j^{-1}d,y)$ and is generated by elements moving $y$ by at most one. For large $j$ these rescaled spaces are $\RCD(-1,N)$. After passing to a further subsequence, let the equivariant limit of the lattices $\delta_j\Z^{k_1}$ be $L\le\mathcal A$. The group $L$ is closed in the abelian Lie group $\mathcal A$, hence compactly generated, and the rank inequality \cite[Lemma~6.1]{HHWZ26} gives
\begin{equation*}
\rank L\ge k_1.
\end{equation*}
On the other hand, the projection $\R^p\times\T^q\to\R^p$ is proper. Its restriction to $L$ has compact kernel and closed image, so
\begin{equation*}
\rank L=\rank\operatorname{pr}_{\R^p}(L)\le p.
\end{equation*}
Together with \eqref{eq:homogeneous-dimension}, this gives $k_1\le p\le n_Z$. The metric-measure splitting and (P1) now yield
\begin{equation*}
\dim Y\ge\dim Y_\infty=m+n_Z\ge m+k_1,
\end{equation*}
as required.

Under $(\ast)$, the two ends of the chain agree. Hence equality holds at every step: $\dim Y=N$, $k_1=k$, and $k_2=b$.

By (P1), $\mathfrak m_Y=c_Y\mathcal H^N$ for a constant $c_Y>0$, so $Y$ is non-collapsed. Since $\dim\hat X_i=\dim X_i$, which is at least $\dim Y=N$ for large $i$, another application of (P1) gives $\mathfrak m_i=c_i\mathcal H^N$.

The covers are uniformly non-collapsed: $\inf_{i,\hat x}\mathcal H^N(B(\hat x,1))>0$ for large $i$. Indeed, by the volume convergence dichotomy of \cite{DPG18}, either $\mathcal H^N(B(\hat p_i,1))\to0$ and $\dim_{\mathcal H}Y\le N-1$, or $\mathcal H^N(B(\hat p_i,1))\to\mathcal H^N(B(\hat p,1))>0$. The first alternative contradicts $\dim Y=N$. For $\hat x\in\hat X_i$ there is $h\in H_i$ with $d(h\hat x,\hat p_i)\le D$, and Bishop--Gromov bounds $\mathcal H^N(B(h\hat x,1))$ below by $c(K,N,D)\,\mathcal H^N(B(\hat p_i,1))$.

\end{proof}

We now prove the second consequence of equality. Over a regular point of $X$ the tangent cone is Euclidean of dimension $m$, which leaves the remaining $k$ directions no room to do anything but translate.

\begin{lemma}[Blow-Ups at Regular Points]\label{lem:blowupreg}
Let $x\in\mathcal R(X)$, and let $y\in Y$ lie over $x$. Then every blow-up $(s_j^{-1}Y,y,H)\to(Y_\infty,o,\mathcal A)$, $s_j\to0$, has $Y_\infty=\R^N$, and $\mathcal A$ is the translation group of an $\R^k$-factor. In particular $y\in\mathcal R(Y)$. Moreover $C_y=\Stab_H(y)$ is trivial.
\end{lemma}

\begin{proof}
As $H$ is a closed abelian subgroup of $\Isom(Y)$ with $X=Y/H$, the quotient of $(s_j^{-1}Y,y)$ by $H$ is $(s_j^{-1}X,x)$, and these quotients converge to $Y_\infty/\mathcal A$ (\cref{ssec:eqconv}, the case of the whole group). The group $\mathcal A$ is abelian: for $a,b\in\mathcal A$ the elements $\tilde\psi_j(a)$ and $\tilde\psi_j(b)$ of $H$ commute, and \ref{ax:E1} and \ref{ax:E3} give $d(abf_j(z),baf_j(z))<4\tilde\epsilon_j$ wherever the points involved lie in the domains; as the images of the $\tilde f_j$ become dense, $ab=ba$. So $Y_\infty/\mathcal A$ is a tangent cone of $X$ at $x$, and hence is $\R^m$; the projection $\pi_\infty\colon Y_\infty\to\R^m$ is a submetry. By (P2), $Y_\infty$ is an $\RCD(0,N)$ space. The first three steps use only these facts. Beginning with the fourth step, we also use \cref{prop:chain} and (P2), which show that $Y_\infty$ is non-collapsed and therefore a metric cone with vertex $o$.

\emph{Lifting Lines.} Let $\gamma$ be a line through $\pi_\infty(o)$. Lift the two rays of $\gamma$ horizontally from $o$ and concatenate. The result $\tilde\gamma$ has the length of $\gamma$ on every segment, and $\pi_\infty$ is $1$-Lipschitz, so $\tilde\gamma$ is a line through $o$.

\emph{Busemann Functions.} With the convention $b^\pm_\gamma:=\lim_{t\to\infty}\bigl(t-d(\cdot,\gamma(\pm t))\bigr)$ of \cite[(2.4)]{Gig26}, since $\pi_\infty$ is $1$-Lipschitz and $\pi_\infty\tilde\gamma=\gamma$, we have $b^\pm_{\tilde\gamma}\le b^\pm_\gamma\circ\pi_\infty$. On $\R^m$, $b^+_\gamma+b^-_\gamma=0$, and on $Y_\infty$, $b^+_{\tilde\gamma}+b^-_{\tilde\gamma}=0$ \cite[Theorem~4.11]{Gig26}. Hence $b^\pm_{\tilde\gamma}=b^\pm_\gamma\circ\pi_\infty$, which is $\mathcal A$-invariant.

\emph{Splitting.} Choose lines $\gamma_1,\dots,\gamma_m$ of $\R^m$ through $\pi_\infty(o)$ in orthonormal directions, and lift them as in the first step. The lifted line $\tilde\gamma_1$ and the splitting theorem \cite{Gig26} give $Y_\infty=\R\times Z_1$, with first coordinate $b^+_{\tilde\gamma_1}$ and $Z_1$ an $\RCD(0,N-1)$ space. For $l\ge2$, the second step gives $b^+_{\tilde\gamma_1}\circ\tilde\gamma_l=b^+_{\gamma_1}\circ\gamma_l$, which vanishes as $\gamma_l\perp\gamma_1$; so $\tilde\gamma_l$ is a line of $Z_1$. Repeating this in $Z_1$, and so on, gives $Y_\infty=\R^m\times Z$, with coordinates $b^+_{\tilde\gamma_l}=b^+_{\gamma_l}\circ\pi_\infty$. These are $\mathcal A$-invariant by the second step, so $\mathcal A$ acts trivially on $\R^m$. Hence $\mathcal A$ acts transitively on $Z$. Since $\mathcal A$ is abelian, a stabilizer in $Z$ fixes every point of $Z$; it also fixes the $\R^m$-factor, and hence is trivial because $\mathcal A\le\Isom(Y_\infty)$ acts effectively. Thus the action on $Z$ is simply transitive, and the orbit map $\mathcal A\to Z$ is a homeomorphism by \cref{lem:proper}(a).

\emph{$Y_\infty=\R^N$.} Euclidean translations on the first factor, together with the transitive $\mathcal A$-action on $Z$, act transitively on $Y_\infty$. Since the non-collapsed $\RCD(0,N)$ space $Y_\infty$ has $N$-regular points almost everywhere and regularity is preserved by isometries, every point of $Y_\infty$ is $N$-regular. In particular the cone vertex $o$ is regular. But the tangent cone of $Y_\infty$ at $o$ is $Y_\infty$ itself, so $Y_\infty=\R^N$ and $Z=\R^k$.

\emph{$\mathcal A$.} By the third step, $\mathcal A$ acts simply transitively on $Z=\R^k$. If $g(z)=A_gz+b_g$, then $d(gz,z)$ is $\mathcal A$-invariant, by commutativity, and therefore constant in $z$. Thus \mbox{$|(A_g-I)z+b_g|$} is constant on $\R^k$, which forces $A_g=I$. Hence $\mathcal A$ is the full translation group.

Finally, suppose that $C_y\ne1$. Under the blow-up, the constant sequence of subgroups $C_y$ converges to a subgroup of $\mathcal A$ fixing $o$, and that subgroup is trivial, since $\mathcal A$ acts by translations. Thus $C_y$ becomes small in the sense of (P4), which forces $C_y$ itself to be trivial---a contradiction. In detail, put $Y_j:=(Y,s_j^{-1}d,y)$, with normalized measures; for $s_j\le1$ these are $\RCD(\min\{K,0\},N)$ spaces by (P2), converging to $Y_\infty=\R^N$. We claim that $\sup\{s_j^{-1}d(gz,z):g\in C_y,\ d(y,z)\le Rs_j\}\to0$ for every $R$; then (P4), applied to the constant subgroups $C_y\le\Isom(Y_j)$, makes $C_y$ trivial. Otherwise there are $R,\eta>0$ and, along a subsequence, $g_j\in C_y$ and $z_j$ with $d(y,z_j)\le Rs_j$ and $d(g_jz_j,z_j)\ge\eta s_j$. By \cref{lem:quotients}(a), along a further subsequence the subgroups $C_y$ converge to a closed subgroup $C_\infty\le\mathcal A$, and by \ref{ax:E1} and \ref{ax:E2} the elements $\tilde\phi_j(g_j)$ subconverge to some $g\in C_\infty$ and $\tilde f_j(z_j)$ to some $w$, with $go=o$ and $d(gw,w)\ge\eta$. But the only element of $\mathcal A$ fixing $o$ is the identity, a contradiction.
\end{proof}

An orbit over a regular point of $X$ is also no larger than it should be. This bound is what excludes a torus factor in $H_0$.

\begin{lemma}[Orbit Dimension]\label{lem:orbitdim}
If $y$ lies over a regular point, then the orbit $H_0y$ has Hausdorff, hence topological, dimension $\le k$.
\end{lemma}

\begin{proof}
By \cref{lem:blowupreg} and a contradiction argument, for every $\epsilon$ there is $s_0$ such that, for all $r<s_0$, the set \mbox{$Hy\cap\bar B(y,r)$}, rescaled by $r^{-1}$, is $\epsilon$-GH close to $\bar B^k_1$. This gives a cover of $H_0y\cap B(y,r)$ by at most $C(k)\epsilon^{-k}$ balls of radius $6\epsilon r$ centred on $H_0y$: take a maximal $6\epsilon r$-separated subset. By homogeneity of the orbit the same holds at every orbit point, with the same $s_0$. Iterating gives upper box dimension $\le\log(C\epsilon^{-k})/\log(6\epsilon)^{-1}\to k$.
\end{proof}

\begin{proof}[Proof of \cref{thm:master}(i)]
\cref{prop:chain} gives the first assertion. Write $H_0=\R^k\times\T^{q'}$, where $\T^{q'}$ is the identity component of $C'$. At a lift $y$ of a regular point, \cref{lem:blowupreg} gives ${\Stab_{H_0}(y)\subset C_y=1}$. Hence $H_0y\cong\R^k\times\T^{q'}$ is a $(k+q')$-manifold, while \cref{lem:orbitdim} bounds its dimension by $k$. Thus $q'=0$. It follows that $C:=C'$ is finite and $H_0=\R^k$. The action of $H_0$ is free because a stabilizer is compact, whereas $\R^k$ has no non-trivial compact subgroup.

The quotient group $\Gamma=H/H_0\cong\Z^b\times C$ is discrete and acts properly on $W=Y/H_0$, with $X=W/\Gamma$ by \cref{lem:proper}(c). Every stabilizer is finite and hence lies in $\Tor(\Gamma)=C$; moreover, for $w:=P_W(y)$, $\Gamma_w$ is the isomorphic image of $C_y$. The asserted local quotient description is standard for properly discontinuous isometric actions, and \cref{lem:blowupreg} gives triviality over regular points.
\end{proof}

\section{Pointwise Linearization and Translational Limit Symmetry}\label{sec:norot}

We next show that the free action of $H_0\cong\R^k$ is asymptotically translational near every regular orbit, that is, every orbit $H_0y$ with $y\in\mathcal R(Y)$. Such orbits may lie over non-regular points of $X$, which \cref{lem:blowupreg} does not cover. Direct differentiation along the orbits is unavailable in the $\RCD$ setting. Instead, we prove that the harmonic differences $u\circ\varphi_v-u$ grow at most linearly in $|v|$ (\cref{prop:lip}), linearize their values at one orbit point, and then use the transformation theorem to obtain linear models at every sufficiently small scale. The first-hitting argument in \cref{lem:linearmodel} ensures that the whole model $k$-plane, not merely a subgroup of it, occurs in each blow-up.

We do not assume (R). The argument is local near a single orbit $H_0y$, with $y\in\mathcal R(Y)$, where $Y$ is almost Euclidean at every small scale (\cref{lem:tube}). All suprema below are taken over compact sets. Under (R), every point is regular, and the conclusions therefore hold everywhere. This section concerns the limit $Y$ alone, and none of its parameters measures the convergence of the covers; in \cref{sec:groups,sec:local,sec:global}, every scale fixed here is chosen before $i$ is taken large.

Fix a linear isomorphism $H_0\cong\R^k$ and use the Euclidean norm on $\R^k$; the constants below depend on this choice. Write $\varphi_v$, $v\in\R^k=H_0$, for the free action. The displacement function ${\mathfrak d_z(v):=d(\varphi_vz,z)}$ satisfies $\mathfrak d_{hz}=\mathfrak d_z$ for $h\in H$, since $H$ is abelian, and
\begin{equation*}
\mathfrak d_z(-v)=\mathfrak d_z(v),\qquad \mathfrak d_z(v+v')\le \mathfrak d_z(v)+\mathfrak d_z(v').
\end{equation*}

\begin{lemma}[Tubes around Regular Orbits]\label{lem:tube}
Let $y\in\mathcal R(Y)$. There is $\eta_0=\eta_0(N)\in(0,1/2]$ such that, for every $\eta\in(0,\eta_0]$, one can choose $r_0=r_0(y,\eta)\in(0,1]$ satisfying $|K|r_0^2\le\eta$ and the following properties. Every ball $B(z,s)$ with $d(z,H_0y)\le\eta r_0$ and $s\le r_0/2$ is $\Psi(\eta|N)s$-close to $B^N(s)$ in the measured sense, and its volume ratio is at least $1-\Psi(\eta|N)$. Moreover, the $H_0$-invariant tube $\{z:d(z,H_0y)<\eta r_0\}$ is a topological $N$-manifold.
\end{lemma}

\begin{proof}
The constant $\eta_0$ is a Reifenberg threshold, fixed at the end of the proof. Since every tangent cone at $y$ is $\R^N$, (P2) gives $\vartheta_s(y)\to1$ as $s\to0$, for the Bishop--Gromov ratio $\vartheta_s$ of \cref{ssec:packages}. Both conditions hold for every sufficiently small radius, so we may choose $r_0$ with $|K|r_0^2\le\eta$ and $\vartheta_{r_0}(y)\ge1-\eta$. If $d(z,y)\le\eta r_0$, then $B(z,r_0)\supseteq B(y,r_0(1-\eta))$, and Bishop--Gromov gives
\begin{equation*}
\vartheta_{r_0}(z)\ \ge\ \frac{\mathcal H^N(B(y,r_0(1-\eta)))}{V_{K,N}(r_0)}\ \ge\ (1-\Psi(\eta))\,\vartheta_{r_0}(y)\ \ge\ 1-\Psi(\eta),
\end{equation*}
and monotonicity extends this estimate to every $t\le r_0$. The volume--Reifenberg package of \cref{ssec:packages} then gives the asserted measured Euclidean control for $s\le r_0/2$. Because the estimates are $H_0$-invariant, they hold throughout the orbit tube. Choosing $\eta_0(N)$ below the thresholds of the metric and canonical Reifenberg theorems makes every point in the interior of that tube have a Euclidean neighbourhood.
\end{proof}

For the rest of this section we fix $y\in\mathcal R(Y)$. The first step bounds $|v|$ by the displacement $\mathfrak d_z(v)$, uniformly over the $H_0$-saturation of a compact set.

\begin{lemma}[Linear Lower Bound]\label{lem:linlower}
Let $Q\subset Y$ be compact and $t_0>0$. There are $c_Q,c_Q'>0$, depending on $Q$, $t_0$ and the action, such that, for all $z\in H_0Q$, we have $\mathfrak d_z(v)\ge c_Q|v|$ for $|v|\le t_0$ and $\mathfrak d_z(v)\ge c_Q'$ for $|v|\ge t_0$.
In particular, $\mathfrak d_z(v)<c_Q'$ implies $|v|\le\mathfrak d_z(v)/c_Q$.
\end{lemma}

\begin{proof}
By $H_0$-invariance, it suffices to take $z\in Q$. Properness gives a uniform positive lower bound for $\mathfrak d_z(v)$ when $|v|$ is large, while freeness and compactness give one on each fixed annulus. This proves the second assertion and, in particular, gives $c_0>0$ such that $\mathfrak d_z(v)\ge c_0$ whenever $t_0\le|v|\le2t_0$. If $0<|v|\le t_0$, choose $n=\lceil t_0/|v|\rceil$. Then $n|v|\in[t_0,2t_0]$, and subadditivity yields $c_0\le \mathfrak d_z(nv)\le n\mathfrak d_z(v)$. Thus $\mathfrak d_z(v)\ge(c_0/2t_0)|v|$.
\end{proof}

\emph{A Harmonic Chart at $y$.} We build one harmonic chart $u$ at a fixed scale $S$; \cref{lem:linearmodel} carries it to all smaller scales, where the conclusions of this section live. The construction depends on a number $\delta\in(0,\delta(N)]$, fixed throughout, which \cref{lem:linearmodel} chooses in terms of its accuracy; once $\delta$ is fixed, $\eta_\delta$, $r_\delta$ and $S$ below are fixed positive numbers, depending on $y$, and so are the constants $c_Q$, $c_Q'$, $t_1$, $t_2$, $\mathsf L$ and the map $A$ chosen after them; the constants $C$ of \cref{prop:lip,lem:pointlinear} may depend on all of these data. In \cref{lem:tube}, choose $\eta_\delta\in(0,\eta_0]$ with $\Psi(\eta_\delta|N)\le\delta$; taking $\Psi(\eta|N)\ge\eta$, as we may, gives $\eta_\delta\le\delta$. Put $r_\delta:=r_0(y,\eta_\delta)$ and choose $S>0$ with $512S\le\eta_\delta r_\delta$. Every ball of radius at most $256S$ centred in $B(y,256S)$ then lies in the tube of \cref{lem:tube}, so it is measured-$\delta$-close to a Euclidean ball and lies in a topological $N$-manifold; and $|K|(256S)^2\le\eta_\delta^3\le\delta$.
For $\delta$ small, \cref{lem:eucsplit}(b), at $(y,32S)$ with $n=N$, gives a harmonic map $u$ on $B(y,128S)$ with $|\nabla u_a|\le C(N)$ almost everywhere, Gram average at most $\Psi(\delta)$ there, and Lipschitz constant at most $1+\Psi(\delta)$ on $B(y,32S)$; by doubling, $u$ is a harmonic $\Psi(\delta)$-splitting map on $B(y,16S)$ and on $B(y,17S)$. The free proper action of $H_0\cong\R^k$ makes $Y$ non-compact, so the non-exhaustion hypothesis of \cref{lem:harmonicpackage} holds at every finite scale. The bi-H\"older clause of \cref{lem:replacement}, with $\xi=F=u$ at $(y,16S)$ and $L=4$, applies: $Y$ is non-collapsed (\cref{prop:chain}), $|K|(16S)^2\le\delta$, and every ball $B(a,s)$ with $a\in B(y,16S)$ and $s\le16S$ is $\delta s$-close to $B^N(s)$. Since $d(z,z')<16S$ on $B(y,8S)$, replacing $\varkappa$ by $4\varkappa$ turns the factor $(16S)^{-\varkappa}$ into $(8S)^{-\varkappa}$, and so, on $B(y,8S)$,
\begin{equation*}
(1-\varkappa)(8S)^{-\varkappa}d(z,z')^{1+\varkappa}\le|u(z)-u(z')|\le(1+\varkappa)d(z,z'),
\end{equation*}
with $\varkappa=\varkappa(\delta)\to0$ as $\delta\to0$. We always take $\delta$ small enough that $\varkappa\le0.07$; this makes the factors $\lambda,\lambda'$ in the proof of \cref{prop:lip} smaller than $1$ (for $\lambda$ this needs $\varkappa<0.078$). Put $Q:=\bar B(y,16S)$, and let $c_Q,c_Q'$ be as in \cref{lem:linlower} for $Q$ and some fixed $t_0>0$. For $0<R\le16S$ and $t>0$ put
\begin{equation*}
F_R(t):=\sup_{|v|\le t,\ z\in B(y,R)}\mathfrak d_z(v).
\end{equation*}
Since the action is continuous and $Q$ is compact, $F_{16S}(t)\to0$ as $t\to0$. Fix $t_1>0$ with $F_{16S}(2t_1)\le S/8$. For $0<R\le16S$ and $0<t\le2t_1$ put
\begin{equation*}
M_R(t):=\sup_{|v|\le t,\ z\in B(y,R)}|u(\varphi_vz)-u(z)| ;
\end{equation*}
here $\varphi_vz\in B(y,17S)$, the domain of $u$. Both $F_R$ and $M_R$ are non-decreasing in $R$ and in $t$. For $t\le t_1$ and $R\le7S$, the points $z$ and $\varphi_vz$ in these suprema lie in $B(y,8S)$, and the bi-H\"older bounds give
\begin{equation}\label{eq:FM}
M_R(t)\le(1+\varkappa)F_R(t),\qquad F_R(t)\le CS^{\varkappa/(1+\varkappa)}M_R(t)^{1/(1+\varkappa)} .
\end{equation}
In particular $M_R(t)\to0$ as $t\to0$. For $|v|\le2t_1$ the function $h_v:=u\circ\varphi_v-u$ is harmonic on $B(y,15S)$, since $\varphi_v$ is a measure-preserving isometry mapping $B(y,15S)$ into $B(y,17S)$. For $\varrho\le S$ and $B(z,2\varrho)\subset B(y,15S)$, Jiang's gradient estimate \cite[Theorem~1.1]{Jia14}, together with $|K|\varrho^2\le\delta$ and $\fint_{B(z,2\varrho)}|h_v|\le\sup_{B(z,2\varrho)}|h_v|$, gives
\begin{equation}\label{eq:CY}
|\nabla h_v|\le C(N)\sup_{B(z,2\varrho)}|h_v|/\varrho\quad\text{a.e.\ on }B(z,\varrho).
\end{equation}
The harmonicity of $u$ is essential here. If we knew only that $|\Delta u|\le\delta/S$, the gradient bound for $h_v$ would contain an additive error independent of $|v|$, and the iteration below would fail.

\begin{proposition}[Lipschitz Modulus in the Chart]\label{prop:lip}
There are $t_2\in(0,t_1]$ and $\mathsf L$, depending on $\delta$, $S$, $u$ and the action, with $M_{2S}(t)\le\mathsf Lt$ for $t\le t_2$.
\end{proposition}

\begin{proof}
\emph{Doubling with a Margin.} Let $t\le t_1$, $R\le5S$ and $\varrho\le S$ with $F_{R+2\varrho}(t)\le\varrho$. For $z\in B(y,R)$ and $|v|\le t$ use the cocycle identity
\begin{equation*}
u(\varphi_{2v}z)-u(z)=h_v(\varphi_vz)+h_v(z) .
\end{equation*}
A geodesic from $z$ to $\varphi_vz$ has length $\mathfrak d_z(v)\le\varrho$, so it stays in $\bar B(z,\varrho)$. By the Sobolev-to-Lipschitz property of $\RCD$ spaces, the almost-everywhere bound from \eqref{eq:CY} gives the same Lipschitz bound for the continuous representative of $h_v$ on $B(z,\varrho)$. If $\mathfrak d_z(v)=\varrho$, apply it first to interior points of the geodesic and pass to the endpoint by continuity. Together with \eqref{eq:FM}, this yields $|h_v(\varphi_vz)-h_v(z)|\le CF_{R+2\varrho}(t)F_R(t)/\varrho$, and taking suprema,
\begin{equation}\label{eq:double}
M_R(2t)\ \ge\ 2M_R(t)-CF_{R+2\varrho}(t)\,F_R(t)/\varrho .
\end{equation}
The margin $2\varrho$ is needed because the gradient estimate uses $h_v$ on a neighbourhood of the base points. We let it shrink with the scale, so that the base balls shrink by a bounded amount in total.

\emph{Iteration.} Put $t_{2,j}:=2^{-j}t_2$ and $\varrho_j:=(S/8)2^{-j/2}$. Let $R_0:=4S$ and $R_{j+1}:=R_j-2\varrho_j$. Then $R_j$ decreases to $R_\infty\ge4S-(S/4)\sum_j2^{-j/2}\ge3S$. Put $m_j:=M_{R_j}(t_{2,j})$, so that $m_{j+1}\le m_j$ by monotonicity. Apply \eqref{eq:double} with $R=R_{j+1}$, $t=t_{2,j+1}$ and $\varrho=\varrho_j$, so that $R+2\varrho=R_j$. Using $M_{R_{j+1}}(t_{2,j})\le m_j$ and \eqref{eq:FM},
\begin{equation*}
2m_{j+1}\ \le\ m_j+CS^{2\varkappa/(1+\varkappa)}\varrho_j^{-1}m_j^{2/(1+\varkappa)}\ =\ m_j(1+e_j),\qquad e_j:=C\,2^{j/2}(m_j/S)^\beta,
\end{equation*}
with $\beta:=\frac{1-\varkappa}{1+\varkappa}$, provided $F_{R_j}(t_{2,j})\le\varrho_j$. Since $\varkappa\le0.07$,
\begin{equation*}
\lambda:=2^{1/2}(2/3)^{\beta}<1\qquad\text{and}\qquad\lambda':=2^{1/2}(2/3)^{1/(1+\varkappa)}<1.
\end{equation*}
Choose $t_2\le t_1$ so small that
\begin{equation*}
C(m_0/S)^\beta\le\tfrac14\qquad\text{and}\qquad CS^{\varkappa/(1+\varkappa)}m_0^{1/(1+\varkappa)}\le S/8;
\end{equation*}
this is possible because $m_0=M_{4S}(t_2)\to0$ as $t_2\to0$. We show by induction that ${m_j\le(2/3)^jm_0}$. In the induction step, \eqref{eq:FM} gives
\begin{equation*}
F_{R_j}(t_{2,j})\le CS^{\varkappa/(1+\varkappa)}\bigl((2/3)^jm_0\bigr)^{1/(1+\varkappa)}\le(S/8)\lambda'^j2^{-j/2}\le\varrho_j,
\end{equation*}
so the doubling step applies; moreover $e_j\le C(m_0/S)^\beta\lambda^j\le\frac14$, so $m_{j+1}\le\frac58m_j\le\frac23m_j$.
Hence $e_j\le\frac14\lambda^j$ and
\begin{equation*}
 \frac{m_{j+1}}{t_{2,j+1}}=\frac{2m_{j+1}}{t_{2,j}}\le\frac{m_j}{t_{2,j}}(1+e_j),
 \qquad\text{so}\qquad
 \frac{m_j}{t_{2,j}}\le\frac{m_0}{t_2}\exp\Bigl(\sum_le_l\Bigr)=:\frac{\mathsf L}2 .
\end{equation*}
For $t\in(t_{2,j+1},t_{2,j}]$, $M_{2S}(t)\le m_j\le\mathsf Lt_{2,j}/2\le\mathsf Lt$, since $R_j\ge2S$.
\end{proof}

The value of the harmonic cocycle at one orbit point has an actual linear
part.  This is stronger than merely saying that its rescalings are almost
additive.

\begin{lemma}[Pointwise Linearization]\label{lem:pointlinear}
After decreasing $t_2$, put
\begin{equation*}
 a(v):=h_v(y)=u(\varphi_vy)-u(y),
 \qquad \alpha:=\frac1{1+\varkappa}.
\end{equation*}
There exist a linear map $A\colon\R^k\to\R^N$ and a constant
$C<\infty$ such that
\begin{equation}\label{eq:pointlinear}
 |a(v)-Av|\le C|v|^{1+\alpha}\qquad(|v|\le t_2).
\end{equation}
Moreover, for $|v|\le t_2$ and $z\in B(y,S)$,
\begin{equation}\label{eq:cocycleosc}
 |h_v(z)-h_v(y)|\le C|v|d(z,y).
\end{equation}
\end{lemma}

\begin{proof}
\Cref{prop:lip} and \eqref{eq:CY}, applied at the centre $y$
and scale $S$, give $|\nabla h_v|\le C|v|$ on $B(y,S)$.  The
Sobolev-to-Lipschitz property gives \eqref{eq:cocycleosc} for the continuous
representative.  On the other hand, canonical Reifenberg,
\cref{prop:lip}, and $a(w)=h_w(y)$ give
\begin{equation*}
 \mathfrak d_y(w)\le C|a(w)|^\alpha\le C|w|^\alpha.
\end{equation*}
After decreasing $t_2$, the geodesic from $y$ to $\varphi_wy$ lies in
$B(y,S)$.  The cocycle identity therefore yields
\begin{equation}\label{eq:cocycledefect}
 \begin{split}
 |a(v+w)-a(v)-a(w)|
 &=|h_v(\varphi_wy)-h_v(y)|\\
 &\le C|v|\mathfrak d_y(w)\le C|v|\,|w|^\alpha
 \end{split}
\end{equation}
whenever $|v|,|w|\le t_2/2$.

For such $v$, set $b_n(v):=2^na(2^{-n}v)$.  Applying
\eqref{eq:cocycledefect} with both arguments equal to $2^{-n-1}v$ gives
\begin{equation*}
 |b_{n+1}(v)-b_n(v)|\le C|v|^{1+\alpha}2^{-n\alpha}.
\end{equation*}
Thus $A(v):=\lim_nb_n(v)$ exists and satisfies
\eqref{eq:pointlinear}.  Applying \eqref{eq:cocycledefect} to
$2^{-n}v$ and $2^{-n}w$, and then multiplying by $2^n$, shows that
$A(v+w)=A(v)+A(w)$.  Since $A$ is continuous by
\eqref{eq:pointlinear}, this local additive map is the restriction of a
unique linear map $A\colon\R^k\to\R^N$.

\end{proof}

Comparing charts at the same scale repeatedly uses the following compactness
principle.

\begin{lemma}[Euclidean Alignment]\label{lem:nearisom}
For every $n$ and $\nu>0$ there is $\iota=\iota(n,\nu)>0$ with the
following properties.
\begin{enumerate}[label=\textup{(\alph*)}]
\item If $R>0$ and $f\colon B^n(0,R)\to\R^n$ satisfies
\begin{equation*}
 \bigl||f(p)-f(q)|-|p-q|\bigr|\le\iota R
 \quad(p,q\in B^n(0,R)),
\end{equation*}
then there is a Euclidean isometry $E$ with
$|f-E|\le\nu R$ on $B^n(0,R)$.
\item Let $(Z,z)$ be a pointed metric space, and suppose
$\mathcal E\colon B_Z(z,2R)\to B^n(0,2R)$ is an $\iota R$-GH approximation
with $\mathcal E(z)=0$. If $u\colon B_Z(z,2R)\to\R^n$ satisfies
\begin{equation*}
 |u(x)-u(x')|\le d(x,x')+\iota R
\end{equation*}
and $u(B_Z(z,2R))$ is $\iota R$-dense in $B^n(0,2R)$, then some
Euclidean isometry $E$ satisfies
\begin{equation*}
 \sup_{B_Z(z,R)}|u-E\circ\mathcal E|\le\nu R.
\end{equation*}
In particular, $u$ is an $O(\nu R)$-GH approximation on $B_Z(z,R)$.
\end{enumerate}
The radii $2R$ and $R$ may be replaced by any fixed pair of nested
radii, after changing $\iota$ by a constant depending only on that pair.
\end{lemma}
\begin{proof}
After scaling, a contrary sequence in (a) converges on a countable dense set
to a distance-preserving map, hence to a Euclidean isometry; a finite net
upgrades the convergence to uniform convergence. For (b), take
quasi-inverses to $\mathcal E$. A contrary sequence converges uniformly on compact
subballs to a non-expanding surjection of a Euclidean ball onto itself. Such
a map is an isometry, and part~(a), followed by a finite-net argument on
$Z$, gives the stated estimate.
\end{proof}

The linearization becomes useful only after it is normalized at the metric
scale under consideration.  Part~(c) below prevents the image plane
from containing formal directions that are not realized by
bounded-displacement elements.

\begin{lemma}[Linear Models at One Scale]\label{lem:linearmodel}
For every $R\ge1$ and $\varsigma\in(0,1/10)$ there is
$r_{R,\varsigma}=r_{R,\varsigma}(y)>0$ such that, whenever
$0<r\le r_{R,\varsigma}$, there are a map
\begin{equation*}
 U_r\colon B(y,10Rr)\longrightarrow\R^N,\qquad U_r(y)=0,
\end{equation*}
a $k$-plane $V_r\le\R^N$, and a linear isomorphism
$L_r\colon\R^k\to V_r$ with the following properties:
\begin{enumerate}[label=\textup{(\alph*)}]
\item $U_r$ is a $\varsigma r$-GH approximation from $B(y,10Rr)$ to $B^N(0,10Rr)$;
\item if $\mathfrak d_y(v)\le Rr$, then
\begin{equation}\label{eq:scaletranslation}
 |U_r(\varphi_vz)-U_r(z)-L_rv|\le\varsigma r
 \qquad(z\in B(y,4Rr)),
\end{equation}
and $|\mathfrak d_y(v)-|L_rv||\le2\varsigma r$;
\item if $\xi\in V_r$ and $|\xi|\le Rr$, then $v:=L_r^{-1}\xi$ satisfies
\begin{equation*}
 \mathfrak d_y(v)\le r+2|\xi|.
\end{equation*}
\end{enumerate}
The map $L_r$ has a lower singular-value bound independent of $r$.
\end{lemma}

\begin{proof}
Choose $\epsilon'\in(0,1/4)$ so small that every transformation and alignment error below is at most $\varsigma r/10$ on the displayed fixed radii. We use the transformation theorem \cite[Proposition~3.13]{BNS22}, in its case of $(N,\cdot)$-symmetric balls, for $\delta$-almost splitting maps in the sense of \cite[Definition~3.11]{BNS22}, and the growth bound \cite[Corollary~3.16]{BNS22}. Let $C_\nabla=C_\nabla(N)$ be the gradient constant of \cite[Definition~3.11]{BNS22}, which is that of \cite[Definition~3.1]{BNS22}, and $C_T=C_T(N)$ the constant of \cite[Corollary~3.16]{BNS22}. Then $C_\nabla\ge1$: on $\R^N$, \cite[Theorem~3.8(i)]{BNS22}, with $N$ splitting directions and a one-point factor, gives a $\delta'$-splitting map $w$ on the unit ball for every $\delta'\in(0,1/2)$, and $C_\nabla^2>\fint|\nabla w_1|^2>1-\delta'$. Write $|\cdot|$ for operator norms. Choose $\epsilon_B\in(0,\epsilon']$ with $C_T\epsilon_B\le\epsilon'$ and $C(N)\epsilon_B\le\epsilon'$ for the doubling and Lipschitz constants below; then $\delta_B\in(0,\epsilon_B]$ below the threshold $\delta(N,\epsilon_B)$ of \cite[Proposition~3.13]{BNS22} and with $NC(N)\delta_B\le1/2$; then $\delta\le\delta_B^2/2$ with $N^2(C(N)\Psi(\delta)+4\varkappa(\delta))<\delta_B$, and perform the harmonic-chart construction above with this $\delta$. These numbers depend only on $N$, $R$ and $\varsigma$; the chart radius $S$ depends also on $y$, and every scale $r$ below satisfies $r<S/(80R)$.

Put $\varrho:=4S$, and let $\mathcal Y:=(Y,\varrho^{-1}d,\mathcal H^N)$. As $Y$ is non-collapsed (\cref{prop:chain}), $\mathcal Y$ is $\RCD(\varrho^2K,N)$ by (P2), and $|K|\varrho^2\le\delta\le\delta_B^2(N-1)$. For $0<s\le1$ the ball $B^{\mathcal Y}_s(y)=B(y,\varrho s)$ is $\delta\varrho s$-close to $B^N(\varrho s)$, so its GH distance to the unit-scale Euclidean ball of radius $s$ is at most $2\delta s\le\delta_B^2s$; thus every such ball is $(N,\delta_B^2)$-symmetric. On $B^{\mathcal Y}_2(y)=B(y,8S)$ put $\tilde u:=(1-2\varkappa)\varrho^{-1}(u-u(y))$. It is harmonic, and $|\nabla\tilde u_a|\le(1-2\varkappa)(1+\varkappa)<1\le C_\nabla$ almost everywhere, as $u_a$ is $(1+\varkappa)$-Lipschitz on $B(y,8S)$. Doubling bounds the Gram averages of $u$ over $B(y,8S)$ by $C(N)\Psi(\delta)$, and $|(1-2\varkappa)^2\langle\nabla u_a,\nabla u_b\rangle-\delta_{ab}|\le|\langle\nabla u_a,\nabla u_b\rangle-\delta_{ab}|+4\varkappa$; so the Gram errors of $\tilde u$ on $B^{\mathcal Y}_2(y)$ sum to less than $\delta_B$, and $\tilde u$ is a $\delta_B$-almost splitting map there.

Now let $0<r<S/(80R)$ and $s_r:=80Rr/S$, so that $B^{\mathcal Y}_{s_r}(y)=B(y,320Rr)$. By \cite[Proposition~3.13]{BNS22}, with lower scale $s_r$, there are lower-triangular matrices $T_1$ and $T_{s_r}$ satisfying (i) and (ii) there at the scales $1$ and $s_r$. By (ii) at the scale $1$, $T_1\mathrm{Gr}\,T_1^{\mathsf T}=I$ for the Gram matrix $\mathrm{Gr}$ of $\tilde u$ on $B^{\mathcal Y}_1(y)$, whose entries differ from those of $I$ by at most $C(N)\delta_B\le1/(2N)$; so $|T_1|^2=|\mathrm{Gr}^{-1}|\le2$. By \cite[Corollary~3.16]{BNS22}, $|T_1^{-1}T_{s_r}|_\infty\le s_r^{-C_T\epsilon_B}$, and $|M|\le N|M|_\infty$ in either reading of $|\cdot|_\infty$. Put $T:=(1-2\varkappa)T_{s_r}$ and $U_r:=T(u-u(y))=\varrho\,T_{s_r}\tilde u$. Then
\begin{equation*}
 |T|\le|T_1|\,|T_1^{-1}T_{s_r}|\le\sqrt2\,N\Bigl(\frac S{80Rr}\Bigr)^{C_T\epsilon_B}\le C\Bigl(\frac S{Rr}\Bigr)^{\epsilon'} .
\end{equation*}
By (i), $T_{s_r}\tilde u$ is an $\epsilon_B$-almost splitting map on $B^{\mathcal Y}_{s_r}(y)$. So $U_r$ is harmonic, $|\nabla(U_r)_a|<C_\nabla$ almost everywhere on $B(y,320Rr)$, and its Gram errors there sum to less than $\epsilon_B$; by doubling, its Gram averages over $B(y,80Rr)$ are at most $C(N)\epsilon_B\le\epsilon'$. For a unit vector $\omega\in\R^N$, the harmonic function $\omega\cdot U_r$ has $|\nabla(\omega\cdot U_r)|\le\sqrt N\,C_\nabla$ and $\fint_{B(y,320Rr)}\bigl||\nabla(\omega\cdot U_r)|^2-1\bigr|<\epsilon_B$. In the metric $(80Rr)^{-1}d$, which is $\RCD(-\epsilon_B,N)$ for small $r$ and in which $B(y,320Rr)$ is the ball $4B$ of \cite{HP23}, the function $(80Rr)^{-1}\omega\cdot U_r$ is regular (proof of \cref{lem:replacement}), so \cite[Lemma~4.3]{HP23}, with $L:=\sqrt N\,C_\nabla$ and $\lambda=1$ in its notation, makes $\omega\cdot U_r$, hence $U_r$, $(1+C(N)\epsilon_B)$-Lipschitz on $B(y,80Rr)$. Thus $U_r$ is an $(N,\epsilon')$-splitting map on $B(y,80Rr)$ with Lipschitz bound $2$, and $U_r(y)=0$. As $80Rr<S$, the choice of $S$ gives a $4\delta\cdot80Rr$-GH approximation $\mathcal E$ from $B(y,80Rr)$ to $B^N(80Rr)$ with $\mathcal E(y)=0$. Since $4\delta\le\epsilon'$, \cref{lem:eucsplit}(c), at $(y,80Rr)$ with $j=n=N$ and $L=2$, gives a linear isometry $J$ with $|U_r-J\mathcal E|\le C(R)\Psi(\epsilon')r$ on $B(y,40Rr)$, so $U_r$ is a $C(R)\Psi(\epsilon')r$-GH approximation on $B(y,10Rr)$. The choice of $\epsilon'$ proves part~(a).

Let $A$ be supplied by \cref{lem:pointlinear}, and put
$L_r:=TA$ and $V_r:=L_r(\R^k)$.  If $\mathfrak d_y(v)\le Rr$, then
\cref{lem:linlower} gives $|v|\le Rr/c_Q$, once $r_{R,\varsigma}$ is decreased so that $Rr<c_Q'$ and $Rr/c_Q\le t_2$.  For
$z\in B(y,4Rr)$, equations \eqref{eq:pointlinear} and
\eqref{eq:cocycleosc} give
\begin{align*}
 |U_r(\varphi_vz)-U_r(z)-L_rv|
 &\le |T|\bigl(|h_v(z)-h_v(y)|+|a(v)-Av|\bigr)\\
 &\le C_R\bigl(r^{2-\epsilon'}+r^{1+\alpha-\epsilon'}\bigr)=o(r).
\end{align*}
Here $\epsilon'<\alpha$, and $\varphi_vz\in B(y,5Rr)$ because
$d(y,\varphi_vz)\le \mathfrak d_y(v)+d(y,z)$.  Decreasing $r_{R,\varsigma}$ proves
\eqref{eq:scaletranslation}; part~(a), evaluated at $y$ and
$\varphi_vy$, gives the metric estimate in (b).

We next prove that $L_r$ is uniformly non-degenerate.  For a unit vector
$e\in\R^k$, continuity and \cref{lem:linlower} give a first
$t=t(e,r)>0$ with $\mathfrak d_y(te)=r$, and $t\le r/c_Q$.  Applying (a) and
\eqref{eq:scaletranslation} to $te$ gives
\begin{equation*}
 t|L_re|\ge(1-2\varsigma)r\ge r/2.
\end{equation*}
Thus $|L_re|\ge c_Q/2$, uniformly in $e$ and $r$, so $L_r$ is an
isomorphism onto the $k$-plane $V_r$.

Finally take $\xi\in V_r$, $|\xi|\le Rr$, and let
$v=L_r^{-1}\xi$.  If $\mathfrak d_y(v)\le r$, there is nothing to prove.
Otherwise let $\tau\in(0,1)$ be the first time for which
$\mathfrak d_y(\tau v)=r$.  The same two estimates give
$\tau|\xi|\ge(1-2\varsigma)r\ge r/2$.  With
$n:=\lceil1/\tau\rceil$, we have $1/n\le\tau$, and the definition of
the first hitting time gives $\mathfrak d_y(v/n)\le r$.  Subadditivity now yields
\begin{equation*}
 \mathfrak d_y(v)\le n\mathfrak d_y(v/n)\le nr\le r+2|\xi|,
\end{equation*}
which is (c).
\end{proof}

The linear models of \cref{lem:linearmodel} identify the equivariant blow-ups directly. To follow the isotropy group we also use the labelled pointed equivariant GH distance: its approximations are $\epsilon$-GH approximations on radius-$1/\epsilon$ balls that almost intertwine both the $H_0$-action, as in \cite{FY92}, and the action of every labelled element of $\Stab_H(y)$.

\begin{proposition}[Blow-Ups]\label{prop:blowup}
Let $s_j\to0$.
\begin{enumerate}[label=\textup{(\alph*)}]
\item Every limit $(s_j^{-1}Y,y,H_0)\to(\R^N,0,\mathcal A)$ has $\mathcal A$ equal to the translation group of a $k$-plane $V$. In particular every tangent cone of $W$ at $P_W(y)$ is $\R^N/V\cong\R^m$ (\cref{ssec:eqconv}).
\item $C_y$ fixes $H_0y$ pointwise. It acts on the limit by $(\mathrm{id}_V,\rho_y)$, with $\rho_y\colon C_y\to O(V^\perp)$ faithful.
\end{enumerate}
\end{proposition}

\begin{proof}
(a) Choose integers $n(j)\to\infty$ so slowly that
$s_j\le r_{n(j),1/n(j)}$ in \cref{lem:linearmodel}; we may assume
$n(j)\ge20$, so that $1/n(j)<1/10$ as that lemma requires.  Apply it with
$R=n(j)$ and $\varsigma=1/n(j)$, and denote its output by
$(U_j,V_j,L_j)$. The chart behind $U_j$ changes with $j$, and its radius $S_j$ tends to $0$; this is harmless, as $s_j\le r_{n(j),1/n(j)}<S_j/(80n(j))$.  After rescaling by $s_j^{-1}$, the maps $U_j$ are
$1/n(j)$-GH approximations on balls of radius $10n(j)$.  Comparing them
with the space approximations defining the given limit, using
\cref{lem:nearisom}(a), and passing to a subsequence, we may use
$s_j^{-1}U_j$ to identify the limit with $\R^N$ and assume that the
$k$-planes $V_j$ converge to a $k$-plane $V$.

Let $g\in\mathcal A$.  Equivariant convergence supplies $v_j\in H_0$ converging
to $g$, with $\mathfrak d_y(v_j)/s_j$ bounded.  Part~(b) of
\cref{lem:linearmodel} shows on every fixed ball that $\varphi_{v_j}$
is $o(s_j)$-close in the chart $U_j$ to translation by $L_jv_j$.
Moreover, $s_j^{-1}L_jv_j$ is bounded and lies in $V_j$.  After passage to
a subsequence it converges to some $\xi\in V$, and $g$ is translation by
$\xi$.  Hence $\mathcal A$ contains no rotations and is contained in the
translation group of $V$.

Conversely, let $\xi\in V$ and choose $\xi_j\in V_j$ with
$\xi_j\to\xi$.  Part~(c) of \cref{lem:linearmodel}, applied to
$s_j\xi_j$, gives $v_j\in H_0$ such that
\begin{equation*}
 L_jv_j=s_j\xi_j,
 \qquad \mathfrak d_y(v_j)\le s_j(1+2|\xi_j|).
\end{equation*}
Since $n(j)\to\infty$, part~(b) applies for large $j$ and shows that
$\varphi_{v_j}$ converges to translation by $\xi$.  Thus $\mathcal A$ is the full
translation group of $V$.  The tangent cones of $W$ at $P_W(y)$ are the
limits of $(s_j^{-1}Y,y)/H_0$, so they are
$\R^N/V\cong\R^m$ by \cref{ssec:eqconv}, the case of the whole group.

(b) Since $H$ is abelian, every $c\in C_y$ satisfies $c(\varphi_vy)=\varphi_v(cy)=\varphi_vy$. We use labelled equivariant convergence, which tracks each element of $C_y$ along the sequence. The limit is therefore a representation of $C_y$, rather than merely an unspecified subgroup of $O(N)$. This limit action fixes $0$ and is consequently linear. Because it commutes with every translation along $V$, it fixes $V$ pointwise and acts orthogonally on $V^\perp$. As in \cref{lem:blowupreg}, (P4) makes the representation faithful.
\end{proof}

The next lemma isolates the rigidity used in part~(b).

\begin{lemma}[Isotropy Representations]\label{lem:isorep}
Let $\mathcal B_y$ be the set of labelled limits $(\R^N,0,\mathcal A,(\bar\rho(c))_{c\in C_y})$ of $(s^{-1}Y,y,H_0,(c)_{c\in C_y})$ as $s\to0$.
\begin{enumerate}[label=\textup{(\alph*)}]
\item \emph{Local rigidity.} There is $\delta_\rho=\delta_\rho(N,|C_y|)>0$ such that, if two elements of $\mathcal B_y$ are $\delta_\rho$-close, then their representations $\bar\rho$ and $\bar\rho'$ of $C_y$ are conjugate in $O(N)$.
\item \emph{Connectedness.} $\mathcal B_y$ is compact and connected.
\end{enumerate}
Hence the conjugacy classes of $\bar\rho$ and of $\rho_y$ are constant on $\mathcal B_y$.
\end{lemma}

\begin{proof}
(a) A labelled $\delta_\rho$-approximation is, by \cref{lem:nearisom}(a) applied on $B(0,1)$, $\Psi(\delta_\rho|N)$-close there to a linear isometry $E$, and labelled equivariance gives $|E\bar\rho(c)E^{-1}-\bar\rho'(c)|\le\Psi(\delta_\rho|N)$. For a finite group, irreducible multiplicities are integer-valued continuous functions of the character. Thus sufficiently close representations have equal real characters and are linearly equivalent. Since both representations are orthogonal, the polar part of an intertwining operator is an orthogonal intertwiner; hence they are conjugate in $O(N)$.

(b) Equivariant precompactness \cite{FY92} makes the labelled rescalings precompact, while the identity map compares nearby scales continuously. Hence $\mathcal B_y$ is a decreasing intersection of compact connected closures of scale intervals. It is therefore compact and connected.

By (a), the conjugacy class of $\bar\rho$ is locally constant on $\mathcal B_y$. By (b) it is constant. Since $\bar\rho=\mathrm{id}_V\oplus\rho_y$ with $\dim V=k$, the characters satisfy $\chi_{\bar\rho}=k+\chi_{\rho_y}$, so the class of $\rho_y$ is constant as well.
\end{proof}

\emph{The Scale $r_\epsilon(y)$.} The conjugacy class of the representation in \cref{prop:blowup}(b) does not depend on the blow-up, by \cref{lem:isorep}; we fix a representative, a homomorphism $\rho_y\colon C_y\to O(m)$. By \cref{prop:blowup} and a contradiction argument, for every $\epsilon\in(0,\eta_0(N)]$ there are scales $r_\epsilon(y)$ with the closeness property of \cref{def:reps} below; that definition fixes $r_\epsilon(y)$, after the decrease made in the next lemma, for the rest of the paper.

The later lattice construction needs the group component of this approximation to be linear on a fixed multiple of the same scale, not merely along a limiting sequence. We record the fixed-range consequence of \cref{lem:linearmodel}.

\begin{lemma}[Fixed-Scale Linearization]\label{lem:fixedlinear}
There is $\bar\epsilon=\bar\epsilon(N)>0$ such that, for every
$0<\epsilon\le\bar\epsilon$, the following holds after decreasing
$r_\epsilon(y)$. Let $\mathcal F$ be the space component of any labelled
equivariant $\epsilon r$-approximation at a scale $r\le r_\epsilon(y)$,
and let $\phi_r$ be its group component from $H_0$ to the translation group
$\R^k$. There is a linear isomorphism $\bar L_r\colon H_0\to\R^k$ such that
\begin{equation*}
 |\phi_r(v)-\bar L_rv|\le\Psi(\epsilon|N)r
\end{equation*}
whenever $\mathfrak d_y(v)\le50r$.
\end{lemma}

\begin{proof}
Choose $\bar\epsilon(N)\le\min\{\eta_0(N),1/200\}$ so small that, for all
$0<\epsilon\le\bar\epsilon$,
\begin{equation*}
 C\epsilon+\Psi(\epsilon\mid N)\le\frac12,
\end{equation*}
where $C=C(N)$ is the distortion constant below. Since $\epsilon\le1/200$, every approximation
below is defined on all the fixed balls used in the argument.
Apply \cref{lem:linearmodel} with $R=200$ and $\varsigma:=\epsilon$.
After decreasing $r_\epsilon(y)$, it gives $(U_r,V_r^0,L_r^0)$ at every
scale under consideration.  Choose $O_r\in O(N)$ carrying $V_r^0$ to
$\R^k\times\{0\}$, and replace $U_r$ by
$\mathcal F^0:=O_rU_r$ and $L_r^0$ by $O_rL_r^0$, viewed as a linear
isomorphism $H_0\to\R^k$.  Then $\mathcal F^0$ is an
$\epsilon r$-GH approximation on $B(y,2000r)$, and
\begin{equation*}
 |\mathcal F^0(\varphi_vz)-\mathcal F^0(z)-(L_r^0v,0)|
 \le\epsilon r
\end{equation*}
on the fixed balls whenever $\mathfrak d_y(v)\le100r$.

Let $(\mathcal F,\phi_r)$ be any labelled approximation in the statement.
On the fixed ball $B(y,100r)$, compare $\mathcal F$ and $\mathcal F^0$
through a quasi-inverse.  \Cref{lem:nearisom}(a), with fixed nested
radii, gives a Euclidean isometry $E(x)=E_0x+b$ such that
\begin{equation*}
 |\mathcal F-E\mathcal F^0|\le\Psi r
 \quad\text{on }B(y,60r),
 \qquad |b|\le\Psi r.
\end{equation*}
Write $E_{11}$ for the upper-left $k\times k$ block of $E_0$ with respect to $\R^N=\R^k\times\R^m$.  Labelled
equivariance, evaluated at $y$, gives
\begin{equation*}
 |\phi_r(v)-E_{11}L_r^0v|\le\Psi r
 \qquad\text{if }\mathfrak d_y(v)\le50r.
\end{equation*}
Set $\bar L_r:=E_{11}L_r^0$.  Let $v\in H_0$ satisfy $|L_r^0v|=r$.
\Cref{lem:linearmodel}(c) gives
$\mathfrak d_y(v)\le3r$, and part~(b) then gives
$|\mathfrak d_y(v)-r|\le2\epsilon r$.  The distortion and equivariance errors of
$\mathcal F$ give $|\phi_r(v)|\ge r-C\epsilon r$, while the preceding
display gives $|\phi_r(v)-\bar L_rv|\le\Psi r$.
Consequently
\begin{equation*}
 |\bar L_rv|\ge r-C\epsilon r-\Psi(\epsilon\mid N)r\ge\frac r2
\end{equation*}
by the choice of $\bar\epsilon(N)$. Thus $\bar L_r$ is injective and hence an
isomorphism. Since every comparison was made on a fixed radius
ratio, taking
monotone envelopes produces a modulus $\Psi(\epsilon|N)\to0$ without any
rate assumption on \cref{lem:nearisom}.
\end{proof}

The constructions of \cref{sec:eqmaps,sec:local} take place at a scale at which the action near a regular orbit is close to its blow-up model and the fixed-scale linearization applies. The next definition fixes such a scale at every regular point.

\begin{definition}[The Scale $r_\epsilon(y)$]\label{def:reps}
Let $y\in\mathcal R(Y)$ and $\epsilon\in(0,\eta_0(N)]$, so that $r_0(y,\epsilon)$ of \cref{lem:tube} is defined. The scale $r_\epsilon(y)$ is a number in $(0,\epsilon^2r_0(y,\epsilon)/2]$ chosen so that, for every $r\le r_\epsilon(y)$, the rescaled action $(r^{-1}Y,y;H_0,C_y)$ is $\epsilon$-close on radius-$1/\epsilon$ balls, in the labelled equivariant GH sense, to
\begin{equation*}
(\R^k\times\R^m,0;\ \text{translations of }\R^k,\ (\mathrm{id},\rho_y)),
\end{equation*}
for the fixed representative $\rho_y$; and, if $\epsilon\le\bar\epsilon(N)$, so that the conclusion of \cref{lem:fixedlinear} holds at every such scale $r$. By \cref{lem:tube}, every ball of radius at most $r/\epsilon$, centred within $r/\epsilon$ of $H_0y$, is then $\Psi(\epsilon|N)$-close to Euclidean and lies in a topological $N$-manifold.
\end{definition}

Every $r\le r_\epsilon(y)$ satisfies $|K|(320r)^2\le\epsilon$ when $\epsilon\le\bar\epsilon(N)$, since $r_\epsilon(y)\le\epsilon^2r_0(y,\epsilon)/2$ and $|K|r_0(y,\epsilon)^2\le\epsilon$ (\cref{lem:tube}).

\section{The Transverse Coordinate and the Orbifold Structure}\label{sec:eqmaps}

We now construct the transverse coordinate from the Introduction: an equivariant map whose local level sets are precisely the $H_0$-orbits. After passing to the quotient, this map supplies the orbifold charts and proves parts~(ii) and~(iii) of the Main Theorem.

\subsection{Equivariant Harmonic Splitting Maps}
The construction takes place at a scale on which the action is close to the model of \cref{prop:blowup}. We pass to a fine lattice quotient $Y/\Lambda$, replace an aligned splitting map harmonically, and average the replacement over $(H_0/\Lambda)\times C_y$. Averaging preserves harmonicity, so the resulting map is both harmonic and exactly equivariant. After lifting to $Y$, the orbit--slice criterion identifies its local level sets with the $H_0$-orbits. The splitting map to be averaged is built in the same way as in the parallel construction on the approximating spaces in \cref{ssec:eqsplit}, so we record that step once.

\begin{lemma}[Aligned Splitting Maps]\label{lem:align}
Let $1\le m\le N$, let $\epsilon>0$ be small depending on $N$, let $\Psi_0=\Psi_0(\epsilon|N)$ be given, and write $\Psi=\Psi(\epsilon|N,\Psi_0)$. Let $(Z,d,\mathcal H^N)$ be a non-collapsed $\RCD(K,N)$ space with $|K|r^2\le\epsilon$, let a compact group $\mathsf G$ act on $Z$ by measure-preserving isometries, let $\rho\colon\mathsf G\to O(m)$, and let $z\in Z$. Assume that
\begin{equation*}
\mathcal E\colon B(z,r/\epsilon)\longrightarrow\R^m,\qquad\mathcal E(z)=0,
\end{equation*}
is a $\Psi_0r$-GH approximation onto $B^m(r/\epsilon)$ such that $|\mathcal E(gx)-\rho(g)\mathcal E(x)|\le\Psi_0r$ whenever $g\in\mathsf G$ and $x,gx\in B(z,10r)$. Then $Z\setminus B(z,11r)\ne\varnothing$, the ball $B(z,400r)$ is measured-$\Psi$-close to the Euclidean ball $B^m(400r)$, and there is an $(m,\Psi)$-splitting map $u\colon B(z,40r)\to\R^m$ with
\begin{equation*}
|u-\mathcal E|\le\Psi r\ \text{on }B(z,5r),\qquad|u(gx)-\rho(g)u(x)|\le\Psi r\ \text{whenever }g\in\mathsf G\text{ and }x,gx\in B(z,5r).
\end{equation*}
\end{lemma}

\begin{proof}
Rescale so that $r=1$. Since $\mathcal E$ is onto $B^m(1/\epsilon)$ and $\epsilon$ is small, there are points at distance more than $11$ from $z$. Part~(a) of \cref{lem:measured} upgrades the metric closeness on the radius-$1/\epsilon$ ball to measured closeness on the radius-$400$ ball. Part~(b), at $(z,40)$ with $n=m$, gives a harmonic $(m,\Psi)$-splitting map $u_0\colon B(z,40)\to\R^m$ with Lipschitz bound $2$. The restriction of $\mathcal E$ to $B(z,40)$ changes distances by at most $\Psi_0$, so part~(c), at $(z,40)$ with $j=n=m$, gives an orthogonal map $A$ of $\R^m$ with $|u_0-u_0(z)-A\mathcal E|\le\Psi$ on $B(z,20)$. Put $E:=A^{-1}$ and $u:=E\circ(u_0-u_0(z))$, so that $|u-\mathcal E|\le\Psi$ on $B(z,5)$. Translations do not affect derivatives, while the orthogonal part conjugates the Gram matrix, so $u$ is an $(m,C(m)\Psi)$-splitting map. Finally, for $g\in\mathsf G$ and $x$ with $x,gx\in B(z,5)$,
\begin{equation*}
|u(gx)-\rho(g)u(x)|\le|u(gx)-\mathcal E(gx)|+|\mathcal E(gx)-\rho(g)\mathcal E(x)|+|\rho(g)(\mathcal E(x)-u(x))|\le2\Psi+\Psi_0. \qedhere
\end{equation*}
\end{proof}

The next proposition carries out the lattice-quotient averaging described above and records the properties of the resulting transverse coordinate that the later sections use.

\begin{proposition}[Equivariant Harmonic Splitting Maps]\label{prop:U2}
Let $y\in\mathcal R(Y)$ and $w:=P_W(y)$, so that $\Gamma_w\cong C_y$. Let $\epsilon>0$ be small, depending on $N$, and let $r\le r_\epsilon(y)$. Fix an approximation $\mathcal F=(\mathcal F',\mathcal F'')$ realizing the labelled closeness of \cref{def:reps} at the scale $r$, written in the unscaled metric. There is a harmonic map $\hat\psi\colon H_0\cdot B(y,8r)\to\R^m$ with $\hat\psi(y)=0$ such that:
\begin{enumerate}[label=\textup{(\alph*)}]
\item $\hat\psi$ is $H_0$-invariant and $C_y$-equivariant: $\hat\psi\circ c=\rho(c)\circ\hat\psi$ for $c\in C_y$, where $\rho=\rho_y$ under the fixed target identification;
\item $\hat\psi$ is an $(m,\Psi(\epsilon|N))$-splitting map on $B(y,2r)$;
\item if $m\ge1$, there is a harmonic map $\zeta=(\zeta',\zeta'')\colon B(y,80r)\to\R^k\times\R^m$, which is an $(N,\Psi(\epsilon|N))$-splitting map on $B(y,40r)$ and satisfies $|\zeta''-\mathcal F''|\le C(N)\Psi(\epsilon|N)r$ on $B(y,10r)$, with the following property. Let $L\ge1$ and $\delta>0$, and let $\mathfrak b\colon B(y,6r)\to\R^m$ be $H_0$-invariant, Lipschitz, and in the domain of the Laplacian, with $\operatorname{Lip}\mathfrak b\le L$, $|\Delta\mathfrak b|\le\delta/r$ and $|\mathfrak b-\zeta''|\le\delta r$. If $\delta$ and $\epsilon$ are small depending on $N$ and $L$, then, for every $z\in B(y,2r)$,
\begin{equation*}
 \Upsilon_z^{\mathfrak b}:=(\zeta',\mathfrak b)\colon B(z,r)\longrightarrow\R^N
\end{equation*}
is a homeomorphism onto an open image and, for $p,q\in B(z,r)$,
\begin{equation*}
 (1-\varkappa)r^{-\varkappa}d(p,q)^{1+\varkappa}
 \le |\Upsilon_z^{\mathfrak b}(p)-\Upsilon_z^{\mathfrak b}(q)|
 \le(1+\varkappa)d(p,q),
 \qquad \varkappa=\varkappa(\delta,\epsilon|N,L)\longrightarrow0.
\end{equation*}
Moreover, for $x\in B(y,r/8)$, the level set of $\mathfrak b$ through $x$ in $B(y,r/8)$ is $H_0x\cap B(y,r/8)$, so $\mathfrak b$ descends to a continuous, injective, open map on $B_W(w,r/8)$;
\item if $m\ge1$, then $\hat\psi$ satisfies the hypotheses of (c) with $L=C(N)$ and $\delta=C(N)\Psi(\epsilon|N)$; we write $\Upsilon_z:=\Upsilon_z^{\hat\psi}$.
\end{enumerate}
Hence $\hat\psi$ descends to a $\Gamma_w$-equivariant open embedding $h\colon B_W(w,r/8)\to\R^m$ with $h(w)=0$, where $\Gamma_w$ acts on $\R^m$ through $\rho$.
\end{proposition}

\begin{proof}
If $m=0$, then $W$ is a point and there is nothing to prove. Assume $m\ge1$ and write $\Psi=\Psi(\epsilon|N)$. By \cref{def:reps}, every ball of radius at most $r/\epsilon$, centred within $r/\epsilon$ of $H_0y$, is $\Psi$-close to Euclidean and lies in a topological $N$-manifold. We construct the chart first on $Y/\Lambda$ and then on $Y$; on both spaces the $\RCD$ theory applies with measure $\mathcal H^N$. Throughout the proof, we use the approximation $\mathcal F$ of the statement.

\emph{Averaging on $Y/\Lambda$.} By continuity of the action and compactness, we may choose a lattice $\Lambda\subset H_0$ whose fundamental parallelepiped $\mathcal P$ satisfies $\mathfrak d_x(v)\le\epsilon r$ whenever $|v|\le\diam\mathcal P$ and $x\in\bar B(y,2r/\epsilon)$. With $\pi_\Lambda\colon Y\to Y/\Lambda$ the quotient map and $\bar y:=\pi_\Lambda(y)$, every orbit of $\T_\Lambda:=H_0/\Lambda$ through $\pi_\Lambda(B(y,2r/\epsilon))$ then has diameter at most $\epsilon r$: representatives may be chosen in $\mathcal P$, and
$d(\pi_\Lambda\varphi_vx,\pi_\Lambda\varphi_{v'}x)\le \mathfrak d_x(v-v')$. Set $\mathsf G:=\T_\Lambda\times C_y$. By (P3), it acts by measure-preserving isometries on the $\RCD(K,N)$ space $Y/\Lambda$.

\emph{An Aligned Splitting Map.} The labelled approximation makes $(W,r^{-1}d_W,w)$ $\Psi$-close to $\R^m$ on radius $1/(4\epsilon)$, while $Y/\Lambda\to W$ has fibres of diameter $\le\epsilon r$ over $B_W(w,r/\epsilon)$. The transverse part of the labelled approximation therefore descends to a $\Psi r$-GH approximation $\mathcal E$ of $B(\bar y,r/(4\epsilon))$ onto a Euclidean ball, with $|\mathcal E\circ g-\rho(g)\circ\mathcal E|\le\Psi r$ whenever both points lie in $B(\bar y,20r)$, for $\rho(t,c):=\rho_y(c)$. Replace $\mathcal E$ by $\mathcal E-\mathcal E(\bar y)$; the equivariance error changes only by $\Psi r$. \Cref{lem:align}, applied at scale $2r$ with $8\epsilon$ in place of $\epsilon$, gives the required measured closeness and non-exhaustion and, after restriction, a $\Psi$-splitting map $u\colon B(\bar y,40r)\to\R^m$ with $|u-\mathcal E|\le\Psi r$ on $B(\bar y,10r)$ and $|u(gx)-\rho(g)u(x)|\le\Psi r$ whenever $x,gx\in B(\bar y,10r)$. The alignment prevents the average from degenerating.

\emph{Harmonic Replacement and Averaging.} Put $\Omega:=\T_\Lambda\cdot B(\bar y,8r)$. This is $\mathsf G$-invariant and lies in $B(\bar y,(8+\epsilon)r)$. Let $\tilde u$ be the harmonic replacement of $u$ on $\Omega$, and define
\begin{equation*}
 \psi_\Lambda:=\int_{\mathsf G}\rho(g)^{-1}\tilde u\circ g\,dg.
\end{equation*}
\cref{lem:harmonicpackage}, with fixed nested radii as allowed in
\cref{ssec:packages}, gives $|\tilde u-u|\le C\Psi r$ on $\Omega$ and
$|\nabla\tilde u|\le C(N)$ on $B(\bar y,7r)$. Averaging preserves harmonicity,
so $\psi_\Lambda$ is harmonic. \Cref{lem:averaging}, applied to $\tilde u$ with
fixed nested radii (the approximate equivariance of $\tilde u$ follows from that
of $u$ and the preceding pointwise bound), makes $\psi_\Lambda$ exactly
$\mathsf G$-equivariant and gives
$|\psi_\Lambda-u|\le C\Psi r$ on $B(\bar y,6r)$, while
$|\nabla\psi_\Lambda|\le C(N)$ there. These pointwise estimates lift unchanged to
$Y$; the splitting estimate will be recovered below by replacing the
transverse block of a full harmonic chart.

\emph{Splitting on $Y$.} Define $\hat\psi:=\psi_\Lambda\circ\pi_\Lambda-\psi_\Lambda(\bar y)$. Because $C_y$ fixes $\bar y$, we have $\psi_\Lambda(\bar y)\in\mathrm{Fix}\,\rho_y(C_y)$. Consequently, $\hat\psi$ is harmonic on $\pi_\Lambda^{-1}(\Omega)\supset B(y,8r)$, is $H_0$-invariant and $C_y$-equivariant, satisfies $\hat\psi(y)=0$, and obeys $|\nabla\hat\psi|\le C(N)$ on $B(y,6r)$. By the Euclidean control above, the ball $B(y,640r)$ is $\Psi\cdot640r$-close to $B^N(640r)$, and $Y$ is non-collapsed. So \cref{lem:eucsplit}(b), at $(y,80r)$ with $n=N$, gives a harmonic map $\zeta_0\colon B(y,320r)\to\R^N$ with $|\nabla\zeta_0|\le C(N)$ almost everywhere, which is $\Psi$-splitting on $B(y,80r)$ with Lipschitz bound $2$, and hence, by doubling, on $B(y,40r)$. The approximation $\mathcal F-\mathcal F(y)$ changes distances on $B(y,40r)$ by at most $\Psi r$, so \cref{lem:eucsplit}(c), at $(y,40r)$ with $j=n=N$, gives a linear isometry $A$ with $|\zeta_0-\zeta_0(y)-A(\mathcal F-\mathcal F(y))|\le\Psi r$ on $B(y,20r)$. The harmonic map $\zeta=(\zeta',\zeta''):=A^{-1}(\zeta_0-\zeta_0(y))+\mathcal F(y)$, from $B(y,80r)$ to $\R^k\times\R^m$, satisfies $|\nabla\zeta|\le C(N)$ on $B(y,60r)$, is $\Psi$-splitting on $B(y,40r)$, and is $C\Psi r$-close on $B(y,10r)$ to $\mathcal F$; in particular $|\zeta''-\mathcal F''|\le C\Psi r$ there.
The difference $\hat\psi-\zeta''$ is harmonic and $C\Psi r$-small on
$B(y,6r)$: the estimate above compares $\hat\psi$ with
$u\circ\pi_\Lambda$, while both $u\circ\pi_\Lambda$ and $\zeta''$ are
$C\Psi r$-close to $\mathcal F''$.

\emph{The Property in (c).} Let $\mathfrak b$, $L$ and $\delta$ be as in (c), and let $z\in B(y,2r)$. Apply \cref{lem:replacement}, at scale $2r$, to $\zeta$ and $\Upsilon_z^{\mathfrak b}$. Doubling supplies the splitting estimate for $\zeta$ on $B(z,2r)$, and $B(z,4r)\subset B(y,6r)$; the Euclidean control from the choice of $r$ supplies the Reifenberg hypothesis. Thus $\Upsilon_z^{\mathfrak b}$ is an $(N,C(\delta+\Psi))$-splitting map at $(z,2r)$, with $C=C(N,L)$, and has the bi-H\"older estimates in (c), after absorbing the fixed change of radii into $\varkappa$. For $\delta$ and $\epsilon$ small depending on $N$ and $L$, the exponent $\varkappa$ is below the threshold in \cref{lem:orbitslice}. Since $\mathfrak b$ is $H_0$-invariant, that lemma gives
\begin{equation*}
 \mathfrak b^{-1}(\mathfrak b(z))\cap B(z,r/4)=H_0z\cap B(z,r/4).
\end{equation*}
If $x_1,x_2\in B(y,r/8)$ and $\mathfrak b(x_1)=\mathfrak b(x_2)$, then, with $z=x_1$, we have $x_2\in B(z,r/4)$, hence $x_2\in H_0x_1$. Thus $\mathfrak b$ descends to a continuous injective map on $B_W(w,r/8)$. The map is open because $\mathfrak b=\mathrm{pr}_2\circ\Upsilon_z^{\mathfrak b}$ on $B(z,r)$ and the quotient $Y\to W$ is open.

\emph{Conclusion.} The estimates above show that $\hat\psi$ satisfies the hypotheses of (c) with $L=C(N)$ and $\delta=C(N)\Psi$, which is (d). For $\mathfrak b=\hat\psi$ and $z=y$, the last $m$ components of $\Upsilon_y$ give (b). The map $h$ that $\hat\psi$ induces on $B:=B_W(w,r/8)$ satisfies $h(w)=0$, and it is $\Gamma_w$-equivariant because $\Gamma_w$ is the image of $C_y$.
\end{proof}

\subsection{Saturated Sets}\label{ssec:saturated}
Hausdorff-dimension bounds on $Y$ do not generally descend to its quotients. For example, when $k\ge2$, the set $\R^{k-2}\times\R^m$ has codimension two in $\R^k\times\R^m$ but projects onto all of $\R^m$. The needed bounds do descend, however, for sets that are unions of whole orbits.

\begin{lemma}[Saturated Sets]\label{lem:saturated}
Assume $(\ast)$; (R) is not assumed.
\begin{enumerate}[label=\textup{(\alph*)}]
\item Every non-empty open subset of an $H_0$-orbit has positive, possibly infinite, $\mathcal H^k$-measure. So every non-empty $H_0$-invariant set $A\subset Y$ has $\dim_{\mathcal H}A\ge k$.
\item Let $A\subset Y$ be $H_0$-invariant and $s\ge0$. If $\mathcal H^{k+s}(A)=0$, then $\mathcal H^s(P(A))=0$. Hence $\dim_{\mathcal H}P(A)\le\dim_{\mathcal H}A-k$ if $A\ne\emptyset$.
\item $\dim_{\mathcal H}X=m$.
\end{enumerate}
\end{lemma}

\begin{proof}
(a) Fix $y\in Y$. The orbit map $\R^k=H_0\to Y$, $v\mapsto\varphi_vy$, is injective because the action is free (\cref{thm:master}(i)), and it is proper by \cref{lem:proper}(a). It is therefore a homeomorphism onto the closed orbit $H_0y$. Every non-empty open subset of $H_0y$ contains a closed $k$-ball and has covering dimension $k$. Szpilrajn's theorem \cite{Szp37}, \cite[Chapter~VII]{HW41} says that a separable metric space of vanishing $\mathcal H^k$-measure has covering dimension at most $k-1$. Hence the subset has positive $\mathcal H^k$-measure.

(b) The projection $P_W\colon Y\to W$ is $1$-Lipschitz, and $Y$ and $W$ are proper (\cref{lem:proper}(b)). So Eilenberg's inequality \cite[2.10.25]{Fed69} applies (see also \cite{EH21}):
\begin{equation*}
\int^*_W\mathcal H^k\bigl(A\cap P_W^{-1}(w)\bigr)\,d\mathcal H^s(w)\ \le\ c(k,s)\,\mathcal H^{k+s}(A)\ =\ 0 .
\end{equation*}
For every $w\in P_W(A)$, invariance of $A$ implies that $A\cap P_W^{-1}(w)$ is an entire $H_0$-orbit. Part~(a) makes the integrand positive on $P_W(A)$, so $\mathcal H^s(P_W(A))=0$. The quotient map $W\to X=W/\Gamma$ is $1$-Lipschitz, and therefore $\mathcal H^s(P(A))=0$. To obtain the dimension bound, apply this conclusion for every $s>\dim_{\mathcal H}A-k$, which is non-negative by part~(a).

(c) Bishop--Gromov implies that $\mathcal H^N$ is finite on bounded sets, hence $\mathcal H^{N+t}(Y)=0$ for every $t>0$. Applying part~(b) with $A=Y$ gives $\mathcal H^{m+t}(X)=0$ and thus $\dim_{\mathcal H}X\le m$. For the reverse inequality, choose a regular point $x\in X$ and a lift $y\in Y$, and put $w=P_W(y)$. \Cref{lem:blowupreg} gives $y\in\mathcal R(Y)$ and $\Gamma_w=1$. By \cref{thm:master}(i), a neighbourhood of $x$ is therefore isometric to a ball $B_W(w,r)$ for every sufficiently small $r>0$. \Cref{prop:U2} shows that this ball contains an open set homeomorphic to an open subset of $\R^m$. Szpilrajn's theorem then gives $\mathcal H^m(X)>0$, so $\dim_{\mathcal H}X\ge m$.
\end{proof}

The only feature of the orbits used in (a) is their topological dimension. Their metric regularity plays no role: the orbits may be non-rectifiable.

\subsection{The Set \texorpdfstring{$G$}{G}}\label{ssec:setG}
We obtain the open set $G$ in the Main Theorem by taking the charts of \cref{prop:U2} centred at regular orbits. Their radii are fixed once and for all, small enough to accommodate the later constructions of \cref{sec:local}.

Fix constants $\eta=\eta(N)\in(0,\eta_0]$ and $\epsilon=\epsilon(N)\in(0,1/400]$. The constant $\epsilon$ is the accuracy of the labelled model at the scales $r(y)$ chosen below, and $\eta$ is the parameter of \cref{lem:tube} for the Euclidean control near regular orbits, of accuracy $\Psi(\eta|N)$, that \cref{sec:local} transfers to the covers (\cref{lem:rewinding}(b)). With $\Psi=\Psi(\epsilon,\eta\,|\,N)$, they are taken so small that the following hold.
\begin{enumerate}[label=\textup{(S\arabic*)}]
\item\label{S1} \cref{prop:U2} applies with this $\epsilon$, including part~(c) for the map $\mathfrak b$ of \cref{ssec:eqsplit}, which has $L=C(N)$ and $\delta=C(N)\Psi$; and $\epsilon\le\bar\epsilon(N)$ (\cref{lem:fixedlinear}).
\item\label{S2} $1/16+6\Psi<1/9$ (\cref{lem:bundlecrit}).
\item\label{S3} $\max\{\epsilon,\Psi\}\le\delta_0(N,1)$, the constant of \cref{lem:submersion} with Laplacian constant $1$ (\cref{lem:bundlecrit}).
\item\label{S4} $C\Psi$ is below the threshold of the bi-H\"older clause of \cref{lem:replacement} and below the thresholds of \cref{lem:eucsplit}(b) and (c) with $L=2$, for the finitely many constants $C=C(N)$ of \cref{ssec:fibrecoords,ssec:chart}.
\end{enumerate}
The moduli in \ref{S1}--\ref{S4} are functions of $\epsilon$, $\eta$ and $N$ alone, supplied by results proved later; none depends on the radii chosen below. By \ref{S1} and \cref{lem:fixedlinear}, $\epsilon\le\bar\epsilon(N)\le\eta_0(N)$, as \cref{def:reps} requires. For $y\in\mathcal R(Y)$ with $w=P_W(y)$ choose $r(y)>0$ with the following three properties. First, $r(y)\le r_\epsilon(y)$, as in \cref{def:reps}; then $|K|(320r(y))^2\le\epsilon$, as noted after \cref{def:reps}. Second, $r(y)\le\epsilon\,\eta\,r_0(y,\eta)/1000$, with $r_0$ as in \cref{lem:tube}, so that the balls used in \cref{sec:local} lie where \cref{lem:rewinding}(b) applies. Third, $r(y)\le\epsilon\,r_\Gamma(w)/100$, where $r_\Gamma(w)>0$ is \emph{admissible at $w$}: $\gamma B_W(w,3r_\Gamma(w))\cap B_W(w,3r_\Gamma(w))=\emptyset$ for $\gamma\in\Gamma\setminus\Gamma_w$. Admissible radii exist by proper discontinuity (\cref{lem:proper}(c)), smaller radii remain admissible, and admissibility depends only on the image of $w$ in $X$. The factor $\epsilon/100$ makes $W\to W'$ isometric on $B_W(w,r(y)/\epsilon)$ (\cref{ssec:eqsplit}).
Let $h_y\colon B_W(w,r(y)/8)\to\R^m$ be the chart of \cref{prop:U2}. Choose these data for one representative of each $H$-orbit in $\mathcal R(Y)$, and transport them by the $H$-action. Equivariance ensures that two transports of the same chart differ only by a linear map. Define
\begin{equation*}
G_W:=\bigcup_{y\in\mathcal R(Y)}B_W\bigl(P_W(y),r(y)/16\bigr),\qquad G:=G_W/\Gamma\subset X .
\end{equation*}

The next lemma proves parts~(ii) and~(iii) of the Main Theorem for this set.

\begin{lemma}[The Set $G$]\label{lem:setG}
\begin{enumerate}[label=\textup{(\alph*)}]
\item $G_W$ is open and $\Gamma$-invariant. $G$ is open in $X$ and contains $P(\mathcal R(Y))$.
\item \Cref{thm:master}(ii) holds for $G$. Moreover, $\Gamma_{w_1}\le\Gamma_w$ whenever $y\in\mathcal R(Y)$, $w=P_W(y)$ and $w_1\in B_W(w,r(y)/16)$.
\item \Cref{thm:master}(iii) holds for $G$. Moreover, if the $X_i$ have no boundary, then neither do the $\hat X_i$ and $Y$, the singular set satisfies $\dim_{\mathcal H}\mathcal S(Y)\le N-2$, and $\mathcal S(Y)=\emptyset$ if $m\le1$.
\end{enumerate}
\end{lemma}

\begin{proof}
(a) This is clear from the definition.

(b) Let $w_1\in B_W(w,r(y)/16)$, where $w=P_W(y)$ and $y\in\mathcal R(Y)$. Put $h:=h_y$, and let $\Gamma_w$ act through $\rho$. Every translate of $B_W(w,r(y)/8)$ by an element of $\Gamma\setminus\Gamma_w$ is disjoint from that ball. Hence $\Gamma_{w_1}\le\Gamma_w$, and $\rho(\Gamma_{w_1})$ fixes $h(w_1)$. Choose a ball $\mathbb D$ about $h(w_1)$, compactly contained in $h(B_W(w,r(y)/8))$, so small that $h^{-1}(\bar{\mathbb D})\subset B_W(w,r(y)/16)\subset G_W$ and $\rho(\gamma)\mathbb D\cap \mathbb D=\emptyset$ for every $\gamma\in\Gamma_w\setminus\Gamma_{w_1}$. Then $h^{-1}(\mathbb D)$ is a $\Gamma_{w_1}$-invariant neighbourhood of $w_1$ and is disjoint from all its translates by $\Gamma\setminus\Gamma_{w_1}$. Its image in $X$ is a neighbourhood of $[w_1]$ contained in $G$ and homeomorphic to
$h^{-1}(\mathbb D)/\Gamma_{w_1}\cong \mathbb D/\rho(\Gamma_{w_1})$, where $\rho(\Gamma_{w_1})$ acts linearly about $h(w_1)$. Thus $G$ is an orbifold whose local group at the image $x$ of $w_1$ is $\Gamma_{w_1}=C_x\le C$, acting faithfully because $\rho=\rho_y$ under the fixed target identification (\cref{prop:U2}(a)) and $\rho_y$ is faithful (\cref{prop:blowup}(b)).

(c) By \cref{lem:blowupreg}, the regular points of $X$ lie in $P(\mathcal R(Y))\subset G$. They are dense and of full measure \cite{BS20}. By \cref{lem:saturated}(c), $\dim_{\mathcal H}X=m$.

Now assume that the $X_i$ have no boundary. Their covers $\hat X_i$ are locally isometric to them and therefore also have no boundary. Since $(\hat X_i,\mathcal H^N)\to(Y,\mathcal H^N)$ by volume convergence, boundary stability \cite[Theorem~1.6]{BNS22} implies that the non-collapsed limit $Y$ has no boundary either. If one tangent cone at a point is $\R^N$, then the volume density there is $1$; volume convergence and \cite[Corollary~1.7]{DPG18}, applied to the tangent cones, which are non-collapsed $\RCD(0,N)$ cones, then force every tangent cone at that point to be $\R^N$. Thus the singular set $\mathcal S(Y)$ is the stratum $S^{N-1}(Y)$ of \cite{DPG18}. The points of $S^{N-1}(Y)\setminus S^{N-2}(Y)$ are exactly those admitting a half-space tangent cone, and their closure is the boundary. Because the boundary is empty, $\mathcal S(Y)=S^{N-2}(Y)$ and $\dim_{\mathcal H}\mathcal S(Y)\le N-2$ by \cite[Theorem~1.8]{DPG18}. This is the only use of the no-boundary hypothesis; the remaining argument concerns the quotient.
The set $\mathcal S(Y)$ is $H$-invariant, and $X\setminus G\subset P(\mathcal S(Y))$. \Cref{lem:saturated}(b) therefore gives
\begin{equation*}
\dim_{\mathcal H}(X\setminus G)\le\dim_{\mathcal H}P(\mathcal S(Y))\le(N-2)-k=m-2.
\end{equation*}
If $m\le1$, then $\dim_{\mathcal H}\mathcal S(Y)\le N-2<k$ and \cref{lem:saturated}(a) force $\mathcal S(Y)=\emptyset$, so $X\setminus G=\emptyset$.
\end{proof}

\begin{proof}[Proof of \cref{thm:master}(ii) and (iii)]
The set $G$ of \cref{lem:setG} is open and contains $P(\mathcal R(Y))$, by part~(a) of that lemma; (ii) and (iii) are its parts~(b) and~(c).
\end{proof}

\section{Exact Fibre Groups and Local Submersions}\label{sec:groups}

The limit geometry must now be transferred back to the approximating spaces $X_i$. There are two tasks. First, we split the deck groups compatibly with $H\to\Gamma$, producing rank-$k$ fibre groups and covers on which the finite group $C$ acts freely. Second, we lift transverse splitting maps to the maximal abelian covers $\hat X_i$, which are almost Euclidean near regular orbits, and complete them there to Euclidean charts.

\subsection{Compatible Quotients and the Fibre Group}\label{ssec:CQ}
The deck group $H_i$ carries no canonical splitting. The limit group $H$ has a canonical identity component $H_0$ and torsion subgroup $C$, while its free factor depends on the complement fixed below. We transport the quotient map $H\to\Gamma$ to $H_i$ by an approximate-homomorphism argument.

Fix a splitting realizing the complement $\Gamma_{\mathrm f}$ of \cref{ssec:setup}, and write its lifted free factor as $\Z^b$, so that $H=H_0\times\Z^b\times C$ with $H_0\cong\R^k$. Then $\Gamma=H/H_0=\Z^b\times C$; let $\pr\colon H\to\Gamma$ be the quotient map, and fix $\epsilon_i$-approximations $(\hat f_i,\phi_i,\psi_i)$, in the sense of \cref{ssec:eqconv}, for $(\hat X_i,\hat p_i,H_i)\to(Y,\hat p,H)$. In this section we write $H_i$ and $H$ multiplicatively and $\Gamma$ additively. \Cref{prop:CQ} constructs compatible homomorphisms $\Pi_i\colon H_i\to\Gamma$. Their rank-$k$ kernels $H_i'$, which we call the \emph{fibre groups}, converge to $H_0$. The finite factor of $\Gamma$ then yields the $C$-covers $X_i'$ in \cref{cor:Ccover}. For $\delta>0$, set
\begin{equation*}
U_\delta:=\{h\in H:d(hz,z)<\delta\ \text{for all }z\in B(\hat p,1/\delta)\}.
\end{equation*}
Since $H$ is a Lie group, $H_0$ is open, and we fix $\bar\delta$ with $U_{\bar\delta}\subset H_0$. Write
\begin{equation*}
H_i(R):=\{g\in H_i:d(g\hat p_i,\hat p_i)\le R\},
\qquad H(R):=\{h\in H:d(h\hat p,\hat p)\le R\}.
\end{equation*}
For the subgroup $H_i'\le H_i$ of \cref{prop:CQ}(b) we put $H_i'(R):=H_i'\cap H_i(R)$. For every $R$ and all large $i$, the axioms \ref{ax:E1}--\ref{ax:E3} give
\begin{equation}\label{eq:almosthom}
 \phi_i(g_1g_2)^{-1}\phi_i(g_1)\phi_i(g_2)\in U_{8\epsilon_i},\qquad\psi_i(h)\in H_i(R+1),\qquad\phi_i(\psi_i(h))^{-1}h\in U_{8\epsilon_i}
\end{equation}
whenever $g_1,g_2,g_1g_2\in H_i(R)$ and $h\in H(R)$. Indeed, in the first and third cases the two isometries move every point of the image of $\hat f_i$ in a large ball to within $3\epsilon_i$ of each other, and this image is $\epsilon_i$-dense; the second follows from \ref{ax:E1}, \ref{ax:E3} and $f_i(\hat p_i)=\hat p$.

Homomorphisms on $H_i$ are built from maps that are multiplicative on a ball, and the following lemma extends such maps.

\begin{lemma}[Extension of Pseudo-Homomorphisms {\cite[Lemma~6.6]{HHWZ26}}]\label{lem:det}
Let $A$ be an abelian group, and let $\varphi\colon H_i(20D)\to A$ satisfy $\varphi(g_1g_2)=\varphi(g_1)\varphi(g_2)$ whenever $g_1$, $g_2$ and $g_1g_2$ lie in $H_i(20D)$. Then $\varphi$ extends uniquely to a homomorphism $H_i\to A$.
\end{lemma}
\begin{proof}
Fix $i$, and let $p\in X_i$ be the image of $\hat p_i$. By (P3), $X_i$ has a universal cover $\tilde X$, which with the lifted metric and measure is $\RCD(K,N)$, hence proper and geodesic; the group $\mathsf F:=\pi_1(X_i,p)$ acts on it by deck transformations, properly discontinuously, so as a closed subgroup of $\Isom(\tilde X)$, with $\tilde X/\mathsf F=X_i$ of diameter at most $D$. The maximal abelian cover is $\hat X_i=\tilde X/[\mathsf F,\mathsf F]$, with $H_i=\mathsf F/[\mathsf F,\mathsf F]$; write $\mathrm{ab}\colon\mathsf F\to H_i$, and choose $\tilde p$ over $\hat p_i$. The projection $\tilde X\to\hat X_i$ maps $g\tilde p$ to $\mathrm{ab}(g)\hat p_i$ and does not increase distances.

Let $\mathsf F(20D):=\{g\in\mathsf F:d(g\tilde p,\tilde p)\le20D\}$. By \cite[Lemma~76]{SRZ23}, the elements of norm at most $3D$ generate $\mathsf F$, so $\mathsf F(20D)$ generates $\mathsf F$. As $\tilde X$ is simply connected and locally path connected, it has no non-trivial connected covering, so \cite[Lemma~78]{SRZ23} makes $\mathsf F(20D)$ a determining set in the sense of \cite[Definition~59]{SRZ23}. Since $\mathrm{ab}(\mathsf F(20D))\subset H_i(20D)$, the map $\tilde\varphi:=\varphi\circ\mathrm{ab}$ is defined on $\mathsf F(20D)$, and $\tilde\varphi(s_1s_2)=\tilde\varphi(s_1)\tilde\varphi(s_2)$ whenever $s_1,s_2,s_1s_2\in \mathsf F(20D)$. By \cite[Proposition~83]{SRZ23}, $\tilde\varphi$ extends to a homomorphism $\mathsf F\to A$, which vanishes on $[\mathsf F,\mathsf F]$ because $A$ is abelian, so it equals $\bar\varphi\circ\mathrm{ab}$ for a homomorphism $\bar\varphi\colon H_i\to A$.

The homomorphism $\bar\varphi$ extends $\varphi$. Let $h\in H_i(20D)$. As $\hat X_i$ is geodesic, some path from $\hat p_i$ to $h\hat p_i$ has length at most $20D$; its lift from $\tilde p$ ends at some $g\tilde p$ with $\mathrm{ab}(g)=h$ and $d(g\tilde p,\tilde p)\le20D$. Thus $g\in \mathsf F(20D)$ and $\bar\varphi(h)=\varphi(\mathrm{ab}(g))=\varphi(h)$. Finally, $\bar\varphi$ is unique, because $\mathrm{ab}(\mathsf F(20D))\subset H_i(20D)$ generates $H_i$.
\end{proof}

We can now transport the splitting of $H$ to the deck groups.

\begin{proposition}[Compatible Quotient Map]\label{prop:CQ}
For large $i$ there is a surjective homomorphism $\Pi_i\colon H_i\to\Gamma$ such that:
\begin{enumerate}[label=\textup{(\alph*)}]
\item for every $R$ there is $i_0(R)$ such that, for $i\ge i_0(R)$, $\Pi_i(g)=\pr(\phi_i(g))$ for $g\in H_i(R)$ and $\Pi_i(\psi_i(h))=\pr(h)$ for $h\in H(R)$;
\item $H_i':=\ker\Pi_i$ converges to $H_0$. That is, limits of elements of $H_i'$ of bounded displacement lie in $H_0$, and every element of $H_0$ is such a limit;
\item $\phi_i(H_i'(R))\subset H_0$ for $i\ge i_0(R)$.
\end{enumerate}
\end{proposition}

\begin{proof}
\emph{Step 1: Pseudo-Homomorphism.} Fix $R'$. If $g_1,g_2,g_1g_2\in H_i(R')$ and $i$ is sufficiently large depending on $R'$, then $\phi_i(g_1g_2)^{-1}\phi_i(g_1)\phi_i(g_2)\in U_{8\epsilon_i}\subset H_0$ by \eqref{eq:almosthom}. We use this observation with $R'=20D$ here and with $R'=R+3D$ in Step~2. Thus $g\mapsto\pr(\phi_i(g))$ is a pseudo-homomorphism into the abelian group $\Gamma$, and \cref{lem:det} extends it to a homomorphism $\Pi_i$.

\emph{Step 2: Proof of (a).} Let $g\in H_i(R)$ and set $M:=\lceil R/D\rceil+1$. Because $\diam X_i\le D$, we can write $g=a_0\cdots a_{M-1}$ with each $a_j\in H_i(3D)$. Indeed, choose $g_j$ so that $g_j\hat p_i$ lies within $D$ of the equally spaced points on a geodesic from $\hat p_i$ to $g\hat p_i$, and set $a_j:=g_j^{-1}g_{j+1}$. Every partial product lies in $H_i(R+3D)$, so iterating Step~1 with $R'=R+3D$ gives
\begin{equation*}
\phi_i(g)^{-1}\textstyle\prod_j\phi_i(a_j)\in U_{CM\epsilon_i}\subset H_0 .
\end{equation*}
Hence $\Pi_i(g)=\sum_j\pr\phi_i(a_j)=\pr\phi_i(g)$. If $h\in H(R)$, then \eqref{eq:almosthom} gives $\psi_i(h)\in H_i(R+1)$ and
$\phi_i(\psi_i(h))^{-1}h\in U_{8\epsilon_i}\subset H_0$ for large $i$. Applying the first equality with $R+1$ yields
\mbox{$\Pi_i(\psi_i(h))=\pr\phi_i(\psi_i(h))=\pr h$}. This proves both assertions in~(a).

\emph{Step 3: Surjectivity.} $\Gamma$ is generated by $\pr H(3D)$, and (a) with $R=3D$ shows that $\Pi_i\circ\psi_i$ agrees with $\pr$ on $H(3D)$.

\emph{Step 4: Proofs of (b) and (c).} Part~(c) is part~(a) applied to $g\in\ker\Pi_i$. If $g_i\in H_i'(R)$ converge to $h$, then $h\in\phi_i(g_i)U_{\bar\delta}\subset H_0$. Conversely, every $h\in H_0$ is the limit of $\psi_i(h)$, and $\Pi_i(\psi_i(h))=\pr(h)=0$.
\end{proof}

\Cref{prop:CQ} is the abelian case of \cite[Theorem~3.4]{Wan23}, with simple connectivity replaced by \cref{lem:det} as in \cite[Theorem~6.8]{HHWZ26}; the compatibility (a) is implicit in the proof there.

The fibre groups are generated at a uniform scale and have the rank of $H_0$.

\begin{lemma}[The Fibre Group]\label{lem:fibre}
For large $i$, the lifts $h_{j,i}:=\psi_i(0,e_j,0)$ of the standard generators $e_j$ of $\Z^b$ and $\lambda_c:=\psi_i(0,0,c)$ of the elements $c\in C$ have displacement at $\hat p_i$ at most a constant $R_*$ independent of $i$, and $\Pi_i(h_{j,i})=e_j$, $\Pi_i(\lambda_c)=c$. Moreover, for large $i$:
\begin{enumerate}[label=\textup{(\alph*)}]
\item $H_i'$ is generated by $H_i'(R_2)$, for a constant $R_2$ independent of $i$;
\item $\rank H_i'=k$.
\end{enumerate}
\end{lemma}

\begin{proof}
Let $R$ bound the displacements at $\hat p$ of the finitely many elements $(0,e_j,0)$ and $(0,0,c)$. By \eqref{eq:almosthom}, $\psi_i(H(R))\subset H_i(R+1)$ for large $i$, so $R_*:=R+1$ bounds the displacements of the lifts, and \cref{prop:CQ}(a) gives $\Pi_i(h_{j,i})=e_j$ and $\Pi_i(\lambda_c)=c$.

(a) The set $\pr(H(3D+1))$ is finite, so its $\Z^b$-coordinates are uniformly bounded. After correcting each generator in $H_i(3D)$ by these chosen lifts, the remaining factor lies in $H_i'(R_1)$ for a uniform $R_1$. A word in the resulting generators belongs to $H_i'$ exactly when its $\Z^b$-exponents vanish and its $C$-exponents lie in $\ker(\Z^C\to C)$. This kernel is generated by the relations $[c]+[c']-[c+c']$, where $[c]$ is the basis vector of $\Z^C$ indexed by $c$; their lifts lie in $H_i'(3R_*+3)$. Hence $H_i'(R_2)$ generates $H_i'$, with $R_2:=\max\{R_1,3R_*+3\}$.

(b) $H_i/H_i'\cong\Z^b\times C$ has rank $b$, and $\rank H_i=b_1(X_i)=b+k$.
\end{proof}

Dividing $\hat X_i$ by the kernel of the finite part of $\Pi_i$ produces covers on which $C$ acts freely. Every bundle below is built there first and divided by $C$ afterwards.

\begin{corollary}[The $C$-Cover]\label{cor:Ccover}
Let $\chi\colon\Gamma\to C$ be the projection along $\Z^b$. Put $\chi_i:=\chi\circ\Pi_i$, $\mathsf K_i:=\ker\chi_i$ and $X_i':=\hat X_i/\mathsf K_i$. Also put $\mathsf K:=\pr^{-1}(\Z^b)$, so that $W'=W/\Z^b=Y/\mathsf K$.
\begin{enumerate}[label=\textup{(\alph*)}]
\item $\mathsf K_i\to\mathsf K$, and $(X_i',C)\to(W',C)$ equivariantly. For large $i$ the group approximation is the identity of $C$, under $\chi_i$ and $\chi\circ\pr$.
\item $C$ acts effectively on $W'$.
\item $C$ acts freely on $X_i'$, as the deck group $H_i/\mathsf K_i$ of $X_i'\to X_i$.
\end{enumerate}
\end{corollary}

\begin{proof}
(a) \cref{prop:CQ}(a) shows that every bounded sequence $g_i\in \mathsf K_i$ converges into $\mathsf K$. Conversely, each $h\in\mathsf K$ is the limit of elements $\psi_i(h)\in \mathsf K_i$. So, in \cref{lem:quotients}(a), the limit of the $\mathsf K_i$ is $\mathsf K$ along every subsequence, and the diagonal argument at the end of the proof of \cref{lem:quotients}(b) gives $\mathsf K_i\to\mathsf K$ together with $H_i\to H$ along the whole sequence. We may therefore apply \cref{lem:quotients}(b). For representatives of bounded displacement, its group approximation sends the class of $g$ to the class of $\phi_i(g)$, up to $o(1)$. On choosing the representatives $\lambda_c$ from \cref{lem:fibre}, \cref{prop:CQ}(a) identifies this approximation with $\chi_i$.

(b) Suppose that $c\in C$ acts trivially on $W'$. Choose $y\in Y$ over a
regular point of $X$. Since $c[y]_{\mathsf K}=[y]_{\mathsf K}$, there is $a\in\mathsf K$ with
$cy=ay$. \Cref{lem:blowupreg} gives $\Stab_H(y)=1$, hence $c=a$.
But the fixed splitting $H=H_0\times\Z^b\times C$ has $\mathsf K\cap C=1$, so
$c=1$.

(c) Since $\Pi_i$ and $\chi$ are onto, $\chi_i$ induces an isomorphism $H_i/\mathsf K_i\cong C$. The group $H_i/\mathsf K_i$ is the deck group of the normal covering $X_i'=\hat X_i/\mathsf K_i\to\hat X_i/H_i=X_i$, so it acts freely.
\end{proof}

Elements of the fibre group converge into $H_0$, whose orbits project to single points downstairs. At a fixed small scale they therefore cannot move between the sheets over a ball.

\begin{lemma}[The Fibre Group Stabilizes Components over Small Balls]\label{lem:fibreball}
Let $Z_i$ be $X_i$ or $X_i'$ (\cref{cor:Ccover}), and let $s>0$. For large $i$ the following holds. For every $\hat x\in\hat X_i$, with image $x\in Z_i$, $H_i'$ stabilizes the component through $\hat x$ of the preimage of $B(x,s)$ in $\hat X_i$.
\end{lemma}

\begin{proof}
\emph{Reduction to a Bounded Set.} The group $H_i$ is abelian and normalizes the deck group of $\hat X_i\to Z_i$. Translation by $H_i$ therefore preserves the relevant stabilizers, so we may assume $d(\hat x,\hat p_i)\le D$. By \cref{lem:fibre}(a), it is enough to consider generators $g\in H_i'(R_2)$. Their displacement at $\hat x$ is at most $R_2+2D$.

\emph{Contradiction.} Suppose that the conclusion fails along a subsequence. After passing to a further subsequence, let $\hat x_i\to\hat y\in Y$ and let the offending generators satisfy $g_i\to h$. \Cref{prop:CQ}(b) gives $h=\varphi_v\in H_0$. The path $t\mapsto\varphi_{tv}\hat y$, $t\in[0,1]$, lies in one $H_0$-orbit. So it projects to a single point of $Z:=X$ or $W'$. Choose $0=t_0<\cdots<t_M=1$ so that the points $z_n:=\varphi_{t_nv}\hat y$ have consecutive distances $<s/4$. Put $z_0^i:=\hat x_i$ and, for $n\ge1$, $z_n^i:=\psi_i(\varphi_{t_nv})\hat x_i$. Then $z_n^i\to z_n$, and $z_M^i$ lies within $o(1)$ of $g_i\hat x_i$. Every $z_n^i$ projects exactly to the image $x_i$ of $\hat x_i$ in $Z_i$. This is automatic for $Z_i=X_i$. For $Z_i=X_i'$, \cref{prop:CQ}(a) gives $\Pi_i(\psi_i(\varphi_{t_nv}))=0$ for large $i$, so these elements lie in $H_i'\le \mathsf K_i$. Consecutive balls $B(z_n^i,s/4)$ overlap and lie in the preimage of $B(x_i,s)$. So they lie in one component, which also contains $g_i\hat x_i$.
Hence $g_i$ stabilizes the component, a contradiction.
\end{proof}

It remains to identify, on the approximating spaces, the subgroup that plays the role of the isotropy group. It is the $\Pi_i$-preimage of the local group.

\begin{lemma}[Local Groups]\label{lem:localgroups}
Let $x\in X$ have a lift $y\in Y$, with images $w\in W$ and $w'\in W'$. Then $C_y=\Stab_H(y)$ is a finite subgroup of $\Tor(H)=C$, and $C_y=\Gamma_w=C_{w'}$ as subgroups of $C$, where $C_{w'}:=\Stab_C(w')$. Let $r>0$ be admissible at $w$ (\cref{ssec:setG}); we also call $r$ \emph{admissible at $x$}. Let $p_i\to x$, choose lifts $\hat q_i\to y$, and let $\Lambda_i\le H_i$ stabilize the component through $\hat q_i$ of the preimage of $B(p_i,r)$. Then, for large $i$,
\begin{equation*}
H_i'\subset\Lambda_i=\Pi_i^{-1}(C_y),
\end{equation*}
and for each $c\in C_y$ the element $\lambda_c=\psi_i(c)$ of \cref{lem:fibre} lies in $\Lambda_i$, satisfies $\Pi_i(\lambda_c)=c$, and moves $\hat q_i$ by $o(1)$.
\end{lemma}

\begin{proof}
\emph{The Groups.} The stabilizer $C_y$ is compact by \cref{lem:proper}(a) and meets $H_0$ trivially because $H_0$ acts freely, so it embeds in the discrete group $\Gamma$; hence it is finite and contained in $\Tor(H)=C$. As $\Gamma_w=\pr(H_0C_y)=\pr(C_y)$ and $\pr|_C=\mathrm{id}$, we get $\Gamma_w=C_y$. If $c\in C$ fixes $w'$, then $c-\gamma\in\Gamma_w\le C$ for some $\gamma$ in the free factor $\Z^b$ of $\Gamma$, which forces $\gamma=0$; hence $C_{w'}=\Gamma_w$.

\emph{The Inclusion.} The inclusion $H_i'\subset\Lambda_i$ is \cref{lem:fibreball} with $Z_i=X_i$.

\emph{$\Pi_i(\Lambda_i)\subset C_y$.} By Gromov's short-generator argument for the component's length metric, $\Lambda_i$ is generated by elements $\lambda$ for which $\hat q_i$ and $\lambda\hat q_i$ are joined inside the component by a path of length at most $5r$. Here we use that $B(p_i,r)$ has intrinsic diameter at most $2r$. Every limit of such generators preserves the component through $y$ of the preimage of $B(x,2r)$, whose stabilizer is $\pr^{-1}(C_y)$. To see the component assertion explicitly, the balls $B_W(\gamma w,2r)$ for distinct cosets $\gamma\Gamma_w$ are separated by at least $2r$, while each preimage $P_W^{-1}(B_W(\gamma w,2r))=H_0\cdot B(y',2r)$, with $y'$ over $\gamma w$, is connected because both $H_0$ and the ball are connected. The $\hat f_i$-images of the connecting paths are fine chains over $\bar B(x,r+o(1))$, so they cannot jump between these components.

At the base point $\hat p_i$, these generators have displacement at most $5r+2d(\hat q_i,\hat p_i)$, uniformly bounded in $i$. \Cref{prop:CQ}(a) therefore gives $\Pi_i(\lambda)=\pr\phi_i(\lambda)$ for all such generators once $i$ is large. These images lie in $C_y$. Otherwise, after passing to a subsequence, generators $\lambda_i$ with $\pr\phi_i(\lambda_i)\notin C_y$ would converge to an element $h$ with $\pr h\in C_y$; discreteness of $\Gamma$ would then force $\pr\phi_i(\lambda_i)=\pr h\in C_y$ for large $i$, a contradiction.

\emph{The Elements $\lambda_c$, and $C_y\subset\Pi_i(\Lambda_i)$.} Let $c\in C_y$. As $c$ fixes $y$ and $\hat q_i\to y$, \ref{ax:E1} and \ref{ax:E3} show that $\lambda_c=\psi_i(c)$ moves $\hat q_i$ by $o(1)$; so it lies in $\Lambda_i$, and $\Pi_i(\lambda_c)=\pr(c)=c$ by \cref{prop:CQ}(a).

Since $\ker\Pi_i=H_i'\subset\Lambda_i$, we conclude $\Lambda_i=\Pi_i^{-1}(C_y)$.
\end{proof}

\subsection{Almost Euclidean Covers and Local Submersions}
Near a lift of a regular orbit, $\hat X_i$ is almost Euclidean at every sufficiently small scale. A splitting map on the quotient can therefore be lifted and completed to a canonical-Reifenberg chart; in those coordinates, it becomes a projection and hence a submersion.

\begin{lemma}[Almost Euclidean Covers]\label{lem:rewinding}
\begin{enumerate}[label=\textup{(\alph*)}]
\item Assume (R). For every $\eta>0$ there is $r>0$ such that, for large $i$, every ball $B(\hat z,s)\subset\hat X_i$ with $s\le10r$ is $\Psi(\eta|N)s$-close to $B^N(s)$, in the measured sense, and has volume ratio $\ge1-\Psi(\eta|N)$.
\item Let $y_0\in\mathcal R(Y)$, $\eta\in(0,\eta_0]$, $r_0:=r_0(y_0,\eta)$ as in \cref{lem:tube}, and $\hat y_i\in\hat X_i$ with $\hat y_i\to y_0$. For large $i$ the conclusion of (a) holds for the balls $B(\hat z,s)$ with $d(\hat z,\hat y_i)\le\eta r_0/2$ and $s\le r_0/8$. By $H_i$-invariance it also holds for the balls of radius $\le r_0/8$ centred at any lift of a point of $B(\pi(\hat y_i),\eta r_0/2)$, for the projection $\pi$ of $\hat X_i$ to $X_i$ or $X_i'$.
\end{enumerate}
In case (a), $\hat X_i$ is a topological $N$-manifold, and in case (b) so is the set of points within $\eta r_0/4$ of the $H_i$-orbit of $\hat y_i$ \cite[Theorem~A.1.1]{CC97}.
\end{lemma}

\begin{proof}
(a) Replace $\eta$ at the outset by $\eta':=\min\{\eta,\eta_0(N)\}$, with the threshold $\eta_0(N)$ of \cref{lem:tube}; the conclusion for $\eta'$ implies the stated one after enlarging $\Psi$. The Bishop--Gromov ratios $\vartheta_s(y)$ of \cref{ssec:packages} are continuous in $y$, non-increasing in $s$, and converge to $1$ as $s\downarrow0$ at every point. Choose a compact set $Q$ with $HQ=Y$. Dini's theorem on $Q$, followed by $H$-invariance, gives $r>0$ with $|K|(20r)^2\le\eta'$ and $\vartheta_{20r}\ge1-\eta'$ throughout $Y$. For large $i$, volume convergence and $H_i$-invariance transfer the bound $\vartheta_{20r}\ge1-2\eta'$ to every ball of radius $20r$ in $\hat X_i$. The volume--Reifenberg package then gives the conclusion for $s\le10r$.

(b) \cref{lem:tube} supplies the required ratio bound on the limit tube. Volume convergence transfers it to radius-$r_0/4$ balls centred within $\eta r_0/2$ of $\hat y_i$; monotonicity and the same package finish the proof.
\end{proof}

Completion is the step that turns a transverse splitting map into a full chart, by adjoining the missing $N-m$ coordinates.

\begin{lemma}[Relative Completion on a Reifenberg Ball]\label{lem:completion}
For every $N$ there is $\delta_c=\delta_c(N)>0$ with the following property. Let $0<\delta\le\delta_c$, and let $(Z,z)$ be a non-collapsed $\RCD(-\delta,N)$ space such that every ball $B(a,t)$ with $a\in B(z,10)$ and $t\le10$ is $\delta t$-close to $B^N(t)$. If ${u\colon B(z,10)\to\R^m}$ is an $(m,\delta)$-splitting map, then there is a Lipschitz map $v$ from $B(z,1/16)$ to $\R^{N-m}$, in the domain of the Laplacian, such that $\Xi:=(u-u(z),v)$ is an $(N,\Psi(\delta|N))$-splitting map on $B(z,1/16)$ and a bi-H\"older embedding on $B(z,1/32)$. The image $\Xi(B(z,1/32))$ is open.
\end{lemma}

\begin{proof}
Write $\Psi=\Psi(\delta|N)$. \Cref{lem:eucsplit}(b), at $(z,5/4)$ with $n=N$, applied to the ball $B(z,10)$, which is $10\delta$-close to $B^N(10)$, gives a harmonic $(N,\Psi)$-splitting map $\xi_0\colon B(z,5/4)\to\R^N$ with Lipschitz bound $2$. The hypothesis with $a=z$ and $t=1/2$ provides a map $\mathcal E\colon B(z,1/2)\to\R^N$ with $\mathcal E(z)=0$ that changes distances by at most $\delta$. By doubling, $u$ is an $(m,C\delta)$-splitting map on $B(z,1/2)$ and $\xi_0$ an $(N,C\Psi)$-splitting map there. \Cref{lem:eucsplit}(c), at $(z,1/2)$, applied to $u$ with $j=m$ and to $\xi_0$ with $j=N$, gives a linear map $A_u\colon\R^N\to\R^m$ with $A_uA_u^{\mathsf T}=I_m$ and a linear isometry $A_\xi$ of $\R^N$ such that $|u-u(z)-A_u\mathcal E|\le\Psi$ and $|\xi_0-\xi_0(z)-A_\xi\mathcal E|\le\Psi$ on $B(z,1/4)$. The rows of $A_uA_\xi^{-1}$ are orthonormal, so $A_uA_\xi^{-1}=\mathrm{pr}_1E$ for a linear isometry $E$ of $\R^N$, where $\mathrm{pr}_1$ is the projection to the first $m$ coordinates. Hence, on $B(z,1/4)$,
\begin{equation*}
|u-u(z)-\mathrm{pr}_1E(\xi_0-\xi_0(z))|\le|u-u(z)-A_u\mathcal E|+|A_uA_\xi^{-1}(A_\xi\mathcal E-(\xi_0-\xi_0(z)))|\le2\Psi.
\end{equation*}

Put $\xi:=E(\xi_0-\xi_0(z))$, again an $(N,\Psi)$-splitting map, and $v:=\mathrm{pr}_2E(\xi_0-\xi_0(z))$, where $\mathrm{pr}_2$ is the projection to the last $N-m$ coordinates. Then $\Xi=(u-u(z),v)$ satisfies ${\|\Xi-\xi\|_{L^\infty(B(z,1/4))}\le\Psi}$, $\operatorname{Lip}\Xi\le C(N)$ and $|\Delta\Xi|\le\Psi$. \Cref{lem:replacement} at scale $1/8$, whose Euclidean hypothesis is part of ours, makes $\Xi$ an $(N,C\Psi)$-splitting map on $B(z,1/8)$ and a bi-H\"older embedding with open image on $B(z,1/16)$. We restrict these conclusions to $B(z,1/16)$ and $B(z,1/32)$.
\end{proof}

Lifting, completing, and descending now give the submersion statement used in \cref{sec:local,sec:global}.

\begin{lemma}[Local Submersions]\label{lem:submersion}
Given $N$ and $L_\Delta\ge0$, there is $\delta_0=\delta_0(N,L_\Delta)>0$ with the following property. Let $i$ be large enough that $\mathfrak m_i=c_i\mathcal H^N$ (\cref{prop:chain}), and let $Z$ be the quotient of $\hat X_i$ by a subgroup of $H_i$, with projection $\pi$; for instance $Z=X_i$ or $X_i'$. Let $x\in Z$, $\hat x\in\hat X_i$ over $x$, and $s>0$ with $|K|s^2\le\delta_0$. Suppose every ball $B(\hat z,t)$ with $\hat z\in B(\hat x,10s)$ and $t\le10s$ is $\delta_0t$-close to $B^N(t)$. Let $u\colon B(x,10s)\to\R^m$ be Lipschitz, in the domain of the Laplacian, with
\begin{equation*}
\operatorname{Lip}u\le C(N),\qquad |\Delta u|\le L_\Delta/s,\qquad\fint_{B(x,10s)}|\langle\nabla u_a,\nabla u_b\rangle-\delta_{ab}|\le\delta_0 .
\end{equation*}
Then $Z$ is a topological $N$-manifold near $x$, and $u$ is a topological submersion at $x$: there are a neighbourhood $O$ of $x$ and a homeomorphism $O\to U\times V$, with $U\subset\R^m$ and $V\subset\R^{N-m}$ open, under which $u$ becomes the projection to $U$.
\end{lemma}

\begin{proof}
\emph{Lift.} For such $i$ the lifted measure on $\hat X_i$ is a constant multiple of $\mathcal H^N$, and so is the quotient measure on $Z$; both spaces are $\RCD(K,N)$ by (P3), hence non-collapsed, which is what the area formula, Bishop--Gromov for $\mathcal H^N$ and \cref{lem:completion} use below. Put $\hat u:=u\circ\pi$. Since the covering projection is a local isomorphism of metric measure spaces, $\Delta\hat u=(\Delta u)\circ\pi$. Let $\nu$ be the number of deck transformations $g$ of $\pi$ with $d(g\hat x,\hat x)<10s$. Every $z\in B(x,5s)$ has a lift $\hat z\in B(\hat x,5s)$, obtained by lifting a minimizing geodesic from $x$; by the triangle inequality, $z$ has at most $\nu$ lifts in $B(\hat x,5s)$ and at least $\nu$ in $B(\hat x,15s)$. Integrating over $B(x,5s)$, by the area formula for the local isometry $\pi$, gives, for $q:=|\langle\nabla u_a,\nabla u_b\rangle-\delta_{ab}|$,
\begin{equation*}
 \int_{B(\hat x,5s)}q\circ\pi\le\nu\int_{B(x,5s)}q,
 \qquad
 \nu\,\mathcal H^N(B(x,5s))\le\mathcal H^N(B(\hat x,15s)).
\end{equation*}
Dividing the first inequality by $\mathcal H^N(B(\hat x,5s))$, using the second, and applying Bishop--Gromov on the cover and on $Z$, with $|K|s^2\le\delta_0$, we obtain
\begin{equation*}
\fint_{B(\hat x,5s)}|\langle\nabla\hat u_a,\nabla\hat u_b\rangle-\delta_{ab}|\le C(N)\fint_{B(x,10s)}q\le C(N)\delta_0.
\end{equation*}

\emph{Smaller Scale.} Let $\mu=\mu(N,L_\Delta)\in(0,10^{-2})$ be small and $s':=\mu s$. Then
\begin{equation*}
|\Delta\hat u|\le\mu L_\Delta/s',\qquad
\fint_{B(\hat x,10s')}|\langle\nabla\hat u_a,\nabla\hat u_b\rangle-\delta_{ab}|
 \le C(N)\mu^{-N}\delta_0 .
\end{equation*}
Choose $\mu$ first so that $C(N)\mu L_\Delta\le\delta_c/3$, and then $\delta_0$ so that
\begin{equation*}
\delta_*:=C(N)\bigl(\mu L_\Delta+\mu^{-N}\delta_0+\delta_0\bigr)\le\delta_c.
\end{equation*}
After rescaling the metric by $s'^{-1}$ and replacing the target map by
\begin{equation*}
u_{\mathrm{sc}}:=s'^{-1}\bigl(\hat u-\hat u(\hat x)\bigr),
\end{equation*}
the map $u_{\mathrm{sc}}$ is an $(m,\delta_*)$-splitting map on the radius-$10$ ball, and all hypotheses of \cref{lem:completion} are satisfied with $\delta=\delta_*$.

\emph{Completion and Chart.} Apply \cref{lem:completion} in the rescaled metric, let $v_{\mathrm{sc}}$ be the resulting complementary map, and put $v:=s'v_{\mathrm{sc}}$. The map $v$ is defined on $B(\hat x,s'/16)$, and \mbox{$\Xi:=(\hat u-\hat u(\hat x),v)$} is an $(N,\Psi(\delta_*|N))$-splitting map there and a bi-H\"older homeomorphism on $B(\hat x,s'/32)$ onto an open subset of $\R^N$. In this chart, $\hat u-\hat u(\hat x)$ is the projection to the first $m$ coordinates.

\emph{Descent.} Because $\pi$ is a covering map, we may choose a product box $U\times V$ about $\Xi(\hat x)$, compactly contained in the chart image, whose inverse image $\hat O$ lies in an evenly covered neighbourhood of $\hat x$. Then $O:=\pi(\hat O)$ is an open neighbourhood of $x$, the restriction $\pi|_{\hat O}$ is a homeomorphism, and the induced chart $O\to U\times V$ carries $u-u(x)$ to the projection onto $U$; translating $U$ by $u(x)$ gives the statement.

\end{proof}

\section{The Orbit Coordinate and Local Seifert Fibrations}\label{sec:local}

This section proves part~(iv). \Cref{sec:eqmaps} supplies the transverse coordinate and \cref{sec:groups} the fibre group; what remains is the orbit coordinate. We construct it, pair it with the transverse coordinate to form a product chart, and descend to the Seifert model. Write $P'\colon Y\to W'=W/\Z^b$, and fix once and for all the quotient approximations $f_i'\colon X_i'\to W'$ and $f_i\colon X_i\to X$ induced, as in \cref{ssec:eqconv}, by the equivariant approximations $(\hat f_i,\phi_i,\psi_i)$ fixed in \cref{ssec:CQ}; every later subsequence uses their restrictions. If $m=0$, then $X$ is a point and $X_i\cong\T^N$ for large $i$ \cite{MMP22,ZZ26}, proving (iv) and (v). Henceforth $m\ge1$.

\subsection{Local Equivariant Splitting Maps}\label{ssec:eqsplit}
We average splitting maps on $X_i'$ over the local group, making their symmetry exact while preserving the splitting estimates. Their subsequential limits are equivariant charts of the kind constructed in \cref{prop:U2}.

Fix $y_0\in\mathcal R(Y)$, and set $r:=r(y_0)$, $w_0:=P_W(y_0)$, and $w_0':=P'(y_0)$. By \cref{lem:localgroups}, $C_0:=C_{w_0'}=\Gamma_{w_0}=C_{y_0}$, and it acts on $\R^m$ through $\rho_0:=\rho_{y_0}$ (\cref{prop:blowup}(b)). The map $W\to W'$ restricts to an isometry from $B_W(w_0,r/\epsilon)$ onto $B_{W'}(w_0',r/\epsilon)$: non-zero elements of $\Z^b$ lie outside $\Gamma_{w_0}$ and move $w_0$ by at least $6r_\Gamma(w_0)\ge600r/\epsilon$ (here $r_\Gamma(w_0)$ is the radius of \cref{ssec:setG} that separates $B_W(w_0,3r_\Gamma(w_0))$ from its translates by $\Gamma\setminus\Gamma_{w_0}$, not the chart radius $r$), and balls project onto balls (\cref{lem:proper}(b)). Choose $x_i^0\in X_i'$ with $f_i'(x_i^0)\to w_0'$, and lifts $\hat y_i\in\hat X_i$ of $x_i^0$ with $\hat y_i\to y_0$. By \cref{cor:Ccover}(c), $C$ acts freely and isometrically on $X_i'$, and by \cref{cor:Ccover}(a), $f_i'$ is uniformly almost equivariant. Each $X_i'$ is $\RCD(K,N)$ with measure a constant multiple of $\mathcal H^N$, by (P3), as the quotient of $\hat X_i$ by $\mathsf K_i$, which acts freely, properly discontinuously and by measure-preserving isometries (\cref{cor:Ccover}(c) and \cref{prop:chain}); the constant plays no role below.

\emph{The Model at Scale $r$.} Since $r\le r_\epsilon(y_0)$, the labelled equivariant space $(r^{-1}Y,y_0;H_0,C_0)$ is $\epsilon$-close, on balls of radius $1/\epsilon$, to
$(\R^k\times\R^m,0;\text{translations of }\R^k,(\mathrm{id},\rho_0))$.
Let $\mathcal F=(\mathcal F',\mathcal F'')$ denote the corresponding approximation in the unscaled metric. By \cref{lem:fixedlinear}, its group part is $\Psi r$-close, on elements of displacement at most $50r$, to a linear isomorphism $H_0\to\R^k$. We identify $H_0$ with $\R^k$ through that isomorphism and recall $\mathfrak d_{y_0}(v)=d(vy_0,y_0)$. $\mathcal F''$ is $\Psi r$-almost $H_0$-invariant. Let $\bar{\mathcal F}''$ be the map that $\mathcal F''$ induces on $W'$, normalized by $\bar{\mathcal F}''(w_0')=0$. It makes $(r^{-1}W',w_0';C_0)$ $\Psi$-equivariantly close to $(\R^m,0;\rho_0)$ on radius $1/(4\epsilon)$.
Averaging over $C_0$ on the approximating spaces gives the following analogue of \cref{prop:U2}.

\begin{proposition}[Equivariant Splitting Maps]\label{prop:eqsplit}
For large $i$ there are maps $\Phi_i\colon B(x_i^0,30r)\to\R^m$ with the following properties, where $\Psi=\Psi(\epsilon|N)$.
\begin{enumerate}[label=\textup{(\alph*)}]
\item $\Phi_i(cx)=\rho_0(c)\Phi_i(x)$ whenever $c\in C_0$ and $x,cx\in B(x_i^0,30r)$.
\item $\operatorname{Lip}\Phi_i\le C(N)$ and $|\Delta\Phi_i|\le\Psi/r$. Moreover $\Phi_i$ is an $(m,\Psi)$-splitting map on $B(x,r/10)$ for every $x\in B(x_i^0,4r)$.
\item The map $\bar{\mathcal F}''$ is a $\Psi r$-GH approximation on $B_{W'}(w_0',9r)$, and $|\Phi_i-\bar{\mathcal F}''\circ f_i'|\le2\Psi r$ on $B(x_i^0,9r)$ for all large $i$. Every subsequence has a further subsequence along which $\Phi_i$ converges uniformly under the approximations $f_i'$, on $B(x_i^0,7r)$. Every such limit $\Phi\colon B_{W'}(w_0',7r)\to\R^m$ is $C_0$-equivariant, satisfies $\operatorname{Lip}\Phi\le C(N)$ and $|\Phi-\bar{\mathcal F}''|\le\Psi r$, and maps $B_{W'}(w_0',r/8)$ homeomorphically onto an open subset of $\R^m$. The set $\mathcal L$ of these limits is compact in the topology of uniform convergence.
\end{enumerate}
\end{proposition}

\begin{proof}
\emph{Splitting Maps.} Since $X_i'\to W'$ and $|K|(320r)^2\le\epsilon$, the map $\mathcal E:=\bar{\mathcal F}''\circ f_i'-\bar{\mathcal F}''(f_i'(x_i^0))$ is a $\Psi r$-GH approximation of $B(x_i^0,r/(4\epsilon))$ onto a Euclidean ball with $\mathcal E(x_i^0)=0$. By the almost equivariance of both $f_i'$ and $\bar{\mathcal F}''$, for $c\in C_0$ we have $|\mathcal E\circ c-\rho_0(c)\circ\mathcal E|\le\Psi r$ whenever both points lie in $B(x_i^0,20r)$. \Cref{lem:align}, applied with $\mathsf G=C_0$ at scale $2r$ and with $8\epsilon$ in place of $\epsilon$, gives, after restriction, $(m,\Psi)$-splitting maps $u_i\colon B(x_i^0,40r)\to\R^m$ with
\begin{equation*}
|u_i-\mathcal E|\le\Psi r,\qquad
|u_i\circ c-\rho_0(c)\circ u_i|\le\Psi r
\end{equation*}
on $B(x_i^0,10r)$ for every $c\in C_0$, the second whenever $x,cx\in B(x_i^0,10r)$.

\emph{Averaging.} Elements of $C_0$ move $x_i^0$ by $o(1)$. Put
\begin{equation*}
\Phi_i:=\frac1{|C_0|}\sum_{c\in C_0}\rho_0(c)^{-1}\circ u_i\circ c
\end{equation*}
on the $C_0$-invariant set $\{x:C_0x\subset B(x_i^0,40r)\}\supseteq B(x_i^0,30r)$. \Cref{lem:averaging} gives the exact equivariance in (a), the Lipschitz and Laplacian bounds, and $|\Phi_i-u_i|\le\Psi r$ on $B(x_i^0,9r)$. It also shows that $\Phi_i$ is an $(m,\Psi)$-splitting map on $B(x_i^0,5r)$ and, by doubling, on every $B(x,r/10)$ with $x\in B(x_i^0,4r)$. This is (b).

\emph{The Limit.} Uniform Lipschitz bounds give, from every subsequence, a further subsequence for which $\Phi_i$ converges uniformly under $f_i'$ on $B(x_i^0,7r)$ to a map $\Phi$ on $B_{W'}(w_0',7r)$ with $\operatorname{Lip}\Phi\le C(N)$. The family $\mathcal L$ of all such limits is bounded, equicontinuous, and closed by a diagonal argument, hence compact. Part~(a) and the almost equivariance of $f_i'$ make every $\Phi\in\mathcal L$ $C_0$-equivariant. Since $f_i'(x_i^0)\to w_0'$ and $\bar{\mathcal F}''(w_0')=0$, the centring constant satisfies $|\bar{\mathcal F}''(f_i'(x_i^0))|\le\Psi r+o(1)$. Absorbing it into the modulus gives $|\Phi_i-\bar{\mathcal F}''\circ f_i'|\le2\Psi r$ on $B(x_i^0,9r)$. Passing to the limit adds at most the distortion $\Psi r$ of $\bar{\mathcal F}''$; after $\Psi$ is replaced by $3\Psi$, which keeps the preceding bounds, this gives $|\Phi-\bar{\mathcal F}''|\le\Psi r$ on $B_{W'}(w_0',7r)$. The margin between the radii $9r$ and $7r$ absorbs the $o(1)$ error by which $f_i'$ may move a point across a fixed-radius boundary.

It remains to prove that $\Phi$ is injective and open on
$B_{W'}(w_0',r/8)$. Lift $\Phi_i$ to $\hat X_i$. The lifts are
$\mathsf K_i$-invariant and converge uniformly to $\Phi\circ P'$ under
$(\hat X_i,\mathsf K_i)\to(Y,\mathsf K)$, so $\Phi\circ P'$ is $H_0$-invariant.
Local isometry preserves the Lipschitz and Laplacian bounds in~(b); pmGH
stability of Sobolev functions and Laplacians
\cite[Theorem~4.2, Corollary~4.3, Theorem~4.4]{AH18} therefore gives
\begin{equation*}
 \operatorname{Lip}(\Phi\circ P')\le C(N),\qquad
 |\Delta(\Phi\circ P')|\le\Psi/r
 \quad\text{on }B(y_0,6r).
\end{equation*}

Let $\zeta=(\zeta',\zeta'')$ be the map of \cref{prop:U2}(c) at $y_0$, for the approximation $\mathcal F$ above. Both $\zeta''$ and $\Phi\circ P'$ are $C\Psi r$-close to $\mathcal F''$. Hence
\begin{equation*}
 \|\Phi\circ P'-\zeta''\|_{L^\infty(B(y_0,6r))}\le C\Psi r.
\end{equation*}
By \ref{S1}, \cref{prop:U2}(c) applies to $\mathfrak b:=\Phi\circ P'$, with $L=C(N)$ and $\delta=C\Psi$; so $\Phi\circ P'$ descends to a continuous, injective, open map on $B_W(w_0,r/8)$. This ball maps isometrically to $B_{W'}(w_0',r/8)$, proving (c).
\end{proof}

The limit $\Phi$ may depend on the subsequence. The next lemma gives moduli that are uniform over all of $\mathcal L$; these will make the later geometric choices subsequence-independent.

\begin{lemma}[Uniform Charts]\label{lem:unifcharts}
There are non-decreasing functions $\omega_{\mathrm{inj}}$ and $\omega_{\mathrm{op}}$ on $(0,r/100]$, with positive values, such that, for all $\Phi\in\mathcal L$, distinct $a,b\in B_{W'}(w_0',r/9)$ and $t\in(0,r/100]$,
\begin{equation*}
|\Phi(a)-\Phi(b)|\ge\omega_{\mathrm{inj}}\bigl(\min\{d(a,b),r/100\}\bigr),\qquad \Phi(B(a,t))\supset B(\Phi(a),\omega_{\mathrm{op}}(t)).
\end{equation*}
\end{lemma}

\begin{proof}
If the first modulus failed, compactness of $\mathcal L$ would produce a non-injective limit. For the second, uniform convergence and Brouwer degree show that any map sufficiently close to an embedding on $\bar B(a,t/2)$ covers a fixed ball about its value at $a$. Compactness makes the radius uniform in both $\Phi\in\mathcal L$ and $a\in B(w_0',r/9)$.
\end{proof}

\subsection{Local Bundles}\label{ssec:bundles}
We localize $\Phi_i$ over small Euclidean balls. \Cref{lem:submersion} gives a topological submersion, and properness upgrades it to a fibre bundle.

\begin{lemma}[Local Bundles]\label{lem:bundlecrit}
In the setting of \cref{prop:eqsplit}, let $\Phi_i\to\Phi\in\mathcal L$ along a subsequence, and let $\mathbb D$ be an open Euclidean ball with $\bar{\mathbb D}\subset\Phi(B_{W'}(w_0',r/16))$. For large $i$ in the subsequence, the set $E_i'(\mathbb D):=\Phi_i^{-1}(\mathbb D)\cap B(x_i^0,r)$ lies in $B(x_i^0,r/8)$, satisfies ${f_i'(E_i'(\mathbb D))\subset B_{W'}(w_0',r/9)}$, and $\Phi_i\colon E_i'(\mathbb D)\to \mathbb D$ is a fibre bundle whose fibre is a closed $k$-manifold. Moreover, $B_{W'}(w_0',r/9)$ contains $\Phi^{-1}(\bar{\mathbb D})$, and $(f_i')^{-1}(L)\subset E_i'(\mathbb D)$ for every compact $L\subset\Phi^{-1}(\mathbb D)$ and large $i$.
\end{lemma}

\begin{proof}
\emph{Localization.} Regard $\Phi^{-1}(\bar{\mathbb D})$ as a subset of $B_{W'}(w_0',7r)$. \Cref{prop:eqsplit}(c) gives $|\Phi(w)-\Phi(w_0')|\ge d(w,w_0')-3\Psi r$ there. Since $\bar{\mathbb D}\subset B^m(\Phi(w_0'),r/16+3\Psi r)$, the condition $\Phi(w)\in\bar{\mathbb D}$ forces $d(w,w_0')\le r/16+6\Psi r$. By \ref{S2}, $1/16+6\Psi<1/9$. Thus $\Phi_i=\Phi\circ f_i'+o(1)$ gives both $E_i'(\mathbb D)\subset B(x_i^0,r/8)$ and $f_i'(E_i'(\mathbb D))\subset B_{W'}(w_0',r/9)$ for large $i$. It also gives properness, because the preimage of a compact $L\subset \mathbb D$ is closed in $\bar B(x_i^0,r/8)$.

\emph{Bundle Structure.} Let $x\in E_i'(\mathbb D)$. On $B(x,r/10)$, \cref{prop:eqsplit}(b) verifies the hypotheses of \cref{lem:submersion} at radius $r/100$, with Laplacian constant $1$ and almost orthonormality error $\Psi$; by \ref{S3}, these and the curvature condition $|K|(r/100)^2\le\epsilon$ are within the thresholds of that lemma. \Cref{lem:rewinding}(b), with $r_0:=r_0(y_0,\eta)$ and the lifts $\hat y_i$ of $x_i^0$, supplies the Reifenberg hypothesis: the point $x$ lies within $r\le\eta r_0/1000$ of $x_i^0$, and $r/10\le r_0/8$ by the choice of $r$ in \cref{ssec:setG}. Thus $\Phi_i$ is a topological submersion at every point of $E_i'(\mathbb D)$, and $E_i'(\mathbb D)$ is a topological $N$-manifold. It is non-empty by the last paragraph of this proof, since $f_i'$ has $o(1)$-dense image. The properness established above and \cref{lem:bundlepromotion} make $\Phi_i\colon E_i'(\mathbb D)\to \mathbb D$ a fibre bundle with closed $k$-manifold fibre.

Finally, suppose that $L\subset\Phi^{-1}(\mathbb D)$ is compact and $f_i'(x)\in L$. Then $\Phi_i(x)$ lies near the compact subset $\Phi(L)\subset \mathbb D$, while $x\in B(x_i^0,r)$. Hence $x\in E_i'(\mathbb D)$ for large $i$.
\end{proof}

\subsection{From the Fibre Group to a Lattice Coordinate}\label{ssec:fibrecoords}
The preceding construction gives local bundles with closed $k$-manifold fibres. To identify these fibres as tori, and the local group actions as translations, we first realize $H_i'$ modulo torsion as a lattice in the limiting translation group. Long discrete strings approximate the coordinate axes, rounding places every bounded-displacement element near their span, and the fixed bounded generating set makes that span finite index. The maximal-rank identity $\rank H_i'=k$ forces the strings to be independent. Multiplication by the index gives an exact homomorphism $\ell_i\colon H_i'\to\R^k$, and a doubling estimate controls it on short elements. Torsion will be eliminated only after the local product chart has been constructed.

\emph{Standing Notation for \crefrange{ssec:fibrecoords}{ssec:proofiv}.} Fix $x\in G$. By definition of $G$, there are $y_0\in\mathcal R(Y)$ and a point $w'\in B_{W'}(w_0',r/16)$ over $x$, where $w_0'$ and $r=r(y_0)$ are as in \cref{ssec:eqsplit}. By \cref{lem:setG}(b), applied to the lift of $w'$ to $B_W(w_0,r/16)$, and \cref{lem:localgroups}, $C_{w'}\le C_0$ as subgroups of $C$. Choose $y\in Y$ over $w'$ with $d(y,y_0)=d(w',w_0')$. Let $r_*\le r/100$ be admissible for \cref{lem:localgroups} at $x$, and set $\Lambda_i:=\Pi_i^{-1}(C_{w'})$. Admissibility will separate the sets over $w'$ from their translates by $C\setminus C_{w'}$ (\cref{thm:chart}(c)), and the bound $r_*\le r/100$ keeps the radii below in the range of \cref{lem:unifcharts}. We write $H_i$ and its subgroups additively, and their actions multiplicatively: $\lambda z$, and $\lambda^{-1}z:=(-\lambda)z$. $C$ is the torsion subgroup of $H$ and contains every stabilizer (\cref{lem:localgroups}). Write $\pi\colon\hat X_i\to X_i'$ for the projection; $\hat y_i\to y_0$ are the lifts of $x_i^0$ chosen in \cref{ssec:eqsplit}. In \cref{ssec:fibrecoords,ssec:chart} we measure displacement at $\hat y_i$: $\mathfrak d_i(g):=d(g\hat y_i,\hat y_i)$, $H_i'(R):=\{g\in H_i':\mathfrak d_i(g)\le R\}$ and $\Lambda_i(R):=\{\lambda\in\Lambda_i:\mathfrak d_i(\lambda)\le R\}$. This changes $H_i'(R)$ of \cref{sec:groups} only by a bounded shift of $R$. $\Psi=\Psi(\epsilon,\eta\,|\,N)$, and $o(1)$ is as $i\to\infty$ with the other parameters fixed. The constants $\epsilon,\eta$ are those of \cref{ssec:setG}; in particular $\epsilon\le\bar\epsilon(N)$ by \ref{S1}, and \ref{S4} applies to the finitely many constants $C=C(N)$ below.

\emph{The Model on the Cover.} Composing the approximation $\mathcal F$ of \cref{ssec:eqsplit} with the equivariant convergence $(\hat X_i,\hat y_i,H_i)\to(Y,y_0,H)$ gives a model of the cover near $\hat y_i$, on which the deck elements of $\Lambda_i$ act almost as translations of the first factor and through $\rho_0$ on the second. Define $\phi_i^0(\lambda)$ to be the $H_0$-component of $\phi_i(\lambda)$ under $H=H_0\oplus\Z^b\oplus C$.
\begin{lemma}[The Model on the Cover]\label{lem:modeleq}
For large $i$ there are $\Psi r$-GH approximations
\begin{equation*}
\mathcal F_i=(\mathcal F_i',\mathcal F_i'')\colon B(\hat y_i,r/(2\epsilon))\to\R^k\times\R^m,\qquad \mathcal F_i(\hat y_i)=0,
\end{equation*}
with $|\mathcal F_i''-\bar{\mathcal F}''\circ f_i'\circ\pi|\le\Psi r+o(1)$ on $B(\hat y_i,10r)$. For every bounded-displacement $\lambda\in\Pi_i^{-1}(C)$,
$\phi_i(\lambda)=\phi_i^0(\lambda)+\Pi_i(\lambda)\in H_0\oplus C$, and $\phi_i^0=\phi_i$ on bounded-displacement elements of $H_i'$. For $0<R\le40$, uniformly for $\lambda\in\Lambda_i$ with $\mathfrak d_i(\lambda)\le Rr$ and $z,\lambda z\in B(\hat y_i,Rr)$,
\begin{equation}\label{eq:modeleq}
\bigl|\mathcal F_i'(\lambda z)-\mathcal F_i'(z)-\phi_i^0(\lambda)\bigr|\le\Psi r+o(1),\qquad \bigl|\mathcal F_i''(\lambda z)-\rho_0(\Pi_i\lambda)\mathcal F_i''(z)\bigr|\le\Psi r+o(1).
\end{equation}
\end{lemma}
\begin{proof}
Use the approximations $(\hat f_i,\phi_i,\psi_i)$ of \cref{ssec:CQ}. As $\hat f_i(\hat y_i)\to y_0$, displacements at $\hat y_i$ and at $\hat p_i$ differ by a bounded amount, and \ref{ax:E1}--\ref{ax:E3} are available on $B(\hat y_i,r/\epsilon)$ for large $i$. Put $\mathcal F_i:=\mathcal F\circ \hat f_i-\mathcal F(\hat f_i(\hat y_i))$. As $\mathcal F$ is an $\epsilon r$-GH approximation on $B(y_0,r/\epsilon)$, \ref{ax:E1} makes $\mathcal F_i$ a $(10\epsilon r+o(1))$-GH approximation. The transverse estimate follows from $d(f_i'\circ\pi,P'\circ \hat f_i)=o(1)$ on $B(\hat y_i,10r)$ (\cref{cor:Ccover}(a), \cref{ssec:eqconv}) and the almost $H_0$-invariance of $\mathcal F''$, because points of $B(y_0,r/\epsilon)$ with the same image in $W'$ lie in one $H_0$-orbit (\cref{ssec:eqsplit}). For $\lambda\in\Pi_i^{-1}(C)$ of bounded displacement, \cref{prop:CQ}(a) gives $\pr\phi_i(\lambda)=\Pi_i(\lambda)\in C$, so the $\Z^b$-component of $\phi_i(\lambda)$ vanishes. Finally, let $\lambda\in\Lambda_i$, and let $h_0$ be the $H_0$-component of $\phi_i(\lambda)$; its $C$-component $\Pi_i(\lambda)$ lies in $C_{w'}\le C_0$ and fixes $y_0$. Then \ref{ax:E2} and the almost equivariance of $\mathcal F$ under $H_0$ and $C_0$, whose group part maps $h_0$ to a translation within $\Psi r$ of $h_0$ (\cref{ssec:eqsplit}), give \eqref{eq:modeleq}, with errors $\Psi r+o(1)$ uniform in $\lambda$ and $z$.
\end{proof}

\emph{Three Consequences of Convergence.} The function $\mathfrak d_{y_0}$ is continuous, positive on $H_0\setminus\{0\}$ by freeness, and proper by \cref{lem:proper}(a). For $s>0$ set $\mathfrak d_{y_0}^*(s):=\max\{\mathfrak d_{y_0}(v):|v|\le s\}$, and for $t>0$ set $\kappa(t):=\max\{|v|:\mathfrak d_{y_0}(v)\le t\}$. Both functions are finite and non-decreasing, and $\kappa(t)\to0$ as $t\downarrow0$. \Cref{prop:CQ} and equivariant convergence give the following facts, uniformly for each fixed $R$:
\begin{enumerate}[label=\textup{(C\arabic*)}]
\item $\mathfrak d_i(g)=\mathfrak d_{y_0}(\phi_i(g))+o(1)$ for $g\in H_i'(R)$;
\item $\phi_i^0(g+h)=\phi_i^0(g)+\phi_i^0(h)+o(1)$ and $\phi_i^0(-g)=-\phi_i^0(g)+o(1)$ for $g,h\in\Pi_i^{-1}(C)$ with $\mathfrak d_i(g),\mathfrak d_i(h)\le R$;
\item for every $s>0$, each $v\in H_0$ with $|v|\le s$ satisfies $v=\phi_i(g)+o(1)$ for some ${g\in H_i'(\mathfrak d_{y_0}^*(s)+1)}$.
\end{enumerate}

For $a\in\R^k$, write $\lfloor a\rceil\in\Z^k$ for coordinatewise rounding to a nearest integer, with either choice at half-integers.

\begin{proposition}[The Lattice Coordinate]\label{prop:lattice}
For large $i$:
\begin{enumerate}[label=\textup{(\alph*)}]
\item there is a homomorphism $\ell_i\colon H_i'\to\R^k$ satisfying
$\sup_{g\in H_i'(R)}|\ell_i(g)-\phi_i(g)|\to0$ for every $R$;
\item $\ker\ell_i=\Tor(H_i')$, and $L_i:=\ell_i(H_i')$ is a full
lattice in $\R^k$; more precisely, $N_i^{-1}\Z^k\subset L_i$ for integers $N_i\to\infty$;
\item for every $R$ there is $R'$, independent of $i$, such that every $u\in H_i'$ with $|\ell_i(u)|\le R$ lies in $H_i'(R')$.
\end{enumerate}
\end{proposition}

\begin{proof}
Put $R_1:=\mathfrak d_{y_0}^*(1)+1$, and let $e_1,\dots,e_k$ be the standard basis of the fixed identification $H_0=\R^k$.

\emph{Step 1: Long Coordinate Strings.} We claim that there are integers $N_i\to\infty$, elements $u_{i,1},\dots,u_{i,k}\in H_i'$, and numbers $\gamma_i\to0$ such that, for $|n|\le N_i$,
\begin{equation}\label{eq:strings}
 n u_{i,j}\in H_i'(R_1),\qquad
 \left|\phi_i(n u_{i,j})-\frac n{N_i}e_j\right|\le\gamma_i.
\end{equation}
Fix an integer $M$. By (C3) there are $u_{i,j}^{(M)}\in H_i'(R_1)$ with $\phi_i(u_{i,j}^{(M)})=M^{-1}e_j+o(1)$. Suppose inductively that $n u_{i,j}^{(M)}\in H_i'(R_1)$ and has the expected $\phi_i$-coordinate. Subadditivity first puts $(n+1)u_{i,j}^{(M)}$ in $H_i'(2R_1)$. Then (C2) and (C1), applied on that fixed ball, give
\begin{equation*}
 \phi_i((n+1)u_{i,j}^{(M)})=\frac{n+1}{M}e_j+o(1),
 \qquad
 \mathfrak d_i((n+1)u_{i,j}^{(M)})\le\mathfrak d_{y_0}^*(1)+o(1)<R_1.
\end{equation*}
The inverse relation in (C2) treats negative $n$. For fixed $M$ there are only
finitely many $j$ and $|n|\le M$, so choose increasing indices $i(M)\ge M$
such that all these errors are at most $1/M$ whenever $i\ge i(M)$. Define
\begin{equation*}
 N_i:=\max\{M\le i:i(M)\le i\},
 \qquad u_{i,j}:=u_{i,j}^{(N_i)}.
\end{equation*}
After discarding finitely many indices, this proves \eqref{eq:strings} with $\gamma_i\le N_i^{-1}$.

Write $\operatorname{ev}_i(n):=\sum_{j=1}^kn_ju_{i,j}$ and
$\Lambda_i':=\operatorname{ev}_i(\Z^k)$.

\emph{Step 2: Rounding to $\Lambda_i'$.} For $g\in H_i'(R)$ define
$n_i(g):=\lfloor N_i\phi_i(g)\rceil$ and
$g^\sharp:=\operatorname{ev}_i(n_i(g))$.
We claim that
\begin{equation}\label{eq:rounding}
 \tau_i(R):=\sup_{g\in H_i'(R)}\mathfrak d_i(g-g^\sharp)\longrightarrow0
\end{equation}
for every fixed $R$. By (C1), $|\phi_i(g)|\le\kappa(R+1)$ for large $i$. Put $P_R:=\lceil\kappa(R+1)\rceil+2$. Each coordinate of $n_i(g)$ is a sum of at most $P_R$ integers of absolute value at most $N_i$. Thus $g^\sharp$ is a sum of at most $kP_R$ elements $n u_{i,j}$ with $|n|\le N_i$, which are covered by \eqref{eq:strings}, and all partial sums lie in $H_i'(kP_RR_1)$. Applied uniformly on that fixed ball, (C2) gives
\begin{equation*}
 \phi_i(g^\sharp)=\frac{n_i(g)}{N_i}+o(1)=\phi_i(g)+o(1).
\end{equation*}
Now $g-g^\sharp$ and all terms used to form it lie in one fixed displacement ball depending only on $R$. Applying (C2) there, followed by (C1), proves \eqref{eq:rounding} uniformly in $g$.

\emph{Step 3: Finite Index and Independence.} Fix $\delta>0$ and take $i$ so large that
$\tau_i(2\delta)\le\delta$ and $\tau_i(R_3)\le\delta$, where $R_3$ is the generating radius $R_2$ of \cref{lem:fibre}(a) increased by $2\sup_i d(\hat y_i,\hat p_i)$, so that $H_i'(R_3)$, with displacement measured at $\hat y_i$, contains the generators of that lemma. If $a,b\in H_i'(\delta)$, then
\begin{equation*}
 a+b=(a+b)^\sharp+c,\qquad c\in H_i'(\delta).
\end{equation*}
The set $H_i'(\delta)$ is finite because the deck action is properly
discontinuous. Consequently, its image in $Q_i:=H_i'/\Lambda_i'$ is a finite
symmetric set closed under addition, and is therefore a subgroup.
For every generator $g\in H_i'(R_3)$, the rounding estimate writes
$g=g^\sharp+c$ with $c\in H_i'(\delta)$. Thus the image of
$H_i'(\delta)$ contains the images of the generating set from
\cref{lem:fibre}(a), so it equals $Q_i$. Hence $Q_i$ is finite and
every coset of $\Lambda_i'$ meets $H_i'(\delta)$.

Finite index and \cref{lem:fibre}(b) give $\rank\Lambda_i'=\rank H_i'=k$. The surjection $\operatorname{ev}_i\colon\Z^k\to\Lambda_i'$ therefore has a rank-zero kernel. Every subgroup of $\Z^k$ is free, so the kernel is trivial and $\operatorname{ev}_i$ is an isomorphism onto $\Lambda_i'$.

\emph{Step 4: An Exact Homomorphism.} Put $I_i:=|Q_i|$. For $g\in H_i'$
there is a unique $m_i(g)\in\Z^k$ with
$I_ig=\operatorname{ev}_i(m_i(g))$. Uniqueness makes $m_i$ additive. Define
$\ell_i(g):=m_i(g)/(I_iN_i)$. Then $\ell_i$ is a homomorphism,
$\ell_i(\operatorname{ev}_i(n))=n/N_i$, and
$\ker\ell_i=\Tor(H_i')$. Indeed, $\ell_i(g)=0$ exactly when $I_ig=0$.
Conversely, $\Tor(H_i')$ injects into $Q_i$ because $\Lambda_i'$ is torsion-free, so the order of every torsion element divides $I_i$.

\emph{Step 5: The Doubling Estimate.} For fixed $\delta>0$, let
$\mu_i(\delta):=\max_{c\in H_i'(\delta)}|\ell_i(c)|$, and choose a maximizer
$a_i$. By \eqref{eq:rounding}, for large $i$,
$2a_i=(2a_i)^\sharp+a_i''$ with $a_i''\in H_i'(\delta)$.
Since $2a_i\in H_i'(2\delta)$, (C1) gives $\mathfrak d_{y_0}(\phi_i(2a_i))\le3\delta$ for large $i$. Hence
\begin{equation*}
 \left|\ell_i((2a_i)^\sharp)\right|
 =\frac{|\lfloor N_i\phi_i(2a_i)\rceil|}{N_i}
 \le\kappa(3\delta)+\frac{\sqrt k}{2N_i}.
\end{equation*}
Thus
\begin{equation}\label{eq:doubling}
 \begin{aligned}
 2\mu_i(\delta)=|\ell_i(2a_i)|
 &\le|\ell_i((2a_i)^\sharp)|+|\ell_i(a_i'')|\\
 &\le\kappa(3\delta)+\frac{\sqrt k}{2N_i}+\mu_i(\delta),
 \qquad\text{hence}\qquad
 \mu_i(\delta)\le\kappa(3\delta)+\frac{\sqrt k}{2N_i}.
 \end{aligned}
\end{equation}

\emph{Step 6: Comparison with the Limiting Coordinate.} Let $g\in H_i'(R)$ and write $g=g^\sharp+c$. For each fixed $\delta>0$, equation~\eqref{eq:rounding} gives $c\in H_i'(\delta)$ for all large $i$. Since $\ell_i(g^\sharp)=n_i(g)/N_i$, equation~\eqref{eq:doubling} yields
\begin{equation*}
 |\ell_i(g)-\phi_i(g)|
 \le\frac{\sqrt k}{2N_i}+\mu_i(\delta)
 \le\frac{\sqrt k}{N_i}+\kappa(3\delta).
\end{equation*}
The bound is uniform on $H_i'(R)$. First let $i\to\infty$ and then
$\delta\downarrow0$. This proves (a). By Step~4,
$\ker\ell_i=\Tor(H_i')$. Finally,
\begin{equation*}
 N_i^{-1}\Z^k=\ell_i(\Lambda_i')
 \subset L_i\subset(I_iN_i)^{-1}\Z^k.
\end{equation*}
Therefore $L_i$ is a full lattice. This proves (b).

\emph{Step 7: Proof of (c).} Take $i$ so large that Step~3 applies with $\delta=1$ and that \eqref{eq:doubling} holds with $\delta=1$. Let $u\in H_i'$ with $|\ell_i(u)|\le R$. By Step~3, $u=\operatorname{ev}_i(\nu)+u'$ with $\nu\in\Z^k$ and $u'\in H_i'(1)$, and $|\ell_i(u')|\le\kappa(3)+\sqrt k/2$ by \eqref{eq:doubling}. Since $\ell_i(\operatorname{ev}_i(\nu))=\nu/N_i$, every coordinate satisfies $|\nu_j|\le N_iR''$ with $R'':=R+\kappa(3)+\sqrt k/2$. So each $\nu_j$ is a sum of at most $\lceil R''\rceil$ integers of absolute value at most $N_i$, and by \eqref{eq:strings} $\operatorname{ev}_i(\nu)$ is a sum of at most $k\lceil R''\rceil$ elements of $H_i'(R_1)$. Displacement is subadditive, so (c) holds with $R':=k\lceil R''\rceil R_1+1$.
\end{proof}

Since $\Pi_i^{-1}(C)/H_i'\cong C$ is finite, the lattice coordinate extends to $\Pi_i^{-1}(C)$, in particular to the local group $\Lambda_i$.

\begin{corollary}[The Local Group]\label{cor:localgroup}
$\ell_i$ extends uniquely to a homomorphism $\ell_i\colon\Pi_i^{-1}(C)\to\R^k$, namely $\ell_i(\lambda):=\ell_i(e\lambda)/e$ with $e$ the exponent of $C$. For every $R$, $|\ell_i(\lambda)-\phi_i^0(\lambda)|\to0$ uniformly in $\lambda\in\Pi_i^{-1}(C)$ with $\mathfrak d_i(\lambda)\le R$, in particular uniformly on $\Lambda_i(R)$.
\end{corollary}

\begin{proof}
$\Pi_i(e\lambda)=e\,\Pi_i(\lambda)=0$, so $e\lambda\in H_i'$, and the formula defines a homomorphism extending $\ell_i$. It is unique because $\R^k$ is torsion-free. For $\lambda\in\Pi_i^{-1}(C)$ with $\mathfrak d_i(\lambda)\le R$ we have $e\lambda\in H_i'(eR)$, and (C2) gives \mbox{$\phi_i(e\lambda)=e\,\phi_i^0(\lambda)+o(1)$}. So $\ell_i(\lambda)=\ell_i(e\lambda)/e=\phi_i(e\lambda)/e+o(1)=\phi_i^0(\lambda)+o(1)$, by \cref{prop:lattice}(a).
\end{proof}

\subsection{The Orbit Coordinate and the Seifert Chart}\label{ssec:chart}
We first identify the relevant component of the preimage in $\hat X_i$ and its stabilizer (\cref{lem:tubecover}). The remaining step is analytic: we glue the local harmonic coordinates along the group orbits into a single, exactly equivariant coordinate $U_i$. Two overlapping coordinates differ by a harmonic function that is small in $C^0$, and therefore also small in $C^1$; the gluing consequently preserves the splitting estimates. Pairing $U_i$ with the lift of $\Phi_i$ produces the product chart in \cref{thm:chart}.

\emph{Setting.} Work along a subsequence for which $\Phi_i\to\Phi\in\mathcal L$. Choose a ball $\mathbb D$ centred at $\Phi(w')$ such that $\bar{\mathbb D}\subset\Phi(B_{W'}(w_0',r/16))$. We require its preimage $V:=\Phi^{-1}(\mathbb D)$, taken in $B_{W'}(w_0',7r)$, to satisfy $\bar V\subset B_{W'}(w',r_*/10)$. Set $E_i':=\Phi_i^{-1}(\mathbb D)\cap B(x_i^0,r)$. We record three facts, valid for large $i$. By \cref{lem:bundlecrit}, $E_i'\subset B(x_i^0,r/8)$, $\Phi_i\colon E_i'\to \mathbb D$ is a fibre bundle with compact fibre, hence proper, and $(f_i')^{-1}(L)\subset E_i'$ for compact $L\subset V$. $E_i'$ is $C_{w'}$-invariant, because $\rho_0(C_{w'})$ preserves $\mathbb D$ and $C_{w'}$ moves $x_i^0$ by $o(1)$ (\cref{prop:eqsplit}(a)). \Cref{lem:bundlecrit} gives $f_i'(E_i')\subset B_{W'}(w_0',r/9)$, where \cref{lem:unifcharts} applies. Since $\Phi_i=\Phi\circ f_i'+o(1)$, the modulus $\omega_{\mathrm{inj}}$ then gives $f_i'(E_i')\subset B_{W'}(w',r_*/5)$. Choose $w_i\in E_i'$ with $f_i'(w_i)\to w'$, and lifts $\hat w_i\in\hat X_i$ with $\hat w_i\to y$. Then $q_i(w_i)\to x$, where $q_i\colon X_i'\to X_i$ is the quotient map, and $E_i'\subset B(w_i,r_*/4)$. Also $B(w_i,s')\subset E_i'$ for some $s'>0$ independent of $i$, since $\Phi_i(w_i)\to\Phi(w')$ and $\operatorname{Lip}\Phi_i\le C(N)$; \cref{lem:tubecover} applies \cref{lem:fibreball} at this radius.
Let $\hat E_i:=\pi^{-1}(E_i')\subset\hat X_i$, and let $\tilde T_i$ be the component of $\hat E_i$ containing $\hat w_i$. The next lemma identifies $E_i'$ with the quotient of $\tilde T_i$ by the fibre group, and the stabilizer of $\tilde T_i$ in $H_i$ with $\Lambda_i$.

\begin{lemma}[The Tube in the Cover]\label{lem:tubecover}
For large $i$:
\begin{enumerate}[label=\textup{(\alph*)}]
\item $E_i'$ is connected, and $\pi(\tilde T_i)=E_i'$;
\item the stabilizer of $\tilde T_i$ in $\mathsf K_i$ is $H_i'$; so $E_i'=\tilde T_i/H_i'$;
\item the stabilizer of $\tilde T_i$ in $H_i$ is $\Lambda_i$, and for each $c\in C_{w'}$ there is $\lambda_c\in\Pi_i^{-1}(c)$ with ${d(\lambda_c\hat w_i,\hat w_i)\to0}$;
\item $\tilde T_i\subset H_i'\cdot B(\hat w_i,r_*/4)$.
\end{enumerate}
\end{lemma}

\begin{proof}
\emph{Notation.} Put $t:=r_*/4$. Let $\hat B$ be the component through $\hat w_i$ of $\pi^{-1}(B(w_i,t))$, and $\hat B^X$ the component through $\hat w_i$ of the preimage in $\hat X_i$ of $B(q_i(w_i),t)\subset X_i$; so $\hat B\subset\hat B^X$. \Cref{lem:localgroups}, applied with $p_i:=q_i(w_i)\to x$, $\hat q_i:=\hat w_i\to y$ and the admissible radius $t$, gives $\Stab_{H_i}(\hat B^X)=\Lambda_i$. If $g\in \mathsf K_i$ stabilizes $\hat B$, then $g\hat B^X\cap\hat B^X\ne\emptyset$, so
\begin{equation}\label{eq:stabK}
\Stab_{\mathsf K_i}(\hat B)\subset \mathsf K_i\cap\Lambda_i=H_i' ,
\end{equation}
since $\chi$ is the identity on $C$. Let $\hat B'$ be the component through $\hat w_i$ of $\pi^{-1}(B(w_i,s'))$. By \cref{lem:fibreball} with $Z_i=X_i'$, $H_i'$ stabilizes $\hat B'$ for large $i$. Since $\hat B'$ is connected, contains $\hat w_i$ and lies in $\hat E_i$, we have $\hat B'\subset\tilde T_i$. Also $B(\hat w_i,s')\subset\hat B'$.

(a) Each component of $E_i'$ is open and maps onto $\mathbb D$, since $\Phi_i\colon E_i'\to \mathbb D$ is a fibre bundle. Let $z_0,z_1\in E_i'$ lie over the centre $\Phi(w')$ of $\mathbb D$. Since $\Phi_i=\Phi\circ f_i'+o(1)$, and $\Phi$ has the injectivity modulus $\omega_{\mathrm{inj}}$ of \cref{lem:unifcharts} on $B_{W'}(w_0',r/9)\supset f_i'(E_i')$, we get $d(f_i'(z_0),f_i'(z_1))\to0$; hence $d(z_0,z_1)\to0$, uniformly. The ball $B(z_0,2d(z_0,z_1))$ is connected, and lies in $E_i'$ for large $i$, because $\operatorname{Lip}\Phi_i\le C(N)$ and the centre of $\mathbb D$ is interior. So all points over the centre lie in one component, and $E_i'$ is connected. Since $\hat E_i\to E_i'$ is a covering of a connected, locally path-connected space, each of its components maps onto $E_i'$.

(b) Since $E_i'\subset B(w_i,t)$, we have $\tilde T_i\subset\hat B$. If $g\in \mathsf K_i$ stabilizes $\tilde T_i$, it stabilizes $\hat B$, so $g\in H_i'$ by \eqref{eq:stabK}. Conversely, for $h\in H_i'$ the set $h\tilde T_i\supset h\hat B'=\hat B'$ meets $\tilde T_i$, so $h\tilde T_i=\tilde T_i$. With (a), $E_i'=\tilde T_i/H_i'$.

(c) If $h\in H_i$ stabilizes $\tilde T_i$, it stabilizes $\hat B^X\supset\tilde T_i$, so $h\in\Lambda_i$. For $c\in C_{w'}=C_y$, \cref{lem:localgroups}, applied as above, gives $\lambda_c\in\Lambda_i$ with $\Pi_i(\lambda_c)=c$ that moves $\hat w_i$ by $o(1)$. It maps $\hat E_i$ to itself, as $E_i'$ is $C_{w'}$-invariant, and $\lambda_c\hat w_i\in B(\hat w_i,s')\subset\tilde T_i$ for large $i$. So $\lambda_c\tilde T_i=\tilde T_i$. The $\lambda_c$ and $H_i'$ generate $\Lambda_i$.

(d) Let $z\in\tilde T_i$. Since $\pi(z)\in E_i'\subset B(w_i,t)$, there is $g\in \mathsf K_i$ with $d(g^{-1}z,\hat w_i)<t$. The ball $B(\hat w_i,t)$ is connected and projects into $B(w_i,t)$, so it lies in $\hat B$. Thus $z$ and $g^{-1}z$ both lie in $\hat B$, so $g$ stabilizes $\hat B$, and $g\in H_i'$ by \eqref{eq:stabK}. Hence $z\in gB(\hat w_i,t)\subset H_i'\cdot B(\hat w_i,r_*/4)$.
\end{proof}

The gluing below uses a partition of unity indexed by $\Lambda_i$, so we must bound how many translates can be active at a single point.

\begin{lemma}[Counting Orbit Points]\label{lem:count}
For large $i$ and every $z\in\hat X_i$ with $d(z,\Lambda_i\hat w_i)<r/4$,
\begin{equation*}
\begin{aligned}
\mathcal M(z)&:=\#\{\lambda\in\Lambda_i:d(\lambda^{-1}z,\hat w_i)<2r\}\\
&\le C(N)\,\#\{\lambda\in\Lambda_i:d(\lambda^{-1}z,\hat w_i)<r/2\}=:C(N)\mathcal M'(z).
\end{aligned}
\end{equation*}
\end{lemma}

\begin{proof}
Both counts are $\Lambda_i$-invariant, so assume $d(z,\hat w_i)<r/4$. With $\mathsf N(z,s):=\#\{\lambda\in\Lambda_i:d(\lambda z,z)\le s\}$, the triangle inequality gives $\mathcal M'(z)\ge\mathsf N(z,r/5)$. By Bishop--Gromov and $|K|(320r)^2\le\epsilon$ (\cref{ssec:eqsplit}), the measure of $\hat X_i$ is doubling at scales up to $16r$ with a constant $C(N)$, and $\Lambda_i$ acts by measure-preserving isometries. \Cref{lem:orbitcount}(c), with $p=z$, $q=\hat w_i$, $r_2=2r$, $r_1=r/5$ and $\bar R=16r$, therefore gives $\mathcal M(z)\le C(N)\,\mathsf N(z,r/5)$.
\end{proof}

\emph{Local Orbit Charts.} Harmonic charts close to this model make the deck action almost a translation in the first block.

\begin{lemma}[Local Orbit Charts]\label{lem:orbitcharts}
For large $i$ there are maps
\begin{equation*}
\xi_i=(\xi_i',\xi_i'')\colon B(\hat y_i,40r)\to\R^k\times\R^m,\qquad |\xi_i-\mathcal F_i|\le\Psi r\ \text{on }B(\hat y_i,5r).
\end{equation*}
Here $\xi_i$ is harmonic, $|\nabla\xi_i|\le C(N)$ on $B(\hat y_i,30r)$, and $\xi_i$ is an $(N,\Psi)$-splitting map on every $B(z,s)\subset B(\hat y_i,30r)$ with $s\ge r/100$. If $\lambda\in\Lambda_i$ and $z,\lambda z\in B(\hat y_i,5r)$, then
\begin{equation}\label{eq:approxeq}
\bigl|\xi_i'(\lambda z)-\xi_i'(z)-\ell_i(\lambda)\bigr|\le4\Psi r .
\end{equation}
\end{lemma}
\begin{proof}
By \cref{lem:rewinding}(b), \cref{lem:tube} and $r\le\epsilon\eta r_0/1000$, the ball $B(\hat y_i,320r)$ is $\Psi\cdot320r$-close to $B^N(320r)$ in the measured sense, and $|K|(320r)^2\le\eta^3$. By~\ref{S4}, \cref{lem:eucsplit}(b), at $(\hat y_i,40r)$ with $n=N$, gives a harmonic map $v\colon B(\hat y_i,160r)\to\R^N$ with $|\nabla v|\le C(N)$ and Gram averages at most $\Psi$; by doubling, its Gram averages are at most $C(N)\Psi$ on every ball $B(z,s)\subset B(\hat y_i,30r)$ with $s\ge r/100$. As $\mathcal F_i$ changes distances on $B(\hat y_i,40r)$ by at most $\Psi r$ (\cref{lem:modeleq}), \cref{lem:eucsplit}(c) gives a linear isometry $J$ with $|v-v(\hat y_i)-J\circ\mathcal F_i|\le\Psi r$ on $B(\hat y_i,20r)$. Put $\xi_i:=J^{-1}\circ(v-v(\hat y_i))$. Finally, if $\lambda\in\Lambda_i$ and $z,\lambda z\in B(\hat y_i,5r)$, then $\mathfrak d_i(\lambda)\le20r$, and \eqref{eq:modeleq} with $R=20$, the bound $|\xi_i'-\mathcal F_i'|\le\Psi r$ at $z$ and $\lambda z$, and \cref{cor:localgroup} give \eqref{eq:approxeq}.
\end{proof}

\emph{Partition of Unity.} Let $\beta\colon\hat X_i\to[0,1]$ be a good cut-off function \cite[Lemma~3.1]{MN19}: $\beta=1$ on $B(\hat w_i,r/2)$, $\operatorname{supp}\beta\subset B(\hat w_i,r)$, and $r|\nabla\beta|+r^2|\Delta\beta|\le C(N)$. Put
\begin{equation*}
\Sigma(z):=\sum_{\lambda\in\Lambda_i}\beta(\lambda^{-1}z),\qquad \varpi_\lambda(z):=\beta(\lambda^{-1}z)/\Sigma(z).
\end{equation*}
Let $\Omega_i:=\Lambda_i\cdot B(\hat w_i,r/4)$; it contains $\tilde T_i$. On $\Omega_i$, the sum is locally finite, $S\ge\mathcal M'\ge1$ and $\sum_\lambda\varpi_\lambda=1$. By \cref{lem:count}, at most $C(N)\Sigma(z)$ terms are active on $B(z,r)$. Combining this bound with the cutoff estimates and the chain and product rules \cite[Proposition~4.28, Theorem~4.29]{Gig15}, we obtain, almost everywhere on $\Omega_i$,
\begin{equation}\label{eq:weights}
\sum_\lambda|\nabla\varpi_\lambda|\le\frac{C}{r},\qquad\sum_\lambda|\Delta\varpi_\lambda|\le\frac{C}{r^2},\qquad C=C(N).
\end{equation}

Averaging the translates of the local coordinate $\xi_i'$ with these weights, each shifted by the lattice coordinate $\ell_i$, makes the orbit coordinate exactly equivariant.

\begin{proposition}[Equivariant Orbit Coordinate]\label{prop:orbit}
For large $i$, the map
\begin{equation*}
U_i(z):=\sum_{\lambda\in\Lambda_i}\varpi_\lambda(z)\bigl(\xi_i'(\lambda^{-1}z)+\ell_i(\lambda)\bigr),\qquad z\in\Omega_i,
\end{equation*}
is well defined, locally Lipschitz and in the domain of the local Laplacian on $\Omega_i$, and satisfies:
\begin{enumerate}[label=\textup{(\alph*)}]
\item $U_i(\lambda z)=U_i(z)+\ell_i(\lambda)$ for all $\lambda\in\Lambda_i$;
\item on $B(\hat w_i,r/8)$: $|U_i-\xi_i'|\le4\Psi r$ and $|\nabla U_i-\nabla\xi_i'|\le C\Psi$;
\item $|\Delta U_i|\le C\Psi/r$ and $|\nabla U_i|\le C(N)$ on $\Omega_i$.
\end{enumerate}
\end{proposition}

\begin{proof}
If $\varpi_\lambda(z)>0$, then $\lambda^{-1}z\in B(\hat w_i,r)\subset B(\hat y_i,2r)$, so every summand is defined. Write $u^\lambda:=\xi_i'\circ\lambda^{-1}+\ell_i(\lambda)$; this function is harmonic on its domain.

(a) After the substitution $g=\lambda^{-1}$, $U_i(z)$ is the mean of $\nu_z/\Sigma(z)$, for the measure $\nu_z$ of \cref{lem:reindex} with $\mathsf G=\mathcal A=\Lambda_i$ acting by deck transformations, $M=\R^k$, $\rho_M(\lambda)$ the translation by $\ell_i(\lambda)$, the chart $\xi_i'$ and the weight $\beta$. That lemma gives $U_i(\mu z)=U_i(z)+\ell_i(\mu)$.

\emph{Comparison of Summands.} Let $z\in\Omega_i$, and suppose that $\lambda,\mu\in\Lambda_i$ satisfy $d(\lambda^{-1}z,\hat w_i)<2r$ and $d(\mu^{-1}z,\hat w_i)<2r$. Set $z_\lambda:=\lambda^{-1}z$ and $g:=\lambda-\mu$, so that $\mu^{-1}z=gz_\lambda$. Both points lie in $B(\hat w_i,2r)\subset B(\hat y_i,3r)$. On $B(z,2r)$ we have
\begin{equation*}
u^\lambda-u^\mu=f\circ\lambda^{-1},\qquad f(p):=\xi_i'(p)-\xi_i'(gp)+\ell_i(g).
\end{equation*}
For $p\in B(z_\lambda,2r)$, both $p$ and $gp$ lie in $B(\hat y_i,5r)$; hence \eqref{eq:approxeq} gives $|f|\le4\Psi r$. The function $f$ is harmonic on $B(z_\lambda,2r)$ because it is the difference of two harmonic functions, one composed with a measure-preserving isometry. Jiang's gradient estimate \cite[Theorem~1.1]{Jia14} yields $|\nabla f|\le C\Psi$ on $B(z_\lambda,r)$. Therefore, on $B(z,r)$,
\begin{equation*}
|u^\lambda-u^\mu|\le4\Psi r,\qquad|\nabla u^\lambda-\nabla u^\mu|\le C\Psi .
\end{equation*}

(b), (c). Fix $z\in\Omega_i$ and choose $\lambda_0$ with $\varpi_{\lambda_0}(z)>0$. Put $B_z:=B(z,r/2)\cap\Omega_i$. Only those $\lambda$ with $d(\lambda^{-1}z,\hat w_i)<3r/2$ contribute on $B_z$, so the preceding comparison applies to every pair $(\lambda,\lambda_0)$. Moreover, $U_i$ is locally a finite sum of products of Lipschitz functions with bounded Laplacian, and $\sum_\lambda\nabla\varpi_\lambda=\sum_\lambda\Delta\varpi_\lambda=0$. Thus
\begin{align*}
U_i-u^{\lambda_0}&=\sum_\lambda\varpi_\lambda\,(u^\lambda-u^{\lambda_0}),\\
\nabla U_i-\nabla u^{\lambda_0}&=\sum_\lambda\varpi_\lambda(\nabla u^\lambda-\nabla u^{\lambda_0})+\sum_\lambda(u^\lambda-u^{\lambda_0})\nabla\varpi_\lambda,\\
\Delta U_i&=2\sum_\lambda\langle\nabla\varpi_\lambda,\nabla u^\lambda-\nabla u^{\lambda_0}\rangle+\sum_\lambda(u^\lambda-u^{\lambda_0})\Delta\varpi_\lambda ,
\end{align*}
where we used $\Delta u^\lambda=0$ and the product rule \cite[Proposition~4.28, Theorem~4.29]{Gig15}. The comparison estimates and \eqref{eq:weights} now give $|U_i-u^{\lambda_0}|\le4\Psi r$, $|\nabla U_i-\nabla u^{\lambda_0}|\le C\Psi$, and $|\Delta U_i|\le C\Psi/r$ almost everywhere on $B_z$. On $B(\hat w_i,r/8)$ we may choose $\lambda_0=0$, because $\beta=1$ on $B(\hat w_i,r/2)$; this proves (b). Finally, $|\nabla\xi_i'|\le C(N)$ gives the gradient bound in (c).
\end{proof}

Pairing the orbit coordinate with the transverse one completes the local chart. On the common domain
\begin{equation*}
\Omega_i^\Phi:=\Omega_i\cap\pi^{-1}(B(x_i^0,30r))
\end{equation*}
put $\mathcal Z_i:=(U_i,\hat\Phi_i)$, where $\hat\Phi_i:=\Phi_i\circ\pi$.

\begin{theorem}[Local Seifert Chart]\label{thm:chart}
For large $i$, the restriction $\mathcal Z_i|_{\tilde T_i}\colon\tilde T_i\to\R^k\times\R^m$ satisfies:
\begin{enumerate}[label=\textup{(\alph*)}]
\item $\mathcal Z_i$ is a local homeomorphism, and $\mathcal Z_i(\lambda z)=\bigl(U_i(z)+\ell_i(\lambda),\ \rho_0(\Pi_i\lambda)\hat\Phi_i(z)\bigr)$ for $\lambda\in\Lambda_i$;
\item $\mathcal Z_i$ descends to a homeomorphism $\bar{\mathcal Z}_i\colon E_i'\to\T_i^k\times \mathbb D$, $\T_i^k:=\R^k/L_i$, whose second component is $\Phi_i$;
\item $\bar{\mathcal Z}_i$ is $C_{w'}$-equivariant, where $c\in C_{w'}$ acts on $\T_i^k\times \mathbb D$ by $(a,d)\mapsto(a+\ell_i(\lambda),\rho_0(c)d)$ for any $\lambda\in\Pi_i^{-1}(c)$. The transverse action $\rho_0|_{C_{w'}}$ is faithful. The product action is free, and $C_{w'}$ acts on $\T_i^k$ freely by translations. Moreover $E_i'/C_{w'}$ is an open subset of $X_i$, and
\begin{equation*}
E_i'/C_{w'}\ \cong\ (\T_i^k\times \mathbb D)/C_{w'}\ \longrightarrow\ \mathbb D/C_{w'}
\end{equation*}
is a Seifert model, with $\Phi_i$ inducing the projection; since $\rho_0(C_{w'})$ fixes the centre $\Phi(w')$ of $\mathbb D$, translating $\mathbb D$ by $-\Phi(w')$ puts it in the form of \cref{def:seifert};
\item $\ell_i\colon H_i'\to L_i$ is an isomorphism; in particular $H_i'\cong\Z^k$.
\end{enumerate}
\end{theorem}

\begin{proof}
(a) \emph{Equivariance.} The first component is \cref{prop:orbit}(a). For the second, $\lambda\in\Lambda_i$ acts on $X_i'$ by the deck transformation $\chi_i(\lambda)=\Pi_i(\lambda)\in C_{w'}$, so $\hat\Phi_i(\lambda z)=\rho_0(\Pi_i\lambda)\hat\Phi_i(z)$ by \cref{prop:eqsplit}(a).

\emph{Local Homeomorphism.} Let $z_0\in B(\hat w_i,r_*/4)$. Since $r_*\le r/100$, we have
\begin{equation*}
B(z_0,r/25)\subset B(\hat w_i,r/8)\subset\Omega_i .
\end{equation*}
Moreover the projection of this ball lies in $B(x_i^0,30r)$ for large $i$, so $B(z_0,r/25)\subset\Omega_i^\Phi$, the domain of $\mathcal Z_i$. \Cref{prop:orbit}(b)--(c) gives the bounds $|U_i-\xi_i'|\le4\Psi r$, $\operatorname{Lip}U_i\le C(N)$ and $|\Delta U_i|\le C\Psi/r$ there. After lifting through the local isomorphism $\pi$, \cref{prop:eqsplit}(b) gives the analogous Lipschitz and Laplacian bounds for $\hat\Phi_i$. The estimate $|\Phi_i-\bar{\mathcal F}''\circ f_i'|\le2\Psi r$, together with the choice of $\bar{\mathcal F}''$ and the bound $|\xi_i-\mathcal F_i|\le\Psi r$, also gives
\begin{equation*}
 |\hat\Phi_i-\xi_i''|\le5\Psi r
 \quad\text{on }B(\hat y_i,5r).
\end{equation*}
Consequently, $\|\mathcal Z_i-\xi_i\|_{L^\infty(B(z_0,r/25))}\le C\Psi r$. The map $\xi_i$ is an $(N,\Psi)$-splitting map on $B(z_0,r/50)$. At that scale, \cref{lem:replacement} makes $\mathcal Z_i$ an $(N,C\Psi)$-splitting map and a bi-H\"older open embedding on $B(z_0,r/100)$; \cref{lem:rewinding}(b) supplies the required Reifenberg hypothesis. Because \mbox{$\mathcal Z_i\circ h=\mathcal Z_i+(\ell_i(h),0)$} for $h\in H_i'$, the same conclusion holds on every translate of this ball. \Cref{lem:tubecover}(d) says that these translates cover $\tilde T_i$. Hence $\mathcal Z_i$ is a local homeomorphism throughout $\tilde T_i$.

(b) \emph{Descent.} \cref{lem:tubecover}(b) identifies $\tilde T_i\to E_i'$ as a regular covering with deck group $H_i'$. By (a), since $\Pi_i(h)=0$, each $h\in H_i'$ acts on the target by the translation $(\ell_i(h),0)$. The kernel of $\ell_i$ is finite by \cref{prop:lattice}(b); using an evenly covered neighbourhood and its disjoint translates shows that the equivariant local homeomorphism $\mathcal Z_i$ remains a local homeomorphism after quotienting first by this kernel and then by $H_i'/\ker\ell_i\cong L_i$. We obtain a local homeomorphism $\bar{\mathcal Z}_i\colon E_i'\to\T_i^k\times\R^m$ with second component $\Phi_i$, and its image lies in $\T_i^k\times \mathbb D$. This map is proper as a map into $\T_i^k\times \mathbb D$, because $\Phi_i\colon E_i'\to \mathbb D$ is proper and $\T_i^k$ is compact. A proper local homeomorphism into a locally compact Hausdorff space has open and closed image. Since $\T_i^k\times \mathbb D$ is connected, $\bar{\mathcal Z}_i$ is onto and is therefore a finite covering map.

\emph{Degree One and Torsion.} By \cref{lem:tubecover}(a), $E_i'$ is connected. The connected regular covering $\tilde T_i\to E_i'$ has deck group $H_i'$, so its monodromy map $\theta\colon\pi_1(E_i')\to H_i'$ is onto. If a loop in $E_i'$ lifts to a path from $z$ to $hz$, applying $\mathcal Z_i$ shows that
\begin{equation*}
 \bar{\mathcal Z}_{i*}=\ell_i\circ\theta\colon
 \pi_1(E_i')\longrightarrow\pi_1(\T_i^k\times \mathbb D)=L_i.
\end{equation*}
This map is onto. A connected covering with surjective induced map on
fundamental groups has one sheet, so $\bar{\mathcal Z}_i$ is a
homeomorphism. Thus $\bar{\mathcal Z}_{i*}$ is an isomorphism. Since
$\theta$ is onto and $\ell_i\circ\theta$ is injective, both $\theta$ and
$\ell_i$ are injective. Consequently
$H_i'\cong L_i\cong\Z^k$, and in particular the kernel $\ker\ell_i=\Tor(H_i')$ in
\cref{prop:lattice}(b) vanishes. This proves (d).

(c) By (a), $\bar{\mathcal Z}_i(cz)=c\cdot\bar{\mathcal Z}_i(z)$ for the stated action. It is well defined because $\lambda$ is determined modulo $H_i'$ and $\ell_i(H_i')=L_i$. The representation $\rho_0$ is faithful on $C_0$ by \cref{prop:blowup}(b), hence also on $C_{w'}\le C_0$. The group $C_{w'}$ acts freely on $E_i'$, since $C$ acts freely on $X_i'$ (\cref{cor:Ccover}(c)); so it acts freely on $\T_i^k\times \mathbb D$ through $\bar{\mathcal Z}_i$. The centre $\Phi(w')$ of $\mathbb D$ is fixed by $\rho_0(C_{w'})$. So for $c\ne0$ the translation $\ell_i(\lambda)$ is not in $L_i$, and $C_{w'}$ acts freely on $\T_i^k$ by translations. Finally, $f_i'(E_i')\subset B_{W'}(w',r_*/5)$, the admissibility of $r_*$, and the almost $C$-equivariance of $f_i'$ give $cE_i\cap E_i'=\emptyset$ for $c\in C\setminus C_{w'}$. So $E_i'/C_{w'}\to X_i'/C=X_i$ is injective and open, and $\bar{\mathcal Z}_i$ induces the stated equivalence.
\end{proof}

\subsection{Proof of \texorpdfstring{\cref{thm:master}}{Theorem 1.2}(iv)}\label{ssec:proofiv}
It remains to assemble the charts of \cref{thm:chart} near an arbitrary point of $G$. \Cref{lem:unifcharts} makes the geometric choices uniform over all subsequential limits $\Phi$; a final diagonal argument then removes the subsequences.

\emph{Choices.} Let $x$, $y_0$, $w'$ and $r_*$ be as in \cref{ssec:fibrecoords}. Then $C_x\cong C_{w'}\le C_0$ (\cref{lem:setG}(b) and \cref{lem:localgroups}), and $\rho_0(C_{w'})$ fixes $\Phi(w')$ for every $\Phi\in\mathcal L$. By \cref{lem:unifcharts} and $\operatorname{Lip}\Phi\le C(N)$ (\cref{prop:eqsplit}(c)), there are radii $s_1>s_2>0$ and $t_1>t_0>0$ such that, for every $\Phi\in\mathcal L$, with $\mathbb D_j:=B(\Phi(w'),s_j)$ for $j=1,2$, the following three conditions hold:
\begin{equation*}
\bar{\mathbb D}_1\subset\Phi(B_{W'}(w_0',r/16)),\qquad \Phi^{-1}(\bar{\mathbb D}_1)\subset B_{W'}(w',r_*/10),
\end{equation*}
\begin{equation*}
B_{W'}(w',t_1)\subset\Phi^{-1}(\mathbb D_1)\quad\text{and}\quad B_{W'}(w',t_0)\subset\Phi^{-1}(\mathbb D_2)\subset B_{W'}(w',t_1).
\end{equation*}
Here the preimages are taken in the domain of $\Phi$; once the first condition holds, they lie in $B_{W'}(w_0',r/9)$ by \cref{lem:bundlecrit}, where \cref{lem:unifcharts} applies. Let $U_x$ and $U_x^0$ be the images in $X$ of $B_{W'}(w',t_1)$ and $B_{W'}(w',t_0)$. The first and third conditions and the injectivity of $\Phi$ on $B_{W'}(w_0',r/8)$ give $B_{W'}(w',t_1)\subset B_{W'}(w_0',r/16)$, so $U_x\subset G$ by the definition of $G$.

\emph{Along a Subsequence.} Let $\Phi_i\to\Phi\in\mathcal L$ along a subsequence, and let $\mathbb D_1$, $\mathbb D_2$ be as above for this $\Phi$. Then $\mathbb D:=\mathbb D_1$ satisfies the requirements of \cref{ssec:chart}. Put $\mathbb D_j':=\mathbb D_j-\Phi(w')$. By \cref{thm:chart}(c), for large $i$ in the subsequence, $E_i'(\mathbb D_1)/C_{w'}$ is an open subset of $X_i$, and $\Phi_i$ induces on it a map equivalent to the model
\begin{equation*}
(\T_i^k\times \mathbb D_1')/C_{w'}\longrightarrow \mathbb D_1'/C_{w'},
\end{equation*}
with $C_{w'}$ acting linearly on $\mathbb D_1'$ through $\rho_0$ and freely by translations on $\T_i^k$. Near any point $p\in \mathbb D_1'$ this model restricts to $(\T_i^k\times \mathbb D_3)/(C_{w'})_p\to \mathbb D_3/(C_{w'})_p$ for a small ball $\mathbb D_3$ centred at $p$. Here $(C_{w'})_p$ is the local group at $p$, and it still acts on $\T_i^k$ by translations. With $V=\Phi^{-1}(\mathbb D_1)$ as in \cref{ssec:chart}, put $\sigma_i':=\Phi^{-1}\circ\Phi_i\colon E_i'(\mathbb D_1)\to V$. Then $\sup d(\sigma_i',f_i')\to0$, since $\Phi_i=\Phi\circ f_i'+o(1)$ and $\Phi^{-1}$ is uniformly continuous on $\bar{\mathbb D}_1$. Also $(f_i')^{-1}(L)\subset E_i'(\mathbb D_1)$ for every compact $L\subset V$, by \cref{lem:bundlecrit}.

Let $\bar E_i\subset X_i$ be the image of $(\sigma_i')^{-1}(B_{W'}(w',t_1))$, and let $\bar\sigma_i\colon\bar E_i\to U_x$ be the induced map. Then $\bar\sigma_i$ is a Seifert $\T^k$-fibration over $U_x$, and the image of $\Phi^{-1}(\mathbb D_2)$, which contains $U_x^0$ and lies in $U_x$, is a model neighbourhood of $x$ with local group $C_{w'}\cong C_x$. Write $q\colon W'\to X$ for the quotient map. Both $f_i'$ and $f_i$ are induced by the approximations $(\hat f_i,\phi_i,\psi_i)$ of \cref{ssec:CQ}, so the convention of \cref{ssec:eqconv} gives
\begin{equation*}
\sup_{z\in X_i'}d_X\bigl(q(f_i'(z)),f_i(q_i(z))\bigr)\longrightarrow0.
\end{equation*} Hence $\sup d(\bar\sigma_i,f_i)\to0$. If $L\subset U_x$ is compact, its lift in $B_{W'}(w',t_1)$ has a compact neighbourhood still contained in that ball. The compatibility above and the last assertion of \cref{lem:bundlecrit} then give $f_i^{-1}(L)\subset\bar E_i$ for all large $i$.

\emph{Removing the Subsequences.} For $\delta\in(0,t_1)$, let $L_\delta$ be the image of $\bar B_{W'}(w',t_1-\delta)$ in $X$. Let $\mathcal P_i(\delta)$ assert the existence of a Seifert $\T^k$-fibration $\bar\sigma\colon E\to U_x$, with $E\subset X_i$ open, such that: $x$ has a model neighbourhood containing $U_x^0$ with local group $C_x$; $\sup_Ed(\bar\sigma,f_i)\le\delta$; and $f_i^{-1}(L_\delta)\subset E$. The preceding argument and \cref{prop:eqsplit}(c) show that every subsequence has a further subsequence on which $\mathcal P_i(\delta)$ holds eventually. Thus, for each fixed $\delta$, it holds for all sufficiently large $i$. Moreover, $\mathcal P_i(\delta)$ implies $\mathcal P_i(\delta')$ whenever $\delta'\ge\delta$. We may therefore choose $\delta_i\to0$ such that $\mathcal P_i(\delta_i)$ holds for all large $i$; the fibrations $\bar\sigma\colon E\to U_x$ it provides are the maps $\sigma_i\colon E_i\to U_x$ of \cref{thm:master}(iv). Every compact subset of $U_x$ lies in some $L_\delta$, so this proves (iv). Whenever $C_x=1$, in particular at regular points, the Seifert model reduces to an ordinary $\T^k$-bundle. \qed

\section{Globalization over Smooth Orbifolds}\label{sec:global}

Under no bubbling, $P(\mathcal R(Y))=P(Y)=X$, so $G=X$ (\cref{lem:setG}(a)) and the charts of \cref{sec:local} cover all of $X$; but they are still only charts. Turning them into a single map requires regularity of the base, and we assume it in a strong form: throughout this section $(\ast)$ and (R) hold, $m\ge1$ (the case $m=0$ is settled at the start of \cref{sec:local}), and $X$ is a smooth closed Riemannian orbifold. All orbifolds here are effective.

We first show that $W'$ is a smooth manifold, which provides a target with bounded geometry. \Cref{ssec:regularization} regularizes the quotient approximations $C$-equivariantly and identifies the fibre group of any bundle close to them. \Cref{ssec:Kproduct} builds a global orbit coordinate that is equivariant for the whole deck group, and \cref{sec:seifert} pairs it with the regularization: $X_i'$ is $C$-equivariantly a product $W'\times\T^k$, with $C$ acting on $\T^k$ by translations, and the global Seifert fibration is its quotient by $C$.

\subsection{Unfolding: The Manifold Cover of a Smooth Orbifold Limit}\label{sec:unfold}
The $C$-cover $W'$ is initially only a metric space. The next proposition shows that the smooth orbifold charts of $X$ lift to $W'$ and endow it with a compatible Riemannian structure.

\begin{proposition}[Unfolding]\label{prop:unfold}
Under the standing assumptions of this section, $W'$ is a closed smooth Riemannian manifold on which $C$ acts by isometries, $W'\to X=W'/C$ is a Riemannian orbifold covering, and, for $w'\in W'$ over $x$, the stabilizer $C_{w'}$ is conjugate in $O(m)$ to the orbifold group $\Gamma_x^{\mathrm{orb}}$.
\end{proposition}

\begin{proof}
By (R), freeness of the $H_0$-action, \cref{thm:master}(i), and \cref{lem:localgroups}, $C_y=\Stab_H(y)$ is finite and projects isomorphically onto $\Gamma_{[y]}$. \Cref{prop:blowup} shows that it fixes $H_0y$ pointwise and has faithful transverse representation $\rho_y$ in every blow-up. Since $W\to W'$ is a local isometry, we work on $W'$.

\emph{Step 1: Identify the Local Groups.} Fix $p\in W'$ over $x\in X$, write $\rho_p:=\rho_y$ for $y\in Y$ over $p$, and fix an orbifold chart $\pi_U\colon U\to U/\Gamma_x^{\mathrm{orb}}\cong V\ni x$. Since blow-up commutes with quotient (\cref{ssec:eqconv}, the case of the whole group),
$T_xX=\R^m/\rho_p(C_p)=\R^m/\Gamma_x^{\mathrm{orb}}$.
By \cite[Lemma~1]{Swa02}, the spherical quotient determines a finite subgroup of $O(m)$ up to conjugacy; hence $\rho_p(C_p)$ and $\Gamma_x^{\mathrm{orb}}$ are conjugate. If $x'\in V$ is near $x$, with lifts $p'\in W'$ and $u'\in U$, the metric on $V$ determines both stabilizer orders:
\begin{equation}\label{eq:order}
|C_{p'}|=|\Stab_{\Gamma_x^{\mathrm{orb}}}(u')|=\omega_m/\mathcal H^m\bigl(B_1(o)\subset T_{x'}V\bigr).
\end{equation}

\emph{Step 2: Linearize the Topological Chart.} For $w\in W$ over $p$, the stabilizer $\Gamma_w\le C$ equals $C_p$. \Cref{prop:U2} gives a $C_p$-equivariant open embedding $h_{B_0}$ of an invariant neighbourhood $B_0$ of $p$ into $\R^m$, sending $p$ to $0$ and intertwining $C_p$ with a conjugate of $\rho_p$. This step does not use smoothness of $X$.

\emph{Step 3: Identify the Two Charts.} We prove, by induction on $|C_p|$, that after shrinking $U$ there is an isometry $\Sigma$ from $U$ onto a neighbourhood $B$ of $p$ with $\pi_B\circ\Sigma=\pi_U$, where $\pi_B$ is the quotient map by $C_p$. This step produces $\Sigma$ as a homeomorphism; Step~4 shows that it is an isometry. Shrink $U$ to a small geodesic ball about the point $o_U$ over $x$, put $V:=\pi_U(U)$, and let $B$ be the component through $p$ of the preimage of $V$ in $W'$. For $U$ small, $B$ lies in $B_0$, the domain of $h_{B_0}$, and $\pi_B\colon B\to V$ is the quotient map by $C_p$, hence proper. If $C_p=1$, then $\Gamma_x^{\mathrm{orb}}=1$, and we put $\Sigma:=\pi_B^{-1}\circ\pi_U$.

Let $C_p\ne1$, and let $F_U$ and $F_B$ be the fixed sets of $\Gamma_x^{\mathrm{orb}}$ in $U$ and of $C_p$ in $B$. By \eqref{eq:order}, both are the preimages of the set $S\subset V$ of points whose orbifold group has order $|C_p|$, and each point of $S$ has exactly one preimage in $U$ and one in $B$. Points of $B\setminus F_B$ have smaller stabilizers. By the induction hypothesis, applied with the restrictions of the chart $U$, the map $B\setminus F_B\to V\setminus S$ is, near each point, isometric to a chart of the orbifold $V\setminus S$. So $B\setminus F_B$ is a regular covering orbifold of $V\setminus S$ with deck group $C_p$ \cite[Definition~13.2.2]{Thu80}, and so is $U\setminus F_U$, with deck group $\Gamma_x^{\mathrm{orb}}$. Let $c$ be the codimension of $F_U$; by Steps~1 and~2, $h_{B_0}(F_B)$ is a linear subspace of the same codimension. If $c\ge2$, then $B\setminus F_B$ is connected: it contains the connected $C_p$-invariant set $h_{B_0}^{-1}(B^m(0,\beta))\setminus F_B$ for small $\beta>0$, and $C_p$ permutes its components transitively because $V\setminus S$ is connected.

If $c\ge3$, then $U\setminus F_U$ is simply connected, hence the universal covering orbifold of $V\setminus S$ \cite[Section~13.2]{Thu80}. By \cite[Proposition~13.2.4]{Thu80} it covers $B\setminus F_B$ over $V\setminus S$; as the two coverings of $V\setminus S$ have the same number of sheets, this covering is a homeomorphism.

If $c=2$, the group acts on the normal plane of $F_U$ either by rotations, and then the orbifold fundamental group \cite[Definition~13.2.5]{Thu80} of $V\setminus S$ is $\Z$, or as the Klein four-group generated by two reflections, and then it is infinite dihedral. Both $U\setminus F_U$ and $B\setminus F_B$ are quotients of the universal covering orbifold by normal subgroups of index $|C_p|$, and in either case there is only one such subgroup.

If $c=1$, the group is generated by a reflection, and $V\setminus S$ is contractible. So $U\setminus F_U$ and $B\setminus F_B$ each consist of two components, each mapped homeomorphically onto $V\setminus S$.

In each case there is a homeomorphism $\Sigma^\circ\colon U\setminus F_U\to B\setminus F_B$ with $\pi_B\circ\Sigma^\circ=\pi_U$. As $\pi_U$ and $\pi_B$ are proper and each point of $S$ has a single preimage on either side, $\Sigma^\circ$ and its inverse extend continuously across $F_U$ and $F_B$. This gives $\Sigma$.

\emph{Step 4: Prove That $\Sigma$ Is an Isometry.} Restrict to half-radius balls so that the relevant geodesics remain in the domains of $\Sigma$ and $\Sigma^{-1}$. Let $\gamma$ be a geodesic in $B$ and set $\gamma':=\Sigma^{-1}\circ\gamma$. For $s$ sufficiently close to $t$,
\begin{equation*}
d_U(\gamma'(t),\gamma'(s))=d_V(\pi_U\gamma'(t),\pi_U\gamma'(s))\le d_B(\gamma(t),\gamma(s)).
\end{equation*}
The equality holds because, for $u_2$ sufficiently close to $u_1$, only elements fixing $u_1$ can realize the quotient distance $d_V=\min_gd_U(\cdot,g\,\cdot)$. Continuity of $\Sigma^{-1}$ then gives upper metric derivative at most $1$, so $\Sigma^{-1}$ is $1$-Lipschitz. The symmetric argument makes $\Sigma$ $1$-Lipschitz as well; hence it is an isometry. Compare \cite[Lemma~2.2]{Lan20}. Thus $W'$ is locally Riemannian, and Myers--Steenrod makes the charts smoothly compatible. By Steps~3 and~4, every point of $W'$ has an open neighbourhood isometric to an open subset of $\R^m$ with a smooth Riemannian metric, so $W'$ is a smooth Riemannian manifold without boundary. It is compact, because $X=W'/C$ is compact and $C$ is finite: if $w_j'\in W'$, then after passing to a subsequence their images converge in $X$ to the image of some $w'$, so by the definition of the quotient metric (\cref{lem:proper}(b)) there are $c_j\in C$ with $d(w_j',c_jw')\to0$, and along a further subsequence $c_j=c$ is constant and $w_j'\to cw'$. By \cref{lem:proper}(c), the group $C\cong H/\mathsf K$ acts by isometries on $W'=Y/\mathsf K$ (\cref{cor:Ccover}), and Myers--Steenrod makes them smooth. Since $\pi_B\circ\Sigma=\pi_U$, the map $W'\to X=W'/C$ is a Riemannian orbifold covering, with $C_p$ conjugate to $\Gamma_x^{\mathrm{orb}}$ in $O(m)$.
\end{proof}

\subsection{Equivariant Regularization and Fibre Identification}\label{ssec:regularization}
The base factor of the product is a $C$-equivariant regularization of the quotient approximations $f_i'\colon X_i'\to W'$. We use the equivariant regularization \cite[Theorem~1.11]{HHWZ26} in the following form, which its proof gives for a connected target. Let $\mathcal W$ be a closed connected Riemannian $m$-manifold, $1\le m\le N$, on which a finite group $C$ acts effectively by isometries. For every $\tau\in(0,1]$ there are $s=s(N,\mathcal W,C,\tau)\in(0,\tau]$ and $\epsilon_{\mathcal W}=\epsilon_{\mathcal W}(N,\mathcal W,C,\tau)>0$ with the following property. Let $(\mathcal X,d,\mathcal H^N)$ be a compact non-collapsed $\RCD(-1,N)$ space on which $C$ acts effectively by measure-preserving isometries, and let $f\colon\mathcal X\to\mathcal W$ be an $\epsilon$-GH approximation, $\epsilon\le\epsilon_{\mathcal W}$, with $d_{\mathcal W}(f(cx),cf(x))\le\epsilon$ for all $x\in\mathcal X$ and $c\in C$. Then there is a $C$-equivariant Lipschitz map $F\colon\mathcal X\to\mathcal W$ with $d_{\mathcal W}(F,f)\le\tau s$ such that, for every $x\in\mathcal X$ and normal coordinates $\Theta$ at $F(x)$, the map $F^\Theta:=\Theta\circ F$ is defined on $B(x,s)$, lies in the domain of the local Laplacian there, and satisfies
\begin{equation*}
\operatorname{Lip}F^\Theta\le C_\Delta,\qquad
|\Delta F^\Theta|\le\frac{C_\Delta}{s},\qquad
\fint_{B(x,s)}|\langle\nabla F^\Theta_a,\nabla F^\Theta_b\rangle-\delta_{ab}|\le\tau\quad(1\le a,b\le m),
\end{equation*}
where $C_\Delta=C_\Delta(N)$ depends only on $N$.

Indeed, an $\epsilon$-GH approximation is a $2\epsilon$-almost submetry at scale $1$ in the sense of \cite{HHWZ26}, so the hypotheses of \cite[Theorem~1.11]{HHWZ26} hold with $G=C$ acting on $\mathcal X$, $H=C$ acting on $\mathcal W$, $\phi=\mathrm{id}$ and their $\delta$ equal to $2\epsilon_{\mathcal W}$. Its proof, through \cite[Proposition~7.5]{HHWZ26}, rescales distances by $(2\epsilon_{\mathcal W})^{-1/(2N)}$. It gives the Lipschitz and Laplacian bounds above in one normal chart at $F(x)$, on balls of radius $s:=10(2\epsilon_{\mathcal W})^{1/(2N)}$, with Gram error $\Psi(\epsilon_{\mathcal W}\,|\,N,\mathcal W,C)$ and $d_{\mathcal W}(F,f)\le\Psi(\epsilon_{\mathcal W}\,|\,N,\mathcal W,C)\,s$. Other normal coordinates at $F(x)$ differ by an element of $O(m)$, which changes the bounds only by factors depending on $m$. The map $F$ satisfies $F\circ c=\psi(c)\circ F$ for a homomorphism $\psi\colon C\to C$, and $\psi=\mathrm{id}$: for $y\in\mathcal W$ choose $x$ with $d_{\mathcal W}(f(x),y)\le\epsilon$; then
\begin{equation*}
d_{\mathcal W}(\psi(c)y,cy)\le2d_{\mathcal W}(y,F(x))+d_{\mathcal W}(F(cx),cF(x))\le3\epsilon+4d_{\mathcal W}(F,f),
\end{equation*}
which is below $\min_{c'\ne c''}\max_yd_{\mathcal W}(c'y,c''y)>0$ once $\epsilon_{\mathcal W}$ is small, since $C$ is finite and acts effectively. Taking $\epsilon_{\mathcal W}$ small in terms of $\tau$ gives the form stated. We apply it with $\mathcal X=X_i'$ and $\mathcal W=W'$, using the equivariant convergence $(X_i',C)\to(W',C)$ of \cref{cor:Ccover}.

Comparing a bundle that is uniformly close to the approximations $f_i'$ with the local charts of \cref{sec:local} identifies the deck transformations carried by its fibres.

\begin{lemma}[Fibre Identification]\label{lem:fibreid}
Let $A_i\colon X_i'\to W'$ be fibre bundles with $d_{W'}(A_i,f_i')\to0$ uniformly on $X_i'$, and let $\theta_i\colon\pi_1(X_i')\to \mathsf K_i$ be the monodromy homomorphism of the regular cover $\hat X_i\to X_i'$. Then, for all large $i$, every fibre $F$ of $A_i$ is connected, and $\theta_i$ maps $\pi_1(F)$ onto $H_i'$.
\end{lemma}

\begin{proof}
Fix $x\in X=G$, choose $y_0,w',r,r_*$ as in \cref{ssec:fibrecoords}, and pass to a subsequence as in \cref{ssec:chart}. Choose a ball $\mathbb D_4$ centred at $\Phi(w')$ and satisfying the requirements of that section. Then choose successively a geodesic ball $V_3$ centred at $w'$, a concentric ball $\mathbb D_2\Subset \mathbb D_4$, and a geodesic ball $V_1\Subset V_3$, all below the relevant injectivity radii, so that $\Phi(\bar V_1)\subset \mathbb D_2$, $\Phi^{-1}(\bar{\mathbb D}_2)\subset V_3$ and $\Phi(\bar V_3)\subset \mathbb D_4$. Since $A_i$ and $f_i'$ are $o(1)$-close, $\Phi_i=\Phi\circ f_i'+o(1)$, and \cref{lem:unifcharts} applies, we have for large $i$ in the subsequence
\begin{equation*}
 A_i^{-1}(V_1)\subset E_i'(\mathbb D_2)\subset A_i^{-1}(V_3)\subset E_i'(\mathbb D_4).
\end{equation*}
Over the contractible balls $V_1\subset V_3$ the bundle is trivial, so the fibre $F$ over $w'$ and both preimages are homotopy equivalent through the inclusions. The inclusion $A_i^{-1}(V_1)\subset A_i^{-1}(V_3)$ is therefore a bijection on components, and it factors through $E_i'(\mathbb D_2)$, which is connected (\cref{lem:tubecover}(a)). Hence $A_i^{-1}(V_3)$, and with it $F$, is connected. By \cref{lem:tubecover}(a),(b), applied to $\mathbb D_2$ and to $\mathbb D_4$, $\theta_i$ maps $\pi_1(E_i'(\mathbb D_j))$ onto $H_i'$. The inclusions thus give $H_i'\subset\theta_i(\pi_1(A_i^{-1}(V_3)))\subset H_i'$, so $\theta_i(\pi_1(F))=H_i'$. As $W'$ is connected and $\mathsf K_i$ is abelian, the fibres over other points are freely homotopic to $F$ along paths, and the conclusion holds for every fibre. If it failed for infinitely many $i$, that subsequence would have a further subsequence on which the argument applies.
\end{proof}

\subsection{A Global Orbit Coordinate}\label{ssec:Kproduct}
The torus factor of the product is a \emph{global orbit coordinate}: a map $\hat U\colon\hat X_i\to\R^k$ that is exactly equivariant for the whole deck group $H_i$. Under (R), orbit charts like those of \cref{ssec:chart} exist at one scale at every point (\cref{lem:Kchart}), and two of them differ by an almost linear change of the orbit coordinate (\cref{lem:Kcompare}). A partition of unity on $W'$ therefore glues them, and averaging over $C$ makes the result equivariant for all of $H_i$ (\cref{prop:globalorbit}). Paired with the regularization of \cref{ssec:regularization}, the global orbit coordinate gives a $C$-equivariant local homeomorphism $X_i'\to W'\times\T^k$ (\cref{prop:globalchart}).

\emph{Setting.} Rescale $X_i'$ and $W'$ once, by a factor $\lambda_*$ that is at least $\max\{1,\sqrt{|K|}\}$, where $K$ is the lower Ricci bound in $(\ast)$, and normalize the measures; this does not affect the topological conclusions. Then $|K|\le1$, the measure of $X_i'$ is $\mathcal H^N$, and that of $\hat X_i$ is its lift. For large $i$, $(X_i',d,\mathcal H^N)$ is a compact non-collapsed $\RCD(-1,N)$ space on which $C$ acts by measure-preserving isometries (\cref{prop:chain} and (P3)), and the approximations $f_i'$ of \cref{cor:Ccover} are $\epsilon_i'$-GH approximations with equivariance error $\epsilon_i'$, for some $\epsilon_i'\to0$. From now on, lengths, including $D$, are measured in the rescaled metric. The rescaled sequence satisfies $(\ast)$ and (R), so the results of the earlier sections hold for it. By the opening of this section, $G=X$. Fix $x_0\in X$ with the data $y_0,w',r,r_*$ of \cref{ssec:fibrecoords}, and identify $H_0$ with $\R^k$ as in \cref{ssec:eqsplit}, so that $\varphi_v$, $v\in\R^k$, denotes the action of $H_0$. Let $\ell_i\colon H_i'\to\R^k$ be the homomorphisms of \cref{prop:lattice} for these data, and put $L_i:=\ell_i(H_i')$ and $\T_i^k:=\R^k/L_i$. Since $\Pi_i$ maps $\mathsf K_i=\Pi_i^{-1}(\Z^b)$ onto the free group $\Z^b$ with kernel $H_i'$, there are homomorphisms $s_i\colon\Z^b\to \mathsf K_i$ with $\Pi_i\circ s_i=\mathrm{id}$; put $\ell_i^{\mathsf K}(g):=\ell_i(g-s_i(\Pi_ig))$ for $g\in \mathsf K_i$. As $H_i/\mathsf K_i\cong C$ is finite and $\R^k$ is uniquely divisible, $\ell_i^{\mathsf K}$ extends uniquely to a homomorphism $\ell_i^H\colon H_i\to\R^k$, namely $\ell_i^H(h):=\ell_i^{\mathsf K}(|C|h)/|C|$. Write $\pi\colon\hat X_i\to X_i'=\hat X_i/\mathsf K_i$ for the projection. The constant $\mathsf M:=2\cdot10^4$ is the ratio between the scale of the orbit charts and that of the gluing in \cref{prop:globalorbit}: it keeps the gluing well inside each chart. Constants $C=C(N)$ depend only on $N$, and constants $C_{W'}$ also on $W'$ and on the group $C$.

\begin{lemma}[The Groups on the Covers]\label{lem:Kbasic}
For all large $i$:
\begin{enumerate}[label=\textup{(\alph*)}]
\item $\ell_i\colon H_i'\to L_i$ is an isomorphism and $L_i$ is a lattice in $\R^k$. The map $\ell_i^{\mathsf K}\colon \mathsf K_i\to L_i$ is a surjective homomorphism with $\ell_i^{\mathsf K}|_{H_i'}=\ell_i$.
\item For every $R$, $\sup\{|\ell_i(g)-\phi_i(g)|:g\in H_i',\ d(g\hat p_i,\hat p_i)\le R\}\to0$, where $\phi_i(g)\in H_0=\R^k$.
\item There is $r_{\mathrm f}>0$, independent of $i$, such that: if $\hat z\in\hat X_i$, $g\in \mathsf K_i$ and $d(g\hat z,\hat z)<r_{\mathrm f}$, then $g\in H_i'$.
\end{enumerate}
\end{lemma}

\begin{proof}
(a) By \cref{prop:lattice}(b), $\ker\ell_i=\Tor(H_i')$ and $L_i$ is a full lattice. Along every subsequence on which $\Phi_i$ converges (\cref{prop:eqsplit}(c)), \cref{thm:chart}(d) gives $\Tor(H_i')=0$ for large $i$; so $\ell_i$ is injective for all large $i$. For $g\in \mathsf K_i$, $g-s_i(\Pi_ig)\in\ker\Pi_i=H_i'$, so $\ell_i^{\mathsf K}$ is a homomorphism; it extends $\ell_i$, and its image is $L_i$.

(b) Displacements at $\hat p_i$ and at $\hat y_i$ differ by at most $2d(\hat y_i,\hat p_i)$, which is bounded, so (b) is \cref{prop:lattice}(a).

(c) If $b=0$, then $\mathsf K_i=H_i'$. Let $b>0$. The group $\Z^b\le\Gamma$ acts on $W$ by isometries, properly discontinuously, and freely, as every stabilizer lies in $C$ (\cref{lem:localgroups}). So $\mu(w):=\min\{d(\gamma w,w):\gamma\in\Z^b\setminus0\}$ is a positive, continuous, $\Z^b$-invariant function, and, as $W/\Z^b=W'$ is compact, $\mu_0:=\min_W\mu>0$. Put $r_{\mathrm f}:=\mu_0/2$. Suppose that (c) fails along a subsequence. Translating by $H_i$, which preserves displacements because $H_i$ is abelian, we obtain $\hat z_i$ with $d(\hat z_i,\hat p_i)\le D$ and $g_i\in \mathsf K_i\setminus H_i'$ with $d(g_i\hat z_i,\hat z_i)<r_{\mathrm f}$. By \cref{prop:CQ}(a), $\Pi_i(g_i)=\pr(\phi_i(g_i))\in\Z^b\setminus0$. By \ref{ax:E1}, \ref{ax:E2} and \cref{lem:proper}(a), after passing to a subsequence, $\phi_i(g_i)\to h\in H$ and $\hat f_i(\hat z_i)\to y$, with $\pr(h)=:\gamma\in\Z^b\setminus0$ and $d(hy,y)\le r_{\mathrm f}$. But $P_W$ is $1$-Lipschitz and $P_W(hy)=\gamma P_W(y)$, so $d(hy,y)\ge\mu_0=2r_{\mathrm f}$, a contradiction.
\end{proof}

Under (R) every point of $Y$ is regular, so the linear models of \cref{lem:linearmodel} exist at every point, and they pass to the covers at one scale.

\begin{lemma}[Uniform Orbit Charts]\label{lem:Kchart}
For every $\psi\in(0,1]$ there is $r_1=r_1(\psi)>0$ with the following property. Let $0<r\le r_1$. For all large $i$, depending on $\psi$ and $r$, every $\hat z\in\hat X_i$ has an \emph{orbit chart at scale $r$}: a linear isomorphism $A$ of $\R^k$, a harmonic map $\xi=(\xi',\xi'')\colon B(\hat z,40r)\to\R^k\times\R^m$, and a map $U\colon\Omega\to\R^k$ on $\Omega:=H_i'\cdot B(\hat z,r/4)$, such that, with $C=C(N)$:
\begin{enumerate}[label=\textup{(\alph*)}]
\item $|\nabla\xi|\le C$ on $B(\hat z,30r)$; $\xi$ is an $(N,\psi)$-splitting map on every ball $B(z,t)\subset B(\hat z,30r)$ with $t\ge r/\mathsf M$; and $\xi$ changes distances on $B(\hat z,20r)$ by at most $\psi r$;
\item if $g\in H_i'$ and $z,gz\in B(\hat z,5r)$, then $|\xi(gz)-\xi(z)-(A\ell_i(g),0)|\le\psi r$; and for every $t\in\R^k$ with $|t|\le r$ there is $g\in H_i'$ with $|A\ell_i(g)-t|\le\psi r$ and $d(g\hat z,\hat z)\le|t|+\psi r$;
\item $U(gz)=U(z)+A\ell_i(g)$ for $g\in H_i'$, and $U$ is locally Lipschitz and in the domain of the local Laplacian; every $z\in\Omega$ has $g_0\in H_i'$ with $d(g_0^{-1}z,\hat z)<r$ such that, with $u^{g_0}:=\xi'\circ g_0^{-1}+A\ell_i(g_0)$, almost everywhere on $B(z,r/2)\cap\Omega$,
\begin{equation*}
|U-u^{g_0}|\le\psi r,\qquad|\nabla(U-u^{g_0})|\le C\psi,\qquad|\Delta U|\le C\psi/r,\qquad|\nabla U|\le C .
\end{equation*}
\end{enumerate}
\end{lemma}

\begin{proof}
\emph{Uniform Linear Models on $Y$.} The threshold $r_{R,\varsigma}(y)$ of \cref{lem:linearmodel} can be taken independent of $y$. Indeed, let $S=S_y$ and the harmonic map $u$ on $B(y,16S)$ be the data of the harmonic-chart construction of \cref{sec:norot} at $y$. The proof of \cref{lem:linearmodel} at a point $z\in B(y,S/2)$ can be run with this same $u$. It uses only the closeness of balls fixed by the choice of $S$, the bi-H\"older bound of $u$, the Lipschitz modulus of \cref{prop:lip}, the cocycle estimate \eqref{eq:cocycleosc} with the dyadic argument of \cref{lem:pointlinear}, and the normalization by \cite[Proposition~3.13, Corollary~3.16]{BNS22} and \cite[Lemma~4.3]{HP23}; all of these hold on $B(y,S/2)$ with constants depending only on the data at $y$. So the conclusions of \cref{lem:linearmodel} hold at every $z\in B(y,S/2)$ for $r\le\bar r_y$, a threshold that depends only on the data at $y$. Finitely many such balls cover $\bar B(\hat p,D)$, whose $H$-translates cover $Y$, and $H$-translates preserve the conclusions because $H$ is abelian. Let $\bar r(R,\varsigma)>0$ be the least of these thresholds; it depends on $N$, $R$, $\varsigma$, $Y$ and the identification of $H_0$ with $\R^k$.

\emph{Models on the Covers.} Let $\varsigma\in(0,1/10)$, to be fixed in terms of $\psi$, and $r\le\bar r(200,\varsigma)$. We claim that for large $i$ every $\hat z\in\hat X_i$ has a map $\mathcal G\colon B(\hat z,200r)\to\R^N$ with $\mathcal G(\hat z)=0$, changing distances by at most $10\varsigma r$ and with $10\varsigma r$-dense image in $B^N(0,200r)$, and a linear isomorphism $A$ of $\R^k$, such that $|\mathcal G(gz)-\mathcal G(z)-(A\ell_i(g),0)|\le10\varsigma r$ whenever $g\in H_i'$ and $z,gz,g\hat z\in B(\hat z,40r)$, and such that the second assertion of~(b) holds with $10\varsigma r$ in place of $\psi r$ and $20r$ in place of $r$. If not, choose counterexamples $\hat z_i$ along a subsequence. Translating by $H_i$, which preserves these properties because $H_i$ is abelian, we may assume $d(\hat z_i,\hat p_i)\le D$ and $\hat f_i(\hat z_i)\to y$. Let $U_r,V_r,L_r$ be the model of \cref{lem:linearmodel} at $(y,r)$ for $R=200$, choose $O\in O(N)$ with $O(V_r)=\R^k\times0$, let $A$ be $OL_r$ read in $\R^k$, and put $\mathcal G_i:=O(U_r\circ\hat f_i-U_r(\hat f_i(\hat z_i)))$. Then the distortion and density of $\mathcal G_i$ follow from \cref{lem:linearmodel}(a) and \ref{ax:E1}. The equivariance follows from \cref{lem:linearmodel}(b) and \ref{ax:E2}, since \cref{prop:CQ}(c) and \cref{lem:Kbasic}(b) give $\phi_i(g)=\varphi_v$ with $|v-\ell_i(g)|=o(1)$. The elements $g$ are $\psi_i(\varphi_v)\in H_i'$ for $v:=L_r^{-1}O^{-1}(t,0)$ and $t$ in a finite net, by \cref{lem:linearmodel}(c), \cref{prop:CQ}(a), \eqref{eq:almosthom} and \ref{ax:E3}. This contradicts the choice of $\hat z_i$.

\emph{Charts.} Given $\mathcal G$ and $A$, construct $\xi$ and $U$ as in \cref{lem:orbitcharts} and \cref{prop:orbit}, with $H_i'$ in place of $\Lambda_i$, with $\hat z$ in place of $\hat y_i$ and $\hat w_i$, and with $\mathcal G$ in place of $\mathcal F_i$. So $\xi$ is a harmonic splitting map given by \cref{lem:eucsplit}(b) and aligned with $\mathcal G$ by \cref{lem:eucsplit}(c), and
\begin{equation*}
U:=\sum_{g\in H_i'}\varpi_g\bigl(\xi'\circ g^{-1}+A\ell_i(g)\bigr),\qquad\varpi_g(z):=\frac{\beta(g^{-1}z)}{\sum_{h\in H_i'}\beta(h^{-1}z)},
\end{equation*}
for a good cut-off function $\beta$ equal to $1$ on $B(\hat z,r/2)$ and supported in $B(\hat z,r)$ \cite[Lemma~3.1]{MN19}. By (R) and \cref{lem:rewinding}(a), once $r\le r_1$ every ball of radius at most $320r$ in $\hat X_i$ is close to Euclidean at its own scale, with any prescribed accuracy. The proofs of \cref{lem:orbitcharts} and \cref{prop:orbit} then apply verbatim, with the counting of \cref{lem:count}, and give (a)--(c) once $\varsigma$ and the accuracy are small in terms of $\psi$; the Gram bound on balls of radius $t\ge r/\mathsf M$ follows by doubling, and the distortion bound in~(a) and both parts of~(b) follow from those of $\mathcal G$ and $|\xi-\mathcal G|\le\psi r/3$. These requirements fix $\varsigma=\varsigma(\psi,N)$ and the accuracy $\eta'=\eta'(\psi,N)$, and we take $r_1(\psi):=\min\{\bar r(200,\varsigma),r_{\eta'}/32\}$, with $r_{\eta'}$ the radius of \cref{lem:rewinding}(a) for $\eta'$.
\end{proof}

Two orbit charts at nearby points are related by an almost orthogonal change of the orbit coordinate.

\begin{lemma}[Comparison of Orbit Charts]\label{lem:Kcompare}
There are $\bar\psi=\bar\psi(N)\in(0,1]$ and a function $\Psi_2=\Psi(\psi|N)$ such that, in the setting of \cref{lem:Kchart} with $\psi\le\bar\psi$ and for all large $i$, the following holds. Let orbit charts at scale $r$ at $\hat z_a,\hat z_b\in\hat X_i$ have data $A_j,\xi_j,U_j,\Omega_j$ for $j\in\{a,b\}$, and let $z\in\Omega_a\cap\Omega_b$. Then $M:=A_aA_b^{-1}$ satisfies $|M|\le2$, and
\begin{equation*}
|\nabla(U_a-MU_b)|\le\Psi_2\qquad\text{almost everywhere on }B(z,r/2)\cap\Omega_a\cap\Omega_b .
\end{equation*}
\end{lemma}

\begin{proof}
Let $g_j$ be the elements of \cref{lem:Kchart}(c) for $z$ and the chart~$j$, and put $\tilde\xi_j:=\xi_j\circ g_j^{-1}$. These maps are harmonic $(N,\psi)$-splitting maps on $B(z,4r)$ with $|\nabla\tilde\xi_j|\le C$, and $\tilde\xi_b$ changes distances there by at most $\psi r$. By (R) and \cref{lem:rewinding}(a), \cref{lem:eucsplit}(c) applies at $(z,4r)$, with $\tilde\xi_b-\tilde\xi_b(z)$ as the almost isometry; with Jiang's gradient estimate \cite[Theorem~1.1]{Jia14} it gives $O\in O(N)$ with
\begin{equation*}
|\tilde\xi_a-\tilde\xi_a(z)-O(\tilde\xi_b-\tilde\xi_b(z))|\le\Psi r\ \text{on }B(z,2r),\qquad|\nabla(\tilde\xi_a-O\tilde\xi_b)|\le\Psi\ \text{on }B(z,r).
\end{equation*}
Write $O=\begin{pmatrix}O_{11}&O_{12}\\O_{21}&O_{22}\end{pmatrix}$ in $\R^k\times\R^m$. For $g\in H_i'$ with $d(gz,z)<r/2$, \cref{lem:Kchart}(b), applied at $g_j^{-1}z$ (as $H_i'$ is abelian), gives $|(A_a\ell_i(g),0)-O(A_b\ell_i(g),0)|\le\Psi r$. Given $t\in\R^k$ with $|t|=r/8$, the second part of \cref{lem:Kchart}(b) at $\hat z_b$ gives $g$ with $|A_b\ell_i(g)-t|\le\psi r$; its first part and the distortion bound of~(a), at $g_b^{-1}z$, give $d(gz,z)\le r/8+3\psi r<r/2$. Since $M A_b\ell_i(g)=A_a\ell_i(g)$, it follows that $|Mt-O_{11}t|\le\Psi r(1+|M|)$ and $|O_{21}t|\le\Psi r$. As $|O_{11}|\le1$, this gives $|M|\le2$ for $\psi\le\bar\psi$, and then $|M-O_{11}|\le\Psi$ and $|O_{12}|=|O_{21}|\le\Psi$; the last equality holds because $O$ is orthogonal and $O_{11}$ is square. Finally, $|\nabla U_j-\nabla\tilde\xi_j'|\le C\psi$ by \cref{lem:Kchart}(c), and $\nabla\tilde\xi_a'=O_{11}\nabla\tilde\xi_b'+O_{12}\nabla\tilde\xi_b''+O(\Psi)$, so $|\nabla(U_a-MU_b)|\le\Psi$.
\end{proof}

\begin{proposition}[A Global Orbit Coordinate]\label{prop:globalorbit}
There are $r_{W'}>0$, depending only on $W'$, and, for $\psi\le\bar\psi(N)$, a function $\Psi_3=\Psi(\psi\,|\,N,W')$ with the following property. Let $\psi\le\bar\psi$, $L_F\ge0$ and $0<\rho\le r_{W'}$, with $r:=\mathsf M\rho\le\min\{r_1(\psi),r_{\mathrm f}/2\}$. Then, for all large $i$, depending on $\psi$ and $\rho$ but not on $L_F$, the following holds. Let $F\colon X_i'\to W'$ be continuous and $C$-equivariant, with $d_{W'}(F,f_i')\le\rho$. Suppose that, for every $x\in X_i'$ and normal coordinates $\Theta$ at $F(x)$, the map $F^\Theta:=\Theta\circ F$ is defined on $B(x,4\rho)$, lies in the domain of the local Laplacian there, and satisfies $\operatorname{Lip}F^\Theta\le C_\Delta$ and $|\Delta F^\Theta|\le L_F$, with $C_\Delta$ as in \cref{ssec:regularization}. Then there is a continuous map $\hat U\colon\hat X_i\to\R^k$ with $\hat U(hz)=\hat U(z)+\ell_i^H(h)$ for $h\in H_i$, such that every point of $\hat X_i$ is a $\mathsf K_i$-translate of a point $z$ with the following property. There are an orbit chart $(A,\xi,U)$ at scale $r$ at a point $\hat z$, with domain $\Omega\supset B(z,4\rho)$ and $B(z,4\rho)\subset B(\hat z,r/50)$, an open set $V\subset W'$ containing $F(\pi(B(z,4\rho)))$, and a smooth map $\alpha\colon V\to\R^k$, such that on $B(z,4\rho)$
\begin{equation*}
\hat U=A^{-1}(U+\tilde E)-\alpha\circ F\circ\pi .
\end{equation*}
Here $\tilde E\colon B(z,4\rho)\to\R^k$ is locally Lipschitz and in the domain of the local Laplacian, and almost everywhere $|\tilde E|\le\Psi_3\rho$, $|\nabla\tilde E|\le\Psi_3$ and $|\Delta\tilde E|\le\Psi_3(1/\rho+L_F)$.
\end{proposition}

\begin{proof}
\emph{A Partition of Unity on $W'$.} Put $\rho_V:=50\rho$. Let $\{w_a\}_{a\in I}$ be a maximal $\rho_V/2$-separated subset of $W'$, put $V_a:=B(w_a,\rho_V)$, and fix normal coordinates $\Theta_a$ at $w_a$. Normalizing the functions $\omega(|\Theta_a|/\rho_V)$, for a smooth $\omega\colon\R\to[0,1]$ equal to $1$ on $(-\infty,1/2]$ and to $0$ on $[3/4,\infty)$, gives a smooth partition of unity $\{\varpi_a\}$ with $\operatorname{supp}\varpi_a\subset\bar B(w_a,3\rho_V/4)$. For $r_{W'}$ small, at most $C_{W'}$ of the $V_a$ contain a given point, and in normal coordinates at any point of $W'$ the $\varpi_a$ have first and second derivatives at most $C_{W'}/\rho_V$ and $C_{W'}/\rho_V^2$.

\emph{Charts over the Pieces.} Let $i$ be large. For each $a$ with $F^{-1}(V_a)\ne\emptyset$, choose $x_a\in F^{-1}(V_a)$, a lift $\hat z_a$ and an orbit chart $(A_a,\xi_a,U_a^0)$ at scale $r$ at $\hat z_a$, with domain $\Omega_a$. As $f_i'$ changes distances by at most $\rho$ (\cref{cor:Ccover}(a)) and $d(F,f_i')\le\rho$, we have $F^{-1}(V_a)\subset B(x_a,103\rho)$, and $103\rho<r/100$. By \cref{lem:Kbasic}(c), applied at $\hat z_a$ to $g+h-h'$ for $h,h'\in H_i'$, the sets $g\Omega_a$ and $\Omega_a$ are at distance at least $r_{\mathrm f}-r/2\ge r$ for $g\in \mathsf K_i\setminus H_i'$. Hence, by \cref{lem:Kchart}(c) and \cref{lem:Kbasic}(a),
\begin{equation*}
U_a(gq):=A_a^{-1}U_a^0(q)+\ell_i^{\mathsf K}(g)\qquad(g\in \mathsf K_i,\ q\in\Omega_a)
\end{equation*}
is well defined on $\mathsf K_i\Omega_a=\pi^{-1}(B(x_a,r/4))$, and $U_a(gz)=U_a(z)+\ell_i^{\mathsf K}(g)$. On $g\Omega_a$, $A_aU_a$ is, up to a constant, the orbit coordinate of the transported chart at $g\hat z_a$, as $g$ is an isometry commuting with $H_i'$.

\emph{Overlaps.} The function $\mathcal D_{ab}:=U_a-U_b$ is $\mathsf K_i$-invariant on $\mathsf K_i\Omega_a\cap \mathsf K_i\Omega_b$, so it descends to a function $\bar{\mathcal D}_{ab}$ on $B(x_a,r/4)\cap B(x_b,r/4)$. \Cref{lem:Kcompare}, applied to transported charts, gives $|A_bA_a^{-1}|\le2$, $|\nabla(A_b\mathcal D_{ab})|\le\Psi_2$ and $|\nabla(A_a\mathcal D_{ab})|\le\Psi_2$ almost everywhere. Any two points of $F^{-1}(V_a\cap V_b)$ are joined by a geodesic of length less than $103\rho$ in $B(x_a,r/4)\cap B(x_b,r/4)$, so, by the Sobolev-to-Lipschitz property, $A_b\bar{\mathcal D}_{ab}$ and $A_a\bar{\mathcal D}_{ab}$ vary by at most $103\Psi_2\rho$ on it. Choose $y_{ab}=y_{ba}$ in this set when it is non-empty, and put $c_{ab}:=\bar{\mathcal D}_{ab}(y_{ab})$, and $c_{ab}:=0$ otherwise; so $c_{ba}=-c_{ab}$.

\emph{The Coordinate.} Put $\alpha_b:=\sum_d\varpi_dc_{bd}$ on $V_b$, and
\begin{equation*}
\hat U_0:=\sum_a(\varpi_a\circ F\circ\pi)\bigl(U_a-\alpha_a\circ F\circ\pi\bigr).
\end{equation*}
It is continuous. Since $F\circ\pi$ is $\mathsf K_i$-invariant, $U_a(gz)=U_a(z)+\ell_i^{\mathsf K}(g)$ and $\sum_a\varpi_a=1$ on $F(X_i')$, we have $\hat U_0(gz)=\hat U_0(z)+\ell_i^{\mathsf K}(g)$ for $g\in \mathsf K_i$.

\emph{The Local Form.} Every point of $\hat X_i$ is a $\mathsf K_i$-translate of a point $z\in B(\hat z_b,103\rho)$ for some $b$ with $\varpi_b(F(\pi z))>0$. Put $\hat z:=\hat z_b$, $A:=A_b$, $U:=U_b^0$, $V:=V_b$ and $\alpha:=\alpha_b$. Then $B(z,4\rho)\subset B(\hat z,r/50)\subset\Omega_b$, as $r/50=400\rho$, and $F(\pi(B(z,4\rho)))\subset B(F(\pi z),7\rho)\subset V$, as $f_i'$ changes distances by at most $\rho$, $d(F,f_i')\le\rho$ and $\operatorname{supp}\varpi_b\subset\bar B(w_b,3\rho_V/4)$. Using $\alpha_a-\alpha_b=\sum_d\varpi_d(c_{ad}-c_{bd})$ on $V_a\cap V_b$, we get on $B(z,4\rho)$
\begin{equation*}
\begin{gathered}
\hat U_0=A^{-1}(U+\tilde E)-\alpha\circ F\circ\pi,\qquad\tilde E:=\sum_{a,d}\varpi_{ad}B_{ad},\\
B_{ad}:=A\bigl(\mathcal D_{ab}-c_{ad}+c_{bd}\bigr),\qquad\varpi_{ad}:=(\varpi_a\varpi_d)\circ F\circ\pi .
\end{gathered}
\end{equation*}
Where $\varpi_{ad}\ne0$, the identity $\bar{\mathcal D}_{ab}=\bar{\mathcal D}_{ad}-\bar{\mathcal D}_{bd}$ and the variation bounds give $|B_{ad}|\le309\Psi_2\rho$. The gradient of $B_{ad}$ is that of $A_b\mathcal D_{ab}$, at most $\Psi_2$, and its Laplacian is a difference of Laplacians of orbit coordinates, at most $3C\psi/r$ by \cref{lem:Kchart}(c). For normal coordinates $\Theta$ at $F(\pi z)$, $\varpi_{ad}=((\varpi_a\varpi_d)\circ\Theta^{-1})\circ F^\Theta\circ\pi$, and $\pi$ is a local isomorphism of metric measure spaces. So the chain rule \cite[Proposition~4.28, Theorem~4.29]{Gig15} gives
\begin{equation*}
|\nabla\varpi_{ad}|\le C_{W'}/\rho_V,\qquad|\Delta\varpi_{ad}|\le C_{W'}(L_F/\rho_V+1/\rho_V^2).
\end{equation*} At most $C_{W'}$ pairs $(a,d)$ contribute at a point, and the product rule gives the bounds on $\tilde E$ with $\Psi_3:=C_{W'}(\Psi_2+\psi)$, as $\rho_V=50\rho$ and $r=\mathsf M\rho$.
\emph{Equivariance under $H_i$.} Write $\hat U_0$ for the map just constructed. The elements $\lambda_c=\psi_i(c)$ of \cref{lem:fibre} satisfy $\chi_i(\lambda_c)=c$ for $c\in C$, so they represent the cosets of $\mathsf K_i$; put
\begin{equation*}
\hat U:=\frac1{|C|}\sum_{c\in C}\bigl(\hat U_0\circ\lambda_c-\ell_i^H(\lambda_c)\bigr).
\end{equation*}
For $h\in H_i$ and $c\in C$ write $\lambda_c+h=g_c+\lambda_{c'}$, with $g_c\in \mathsf K_i$ and $c'=c+\chi_i(h)$. Then
\begin{equation*}
\hat U_0(\lambda_chz)-\ell_i^H(\lambda_c)=\hat U_0(\lambda_{c'}z)-\ell_i^H(\lambda_{c'})+\ell_i^H(h),
\end{equation*}
and $c'$ runs through $C$ with $c$. Summing over $c$ gives the asserted equivariance of $\hat U$.

For the local form, let $z$ be a point as above for $\hat U_0$, with data $\hat z,A,U,\Omega$. Each $\lambda_cz$ is a $\mathsf K_i$-translate $g_c'z_c$ of a point $z_c$ with such data $\hat z_c,A_c,U_c,\Omega_c,V_c,\alpha_c,\tilde E_c$, and we may take $z_0=z$. The isometry $T_c:=\lambda_c^{-1}g_c'$ commutes with $H_i'$ and maps $z_c$ to $z$. So $U_c':=U_c\circ T_c^{-1}$ is the orbit coordinate of an orbit chart at scale $r$ at $T_c\hat z_c$, with matrix $A_c$ and with $B(z,4\rho)$ in its domain. As $F\circ\pi$ is $\mathsf K_i$-invariant and $F$ is $C$-equivariant, $F\circ\pi\circ T_c^{-1}=c\circ F\circ\pi$. Hence, on $B(z,4\rho)$ and up to an additive constant,
\begin{equation*}
\hat U_0\circ\lambda_c-\ell_i^H(\lambda_c)=A_c^{-1}\bigl(U_c'+\tilde E_c\circ T_c^{-1}\bigr)-\alpha_c\circ c\circ F\circ\pi .
\end{equation*}
By \cref{lem:Kcompare}, $M_c:=AA_c^{-1}$ satisfies $|M_c|\le2$ and $|\nabla(M_cU_c'-U)|\le\Psi_2$ on $B(z,4\rho)$, and the Laplacian of $M_cU_c'-U$ is at most $3C\psi/r$ by \cref{lem:Kchart}(c). So $\hat U=A^{-1}(U+\tilde E)-\alpha\circ F\circ\pi$ on $B(z,4\rho)$, where $\alpha$ is $|C|^{-1}\sum_c\alpha_c\circ c$ up to a constant, smooth on the open set $V:=\bigcap_cc^{-1}V_c$, which contains $F(\pi(B(z,4\rho)))$, and
\begin{equation*}
\tilde E:=\frac1{|C|}\sum_{c\in C}\bigl(M_cU_c'-U+M_c\,\tilde E_c\circ T_c^{-1}\bigr)+\mathrm{const},
\end{equation*}
with the constant chosen so that $\tilde E(z)=0$. The bounds above and those of the previous step give the asserted bounds for $\tilde E$, with a larger constant $C_{W'}$ in $\Psi_3$; for $|\tilde E|$, integrate the gradient bound along geodesics from $z$.
\end{proof}

With the regularized base map of \cref{ssec:regularization}, read at a smaller scale, the global orbit coordinate completes it to a local product chart.

\begin{proposition}[The Product Chart]\label{prop:globalchart}
There is $\tau_0>0$ such that, for every $\tau\in(0,\tau_0]$ and all large $i$, depending on $\tau$, there are a $C$-equivariant continuous map $F\colon X_i'\to W'$ with $d_{W'}(F,f_i')\le\tau$ and a continuous map $\hat U\colon\hat X_i\to\R^k$ with $\hat U(hz)=\hat U(z)+\ell_i^H(h)$ for $h\in H_i$, such that the following maps are local homeomorphisms:
\begin{equation*}
\hat h:=(F\circ\pi,\hat U)\colon\hat X_i\to W'\times\R^k,\qquad \bar h:=(F,\bar U)\colon X_i'\to W'\times\T_i^k,\quad\bar U(\pi z):=\hat U(z)+L_i .
\end{equation*}
Moreover, $\bar h$ is $C$-equivariant when $c\in C$ acts on $W'\times\T_i^k$ by $(w,t)\mapsto(cw,t+\ell_i^H(\lambda))$ for any $\lambda\in\chi_i^{-1}(c)$.
\end{proposition}

\begin{proof}
The map $\bar U$ and the action of $C$ are well defined because $\ell_i^H(\mathsf K_i)=\ell_i^{\mathsf K}(\mathsf K_i)=L_i$ (\cref{lem:Kbasic}(a)).

\emph{Parameters.} They are fixed in the following order, and none depends on $i$. Let $C_\Delta=C_\Delta(N)$ be the constant of the regularization in \cref{ssec:regularization}. Let $L_Z:=2(C+C_\Delta+1)$, with $C$ the constant of \cref{lem:Kchart}, be the Lipschitz bound of the chart $Z$ below; let $\delta_R$ be the threshold of the bi-H\"older clause of \cref{lem:replacement} for $L=2L_Z$; and let $\delta>0$ be small compared with $\delta_R$, depending only on $N$. Let $\delta_e\le\delta$ be the threshold of \cref{lem:eucsplit} for the accuracy $\delta$ and $L=C_\Delta$. The ratio $\mu:=\min\{1/4,\delta_e/C_\Delta\}$ between the scale $\rho$ of the chart and the scale $s$ of the regularization makes the Laplacian of the base map small at scale $\rho$. Choose $\psi\le\min\{\bar\psi,\delta_e/\mathsf M\}$ with $C\psi+\Psi_3(\psi)(1+C_\Delta)\le\delta/C$. Finally choose $\tau_0\le\mu$ so small that $\mathsf M\mu\tau_0\le\min\{r_1(\psi),r_{\mathrm f}/2,\mathsf Mr_{W'}\}$, that $\mu\tau_0\le\delta_e^{1/2}$ and $\mu\tau_0$ is below the radius given by (R) and \cref{lem:rewinding}(a) for the accuracy $\delta_e$, and that $\tau_0\le\mu^N\delta_e/C$.

\emph{The Maps.} Let $\tau\le\tau_0$ and let $i$ be large. By the Setting, the regularization of \cref{ssec:regularization} applies to $X_i'$, $C$ and $f_i'$, and gives a $C$-equivariant $F$ with $d(F,f_i')\le\tau s\le\mu s$, where ${s=s(\tau)\le\tau}$. Put $\rho:=\mu s$ and $r:=\mathsf M\rho$. For $x\in X_i'$ and normal coordinates $\Theta$ at $F(x)$, the map $F^\Theta$ is defined on $B(x,s)\supset B(x,4\rho)$, $C_\Delta$-Lipschitz, with $|\Delta F^\Theta|\le C_\Delta/s=\mu C_\Delta/\rho=:L_F$ and Gram averages at most $\tau$ over $B(x,s)$. \Cref{prop:globalorbit} gives $\hat U$.

\emph{Reduction.} As $\hat h(gz)=\hat h(z)+(0,\ell_i^{\mathsf K}(g))$ for $g\in \mathsf K_i$, it suffices to show that $\hat h$ is a local homeomorphism at each point $z$ as in \cref{prop:globalorbit}. Fix such a $z$, with its data, and normal coordinates $\Theta$ at $F(\pi z)$. On $B(z,4\rho)$ put
\begin{equation*}
\mathcal U:=U+\tilde E,\qquad u:=F^\Theta\circ\pi,\qquad Z:=(\mathcal U,u)\colon B(z,4\rho)\to\R^k\times\R^m .
\end{equation*}
Then $\hat h=\Sigma\circ Z$ for the homeomorphism $\Sigma(v,w):=(\Theta^{-1}w,A^{-1}v-\alpha(\Theta^{-1}w))$ between open sets. So it suffices that $Z$ maps a neighbourhood of $z$ homeomorphically onto an open set. By \cref{lem:replacement}, applied at $z$ with radius $\rho':=5\rho/36$ and $\xi=F=Z$, it suffices that $Z$ is an $(N,\delta_R)$-splitting map on $B(z,\rho')$ with Lipschitz bound $L_Z$ and $|\Delta Z|\le\delta_R/(2\rho')$. Indeed, $\hat X_i$ is non-collapsed, $|K|\rho'^2\le\delta_R$, and every ball of radius at most $\rho'$ centred in $B(z,\rho')$ is $\delta_R$-close to Euclidean at its own scale.

\emph{The Two Blocks.} Let $g_0$ be the element of \cref{lem:Kchart}(c) for $z$, and put $\tilde\xi:=\xi\circ g_0^{-1}$; it is harmonic, $|\nabla\tilde\xi|\le C$, and it is $(N,\psi)$-splitting on every ball of radius at least $\rho=r/\mathsf M$ in $B(z,29r)$. By \cref{lem:Kchart}(c) and \cref{prop:globalorbit}, almost everywhere on $B(z,4\rho)$,
\begin{equation*}
|\nabla \mathcal U-\nabla\tilde\xi'|\le C\psi+\Psi_3,\qquad|\Delta \mathcal U|\le\frac{C\psi}r+\Psi_3\Bigl(\frac1\rho+L_F\Bigr)\le\frac\delta{C\rho},\qquad|\nabla \mathcal U|\le C+1 .
\end{equation*}
The base block $u$ is $C_\Delta$-Lipschitz with $|\Delta u|\le\mu C_\Delta/\rho\le\delta_e/\rho$. Let $\nu$ be the number of $g\in \mathsf K_i$ with $d(gz,z)<2\rho$. Each point of $B(\pi z,\rho)$ has at most $\nu$ preimages in $B(z,\rho)$ and at least $\nu$ in $B(z,3\rho)$. With the area formula for the local isometry $\pi$ and Bishop--Gromov, the Gram average of $u$ over $B(z,\rho)$ is therefore at most $C$ times that of $F^\Theta$ over $B(\pi z,\rho)$, hence at most $C\mu^{-N}\tau\le\delta_e$.

\emph{Alignment.} By \cref{lem:Kchart}(a), $\mathcal G_z:=\tilde\xi-\tilde\xi(z)$ changes distances on $B(z,\rho)$ by at most $\psi r\le\delta_e\rho$. So \cref{lem:eucsplit}(c), at $(z,\rho)$ with $j=m$, $n=N$ and $L=C_\Delta$, gives a linear map $A_F\colon\R^N\to\R^m$ with $A_FA_F^{\mathsf T}=I_m$ and $|u-u(z)-A_F\mathcal G_z|\le\delta\rho$ on $B(z,\rho/2)$. The first $k$ columns of $A_F$ are small: if $g\in H_i'$ and $d(gz,z)<\rho/4$, then $u(gz)=u(z)$, as $F\circ\pi$ is $\mathsf K_i$-invariant, while $\mathcal G_z(gz)$ is within $\psi r\le\delta_e\rho$ of $(A\ell_i(g),0)$ by \cref{lem:Kchart}(b) at $g_0^{-1}z$; and by the second part of \cref{lem:Kchart}(b), the vectors $A\ell_i(g)$ of such $g$ are $\delta_e\rho$-dense in the sphere of radius $\rho/8$. Hence these columns have norm at most $C\delta$.

\emph{The Gram Matrix.} The comparison clause of \cref{lem:harmonicpackage}, applied on $B(z,\rho/2)$ to $u-u(z)$ and $A_F\mathcal G_z$, gives $\fint_{B(z,5\rho/18)}|\nabla(u-A_F\tilde\xi)|^2\le C\delta^2$. Write
\begin{equation*}
\nabla \mathcal U_c=\nabla\tilde\xi_c+e_c,\qquad\nabla u_a=\sum_{e'}(A_F)_{ae'}\nabla\tilde\xi_{e'}+\nabla(u-A_F\tilde\xi)_a .
\end{equation*} The Gram averages of $\tilde\xi$ are at most $C\psi$, the coefficients $(A_F)_{ac}$ with $c\le k$ are at most $C\delta$, the rows of $A_F$ are orthonormal, and the Gram average of $u$ is at most $\delta_e$. With the Cauchy--Schwarz inequality this gives
\begin{equation*}
\sum_{a,b=1}^N\fint_{B(z,\rho')}|\langle\nabla Z_a,\nabla Z_b\rangle-\delta_{ab}|\le C\delta\le\delta_R ,
\end{equation*}
for $\delta$ small. Moreover $\operatorname{Lip}Z\le L_Z$ and $|\Delta Z|\le\delta/\rho\le\delta_R/(2\rho')$. So \cref{lem:replacement} makes $Z$ a bi-H\"older embedding of $B(z,\rho'/2)$ with open image, and $\hat h$ is a local homeomorphism at $z$.

\emph{Descent.} Let $\pi_T\colon W'\times\R^k\to W'\times\T_i^k$ be the quotient map. Since $\bar h\circ\pi=\pi_T\circ\hat h$, and $\pi$ and $\pi_T$ are covering maps, $\bar h$ is a local homeomorphism as well.

\emph{Equivariance.} If $x=\pi z$, $c\in C$ and $\lambda\in\chi_i^{-1}(c)$, then $cx=\pi(\lambda z)$, so
\begin{equation*}
\bar h(cx)=\bigl(cF(x),\bar U(x)+\ell_i^H(\lambda)\bigr).
\end{equation*}
\end{proof}

\subsection{The Product Structure and the Global Seifert Fibration}\label{sec:seifert}
A local homeomorphism from a compact connected space onto a connected Hausdorff space is a finite covering. The fibre identification of \cref{lem:fibreid} makes the covering of \cref{prop:globalchart} one-sheeted, and the resulting $C$-equivariant product gives the global Seifert fibration as its quotient by $C$.

\begin{theorem}[The Product Structure]\label{thm:virtualproduct}
For all large $i$ there are $C$-equivariant maps $A_i$ from $X_i'$ to $W'$, homomorphisms $\kappa_i$ from $C$ to $\T^k$, and homeomorphisms $h_i$ from $X_i'$ onto $W'\times\T^k$, with the following properties. First, $\sup_{X_i'}d_{W'}(A_i,f_i')\to0$, so the $A_i$ are $o(1)$-GH approximations. Second, $A_i$ is $h_i$ followed by the projection to $W'$. Third, $h_i$ is $C$-equivariant for the action
\begin{equation*}
c\cdot(w,t):=\bigl(cw,t+\kappa_i(c)\bigr)\qquad(c\in C)
\end{equation*}
on $W'\times\T^k$.
\end{theorem}

\begin{proof}
\emph{Step 1: Finite Coverings.} Let $\tau_0$ be as in \cref{prop:globalchart}, and choose $\tau_i\downarrow0$ so slowly that the proposition applies with $\tau=\tau_i$ for all large $i$. Let $A_i:=F$, $\hat U_i:=\hat U$ and $\bar h_i:=(A_i,\bar U_i)$ be the maps it gives. In the original metrics $\sup d(A_i,f_i')\le\tau_i/\lambda_*\to0$, so the $A_i$ are $o(1)$-GH approximations (\cref{cor:Ccover}(a)). The space $X_i'$ is compact, and connected as the quotient of the connected cover $\hat X_i$; the space $W'\times\T_i^k$ is connected and Hausdorff. So the local homeomorphism $\bar h_i$ is a finite covering, and $A_i=\mathrm{pr}_1\circ \bar h_i$ is a fibre bundle.

\emph{Step 2: One Sheet.} Fix $x\in X_i'$, and let $Q$ be the image of $\bar h_{i*}$ in $\pi_1(W',A_i(x))\times L_i$; the number of sheets is the index of $Q$. Let $\theta_i$ be the monodromy of $\pi$, as in \cref{lem:fibreid}. Along the lift of a loop $\gamma$ at $x$, $\hat U_i$ changes by $\pm\ell_i^{\mathsf K}(\theta_i(\gamma))$, with a sign fixed by the convention for $\theta_i$; so $\bar U_i\circ\gamma$ has class $\pm\ell_i^{\mathsf K}\theta_i(\gamma)$ in $\pi_1(\T_i^k)=L_i$. By \cref{lem:fibreid}, applied to the bundles $A_i$, for large $i$ the fibre $F_x$ of $A_i$ through $x$ is connected and $\theta_i$ maps $\pi_1(F_x,x)$ onto $H_i'$, which $\ell_i^{\mathsf K}=\ell_i$ maps onto $L_i$. As the fibres are connected, $A_{i*}$ is onto. So for $\alpha\in\pi_1(W',A_i(x))$ and $t\in L_i$ there are a loop $\gamma_0$ with $A_{i*}\gamma_0=\alpha$ and a loop $\beta$ in $F_x$ with $\bar h_{i*}(\gamma_0\beta)=(\alpha,t)$. Hence $Q$ is everything, and $\bar h_i$ is a homeomorphism.

\emph{Step 3: The Action of $C$.} A linear isomorphism of $\R^k$ that maps $L_i$ onto $\Z^k$ identifies $\T_i^k$ with $\T^k$. Let $h_i$ be $\bar h_i$ followed by this identification, and let $\kappa_i(c)$ be the image of $\ell_i^H(\lambda)$ for $\lambda\in\chi_i^{-1}(c)$. Then $\kappa_i$ is a homomorphism, and $h_i$ is equivariant by \cref{prop:globalchart}.
\end{proof}

The quotient of this product by $C$ is the global Seifert fibration, and translation of the torus factor descends to it.

\begin{corollary}[Global Seifert Fibration]\label{thm:seifert}
For large $i$, the maps $\bar A_i\colon X_i\to X=W'/C$ induced by the maps $A_i$ of \cref{thm:virtualproduct} are $o(1)$-GH approximations and Seifert $\T^k$-fibrations (\cref{def:seifert}). More precisely, let $w'\in W'$ have image $x$, and let $\mathbb D$ be a $C_{w'}$-invariant geodesic ball centred at $w'$, below the injectivity radius, with $c\bar{\mathbb D}\cap\bar{\mathbb D}=\emptyset$ for $c\in C\setminus C_{w'}$. Then
\begin{equation*}
\bar A_i^{-1}(\mathbb D/C_{w'})\cong(\mathbb D\times\T^k)/C_{w'},
\end{equation*}
with $\bar A_i$ corresponding to the projection, $C_{w'}$ acting faithfully and linearly on $\mathbb D$, and freely by translations on $\T^k$. Here $C_{w'}=C_x$ is conjugate in $O(m)$ to the Riemannian orbifold group $\Gamma_x^{\mathrm{orb}}$. Moreover, $\bar A_i$ is the orbit map of a continuous effective action of $\T^k$ on $X_i$, with isotropy group $\kappa_i(C_x)\cong C_x$ at points over $x$. If $X$ is a closed Riemannian manifold, then $C$ acts freely on $W'$, $W'\to X$ is a regular covering with deck group $C$, and $\bar A_i$ is a principal $\T^k$-bundle.
\end{corollary}

\begin{proof}
By \cref{thm:virtualproduct} and \cref{cor:Ccover}(c), $h_i$ induces a homeomorphism from $X_i=X_i'/C$ onto $(W'\times\T^k)/C$, under which $\bar A_i$ is induced by the projection; like $A_i$, it is an $o(1)$-GH approximation. For $\mathbb D$ as in the statement, the translates of $\mathbb D$ by $C\setminus C_{w'}$ are disjoint from $\mathbb D$, so the preimage of $\mathbb D/C_{w'}$ is $(\mathbb D\times\T^k)/C_{w'}$. The group $C_{w'}$ fixes $w'$ and acts by isometries, hence linearly in normal coordinates. It acts faithfully, as $C$ acts effectively on the connected manifold $W'$ (\cref{cor:Ccover}(b)), and by the translations $\kappa_i(c)$ on $\T^k$, freely because $C$ acts freely on $X_i'$ (\cref{cor:Ccover}(c)) and $C_{w'}$ fixes $w'$. Such balls exist at every point of $W'$, so every point of $X$ has a model neighbourhood, and $\bar A_i$ is a Seifert $\T^k$-fibration; the conjugacy is \cref{prop:unfold}. Translation of the second factor commutes with the action of $C$, so it descends to an action of $\T^k$ on $(W'\times\T^k)/C$. Its orbits are the fibres of $\bar A_i$, and the isotropy group at $[w',t]$ is $\kappa_i(C_{w'})$, which is isomorphic to $C_{w'}$ as the translations are free. If $X$ is a closed Riemannian manifold, every point of $X$ is regular, so $C_y=1$ for every $y\in Y$ (\cref{lem:blowupreg}). Hence $C$ acts freely on $W'$ (\cref{lem:localgroups}), $W'\to X$ is a regular covering with deck group $C$, and $\T^k$ acts freely on $X_i$, with the local trivializations $\mathbb D\times\T^k$ above.
\end{proof}

\subsection{Proof of \texorpdfstring{\cref{thm:master}}{Theorem 1.2}(v)}
\Cref{thm:virtualproduct,thm:seifert} supply part~\textup{(v)}. Its first assertion, $G=X$ under (R), was shown at the start of this section. Let $X$ be a smooth closed Riemannian orbifold, and put $\sigma_i:=\bar A_i$. Write $q\colon W'\to X$ and $q_i\colon X_i'\to X_i$ for the quotient maps by $C$. Both $f_i'$ and $f_i$ are induced by the approximations $(\hat f_i,\phi_i,\psi_i)$ of $\hat X_i\to Y$, so $\sup_{X_i'}d_X(q\circ f_i',f_i\circ q_i)\to0$ (\cref{ssec:eqconv}). With $\sup d(A_i,f_i')\to0$ and the surjectivity of $q_i$, this gives $\sup_{X_i}d_X(\sigma_i,f_i)\to0$. Let $x\in X$, choose $w'$ over $x$ and a ball $\mathbb D$ as in \cref{thm:seifert}, let $U_x$ and $U_x^0$ be the images in $X$ of $\mathbb D$ and of the concentric ball of half its radius, and put $E_i:=\sigma_i^{-1}(U_x)$. By \cref{thm:seifert}, $\sigma_i|_{E_i}$ is a Seifert $\T^k$-fibration over $U_x$, and $U_x$ is a model neighbourhood of $x$ containing $U_x^0$, with local group $C_x$ conjugate in $O(m)$ to $\Gamma_x^{\mathrm{orb}}$; near points with trivial local group, $\sigma_i$ is a $\T^k$-bundle. If $L\subset U_x$ is compact, then $f_i^{-1}(L)\subset E_i$ once $\sup d_X(\sigma_i,f_i)$ is less than the distance from $L$ to $X\setminus U_x$. So the local fibrations of part~(iv) may be taken to be restrictions of $\sigma_i$. The product and the torus action are \cref{thm:virtualproduct,thm:seifert}, with $\widetilde X_i:=X_i'$, a regular covering of $X_i$ with deck group $C$ (\cref{cor:Ccover}(c)), and $\widetilde W:=W'$, with $X=W'/C$ as orbifolds (\cref{prop:unfold}). If $X$ is a closed Riemannian manifold, then every point of $X$ is regular, so (R) holds and every $C_x$ is trivial by \cref{lem:blowupreg}, and the last assertions are those of \cref{thm:seifert}. \qed

\section{Manifold Limits and Examples}\label{sec:consequences}

This section collects the manifold-limit corollary and examples showing the sharpness of the hypotheses.

\begin{corollary}[Manifold Limits]\label{cor:conj}
Let $(X_i,d_i,\mathfrak m_i)$ be compact $\RCD(K,N)$ spaces converging in the measured Gromov--Hausdorff sense to a closed $m$-dimensional Riemannian manifold $X$, with $b_1(X_i)-b_1(X)=N-m$ for all $i$. Then for large $i$ there are continuous $\epsilon_i$-GH approximations $X_i\to X$, $\epsilon_i\to0$, that are principal $\T^{N-m}$-bundles. Moreover, the bundles may be chosen so that, for large $i$, there are regular finite coverings $\widetilde X_i\to X_i$ and $\widetilde W_i\to X$ with the same abelian deck group, and a homeomorphism
\begin{equation*}
\widetilde X_i\cong\widetilde W_i\times\T^{N-m}
\end{equation*}
whose first component is an equivariant lift of the bundle. This confirms \cite[Conjecture~2.9]{ZZ26} and gives its virtual-product refinement.
\end{corollary}

\begin{proof}
Let $k:=N-m$, which is an integer.

If $k\ge1$ and $N\ge2$, apply \cref{thm:master}(v). Its subsequential formulation implies both full-sequence conclusions by the usual contradiction and sub-subsequence argument.

If $k\ge1$ and $N<2$, then $(N,m)=(1,0)$ and $b_1(X_i)=1$. The classification of $\RCD^*(K,1)$ spaces in \cite[Proposition~2.5]{BS10} and \cite{KL16} makes every compact such $X_i$ a circle; the constant map to $X$ is the required $\T^1$-bundle, already a product.

Let $k=0$. For $N=m=1$, rescaling arc length gives the claim. If $N=m\ge2$, then (P1) and \cite{BGHZ23} make $(X_i,d_i,\mathcal H^N)$ non-collapsed for large $i$. The collapsing alternative in \cite{DPG18} would force $\dim_{\mathcal H}X\le N-1$, so volume convergence holds. The compact case of topological stability \cite[Theorem~3.3]{KM21} then directly gives bi-H\"older homeomorphisms $X_i\to X$ that are $o(1)$-GH approximations. In either case take the identity covers and use $\T^0=\{*\}$.

\end{proof}

The next three items locate the sharp points of the Main Theorem. The first shows that maximal drop, rather than collapse alone, is what detects a torus direction, and that exceptional fibres do occur. The second shows that maximal drop is also what forces $\dim_{\mathcal H}X=m$. The third collects the remaining sharpness phenomena: bubbling, a non-open regular set, boundary points, sharpness of the codimension bound, and dependence on the subsequence.

\begin{example}[The Klein Bottle, Collapsed in Two Ways]\label{ex:klein}
For $a,\ell>0$ let $\mathbb K_{a,\ell}:=\R^2/\langle t,\tau\rangle$, where $t(x,y)=(x+a,y)$ and $\tau(x,y)=(-x,y+\ell/2)$. Since $\tau t\tau^{-1}=t^{-1}$,
\begin{equation*}
H_1(\mathbb K_{a,\ell};\Z)=\Z[\tau]\oplus\Z_2[t].
\end{equation*}
There are two fibrations. The map $(x,y)\mapsto y\bmod \ell/2$ is a circle bundle over a circle, with fibre class $[t]$ and reflection monodromy, and $(x,y)\mapsto d(x,a\Z)$ is a Seifert fibration over $[0,a/2]$, with generic fibre class $2[\tau]$ and two exceptional endpoint fibres.
Both collapses below have $N=2$ and $m=1$:
\begin{enumerate}[label=\textup{(\Alph*)}]
\item If $a\to0$ with $\ell$ fixed, the limit is a circle but the collapsing class $[t]$ is torsion. Thus $b_1(\mathbb K_{a,\ell})-b_1(S^1)=0<N-m$, and even the maximal abelian covers $\R/2a\Z\times\R$ collapse.
\item If $\ell\to0$ with $a$ fixed, the free class $[\tau]$ dies and the limit is $X=[0,a/2]$, so the drop is maximal. Here (R) holds and $G=X$. The endpoints are mirror points with local group $\Z_2$, and the exceptional fibres have M\"obius-band neighbourhoods, locally modelled by $((-1,1)\times\T^1)/\Z_2$.
In the notation of the construction, this case is completely explicit: $Y=(\R/2a\Z)\times\R$, $H_i'=\langle2[\tau]\rangle$, $C\cong\Z_2^2$, and $X_i'=\hat X_i/H_i'\cong(\R/2a\Z)\times(\R/\ell\Z)=W'\times\T^1$, the fourfold cover of $\mathbb K_{a,\ell}$.
\end{enumerate}
Thus maximal Betti drop detects which homology class collapses. For the Klein bottle, maximal drop can occur only in the interval collapse: maximality forces $m=1$ and $b_1(X)=0$, while a compact one-dimensional $\RCD$ space with vanishing first Betti number is an interval \cite{KL16}.
\end{example}

\begin{example}[The Compact Pan--Wei Double]\label{ex:panwei}
Fix $\alpha>\frac12$, and choose the integer $p\ge2$ sufficiently large as in Pan--Wei \cite[Theorem~A and Remark~1.8]{PW22}. Their quotient--collar--double construction produces closed $(p+1)$-manifolds $M_i$ with uniform lower Ricci and upper diameter bounds, collapsing to a compact metric-measure space $X_\alpha$. Dai--Honda--Pan--Wei give an explicit model for the limit \cite[Section~4.3 and Remark~4.4]{DHPW23}. Begin with the finite Grushin cylinder
\begin{equation*}
 [0,3]\times S^1,
 \qquad ds^2=dr^2+\widetilde h(r)^2\,dv^2,
 \qquad \widetilde h(r)=r^{-2\alpha}\quad(0<r\le1),
\end{equation*}
make $\widetilde h$ constant near $r=3$, and glue two copies along their smooth $r=3$ ends. With its normalized limit measure, $X_\alpha$ is an $\RCD(K,p+1)$ space for some $K<0$. Its rectifiable dimension is $2$, whereas
\begin{equation*}
 \dim_{\mathcal H}X_\alpha=1+2\alpha>2.
\end{equation*}
In particular, its Hausdorff dimension can be non-integral.

The topology makes the Betti-number comparison transparent. The cyclic quotient is diffeomorphic to $\R^p\times S^1$, so its bounded radial core is $D^p\times S^1$, also after the collar modification that makes the boundary totally geodesic; $M_i$, the double of this core, is $S^p\times S^1$. Likewise $X_\alpha$, the double of $[0,3]\times S^1$ along $\{3\}\times S^1$, is homeomorphic to an annulus whose two boundary circles are the singular ends $r=0$. Hence
\begin{equation*}
 b_1(M_i)=b_1(X_\alpha)=1,\qquad b_1(M_i)-b_1(X_\alpha)=0.
\end{equation*}
Here ``boundary'' means topological boundary. Since $N=p+1$ and the rectifiable dimension is $m=2$, maximal drop would require $N-m=p-1>0$. The example therefore lies outside the equality case of the Main Theorem. It shows that maximal Betti drop is essential for the conclusion $\dim_{\mathcal H}X=m$, even when every approximating space is a closed smooth manifold.
\end{example}

\begin{remark}[Examples and Sharpness]\label{rem:sharp}
\emph{Bubbling.} For $X_i=\T^k_{1/i}\times\mathrm{K3}_i$, where $\T^k_{1/i}$ is the flat torus $\R^k/i^{-1}\Z^k$ and the Kummer metrics on $\mathrm{K3}_i$ converge without collapsing to $\T^4/\{\pm1\}$, we have $Y=\R^k\times\T^4/\{\pm1\}$, while $G$ omits the $16$ singular points. Part~(iv) makes no assertion there.

\emph{$\mathcal R(Y)$ Need Not Be Open.} If $\Sigma\subset\R^3$ is a convex surface whose vertices accumulate at a smooth point and $X_i=\T^k_{1/i}\times\Sigma$, then $Y=\R^k\times\Sigma$ and its regular set is not open.

\emph{Boundary and Mirrors.} The flat cylinders $X_i=[0,1]\times S^1_{1/i}$ have the same limit as collapse~(B) of \cref{ex:klein}, but for the cylinders $G=(0,1)$, whereas for the Klein bottles $G=X$ and the endpoints are mirror points. Thus $G$ is not determined by $X$, and the boundary hypothesis in part~(iii) is necessary: a model $(\mathbb D\times\T^k)/C_u$ has no boundary, while the endpoint fibres of the cylinders lie in $\partial X_i$.

\emph{Codimension Two.} For $m\ge2$, let $\Sigma$ be the boundary of a regular tetrahedron and $X_i:=\T^k_{1/i}\times\Sigma\times\T^{m-2}$. Then $X\setminus G$ is the four cone points times $\T^{m-2}$, so the estimate in part~(iii) is sharp. Yet $X_i$ is a trivial torus bundle over $X$: the complement of $G$ records where this method stops, not necessarily where every fibration fails.

\emph{Dependence on the Subsequence.} Use the bubbling sequence above, with index $j$, as the even terms, and let
\begin{equation*}
X_{2j+1}:=(\T^4\times\T^k_{1/j})/\Z_2,
\qquad (x,y)\longmapsto(-x,y+e_1/(2j)).
\end{equation*}
The sequence satisfies $(\ast)$ with $Q:=\T^4/\{\pm1\}$. Along the odd terms (R) holds and $G=Q$; along the even terms $G\ne Q$, and no even term is a Seifert $\T^k$-fibration over the effective orbifold $Q$ with the prescribed orbifold local groups in the sense of (v). Indeed, if $Q_{\rm reg}$ denotes the complement of its $16$ orbifold points, then
\begin{equation*}
 \R^4\setminus\tfrac12\Z^4\longrightarrow Q_{\rm reg}
\end{equation*}
is the universal cover and
$\pi_1(Q_{\rm reg})=\Z^4\rtimes_{-1}\Z_2$, which is non-abelian. Removing
the corresponding codimension-$4$ fibre preimages from an even total space
does not change its fundamental group, by general position. A bundle over
$Q_{\rm reg}$ would therefore make the non-abelian group
$\pi_1(Q_{\rm reg})$ a quotient of the abelian group
$\pi_1(X_{2j})=\Z^k$, a contradiction.

\end{remark}


\appendix
\crefalias{section}{appendix}
\crefalias{subsection}{subappendix}
\section{Equivariant Comparison over Local Groups}\label{app:eqcm}

This appendix records an averaging over the elements of small displacement in a local group; the main text does not use it. \Cref{lem:reindex} gives exact equivariance as soon as the weights vanish where the averaging set $\mathcal A$ and its translate $\mathcal Ag_0$ differ. In gluing arguments for collapsing spaces with nilpotent local symmetry, the transfer between two charts is a homomorphism of Lie groups that need not be injective, and almost equivariance is available only for elements of bounded displacement. Averaging over a local piece $\mathcal A$ of the group, with weights evaluated at the moved point, satisfies that condition for the elements of small displacement (\cref{eqcm:prop:comparison}), and the same averaging produces invariant metrics (\cref{eqcm:prop:metric}). The estimates use the centre of mass with parameters (\cref{eqcm:lem:cm}) and the doubling estimate for the weights (\cref{lem:orbitcount}).

\subsection{The Centre of Mass}\label{eqcm:sec:cm}
The averages in this appendix are Riemannian centres of mass of measures that depend on a parameter. The next lemma gives them with $C^m$ estimates in which the weights enter only through the total variation of their derivatives, which \cref{lem:orbitcount} controls.

\begin{lemma}[Centre of Mass]\label{eqcm:lem:cm}
Let $(M,\mathsf g)$ be a Riemannian manifold, not necessarily complete, $Q\subset M$ compact and $m\ge1$.
\begin{enumerate}[label=\textup{(\alph*)}]
\item There is $r_{\mathrm{cm}}>0$ such that for every $r_0\in(0,r_{\mathrm{cm}}]$ the following holds. Let $\varphi$ be an isometry of $(M,\mathsf g)$, $q\in\varphi(Q)$ and $r\le r_0$. Every Borel probability measure $\nu$ supported in $B_{\mathsf g}(q,r)$ has a \emph{centre of mass} $\operatorname{cm}(\nu)\in B_{\mathsf g}(q,r)$: the unique minimizer in $B_{\mathsf g}(q,3r_0)$ of $y\mapsto\int d_{\mathsf g}(y,q')^2\,d\nu(q')$, and the unique zero there of $y\mapsto\int\exp_y^{-1}(q')\,d\nu(q')$. It does not depend on $(\varphi,q,r,r_0)$, and $\operatorname{cm}(\varphi'_*\nu)=\varphi'(\operatorname{cm}(\nu))$ for every isometry $\varphi'$.
\item Let $\kappa\colon O\to\R^d$ be a chart with $Q\subset O$. For every $A\ge1$ there is $C_{\mathrm{cm}}\ge1$, depending only on $(M,\mathsf g,\kappa,Q,m,A)$, with the following property. Let $\zeta_0$ be a point of an open set $\mathcal P\subset\R^{d'}$, let $\mathcal A$ be a compact space with a finite Borel measure $\mu$, and let $\varpi\colon\mathcal A\times\mathcal P\to[0,\infty)$ and $\xi\colon\mathcal A\times\mathcal P\to O$ be continuous maps, of class $C^m$ in $\zeta$ with $\zeta$-derivatives up to order $m$ continuous on $\mathcal A\times\mathcal P$, such that $\int_{\mathcal A}\varpi(g,\zeta)\,d\mu(g)=1$. Let $\xi_0\colon\mathcal P\to O$ be of class $C^m$ with $\xi_0(\zeta_0)\in Q$. Suppose that, for all $g\in\mathcal A$ and $1\le j\le m$,
\begin{equation*}
 \bigl|\kappa(\xi(g,\zeta_0))-\kappa(\xi_0(\zeta_0))\bigr|\le\eta,\qquad\bigl\|D^j_\zeta(\kappa\circ\xi)(g,\zeta_0)-D^j(\kappa\circ\xi_0)(\zeta_0)\bigr\|\le\eta,
\end{equation*}
\begin{equation*}
 \int_{\mathcal A}\bigl|D^j_\zeta \varpi(g,\zeta_0)\bigr|\,d\mu(g)\le A,\qquad\bigl\|D^j(\kappa\circ\xi_0)(\zeta_0)\bigr\|\le A,
\end{equation*}
and that $C_{\mathrm{cm}}\eta\le1$. Then $\bar\xi(\zeta):=\operatorname{cm}\bigl(\xi(\cdot,\zeta)_*(\varpi(\cdot,\zeta)\mu)\bigr)$ is defined and of class $C^m$ near $\zeta_0$, and
\begin{equation*}
 \bigl|\kappa(\bar\xi(\zeta_0))-\kappa(\xi_0(\zeta_0))\bigr|\le C_{\mathrm{cm}}\eta,\qquad\bigl\|D^j(\kappa\circ\bar\xi)(\zeta_0)-D^j(\kappa\circ\xi_0)(\zeta_0)\bigr\|\le C_{\mathrm{cm}}\eta
\end{equation*}
for $1\le j\le m$.
\end{enumerate}
If the components of $\mathsf g$ and $\mathsf g^{-1}$ in the chart $\kappa$ are bounded by $\Lambda$ in $C^{m+3}(\kappa(O))$ and $\kappa(Q)$ has distance at least $\delta>0$ from $\partial\kappa(O)$, then $r_{\mathrm{cm}}$ and $C_{\mathrm{cm}}$ can be chosen depending only on $(d,m,A,\Lambda,\delta)$.
\end{lemma}

\begin{proof}
(a) Choose $r_{\mathrm{cm}}\le1$ so small that $Q':=\{y:d_{\mathsf g}(y,Q)\le4r_{\mathrm{cm}}\}$ is compact and that every ball $B_{\mathsf g}(q,r')$ with $q\in Q$ and $r'\le4r_{\mathrm{cm}}$ is strongly convex and has radius below $\pi/(4\sqrt\Delta)$, where $\Delta>0$ bounds the sectional curvature on $Q'$ from above. By \cite[Theorem~1.2]{Kar77}, see also \cite[Section~8.1]{BK81}, if $\nu$ is supported in such a ball $B$, the function $y\mapsto\tfrac12\int d_{\mathsf g}(y,q')^2\,d\nu(q')$ has gradient $-\int\exp_y^{-1}(q')\,d\nu(q')$ on $B$ and a unique critical point there, its minimizer. Let $r_0\le r_{\mathrm{cm}}$ and first $\varphi=\mathrm{id}$. Applying this to $B=B_{\mathsf g}(q,r)$ and to $B=B_{\mathsf g}(q,3r_0)$ shows that the critical point in the smaller ball is the one in the larger. If the support also lies in $B_{\mathsf g}(q',r')$ with $q'\in Q$ and $r'\le r_0'\le r_{\mathrm{cm}}$, then, the support being non-empty, $B_{\mathsf g}(q',r')$ and $B_{\mathsf g}(q,r)$ lie in $B_{\mathsf g}(q,3\max\{r_0,r_0'\})$, where the critical point is unique; so the centres of mass agree. The properties defining $r_{\mathrm{cm}}$ are invariant under isometries, which gives the general case, with two different isometries allowed in the independence statement. The last assertion holds because the minimized function is transported by $\varphi'$.

(b) Decrease $r_{\mathrm{cm}}$ so that $Q'\subset O$; by (a) this does not change centres of mass. Identify points of $O$ with their coordinates. On the compact set
\begin{equation*}
\mathcal D:=\{(y,q):y,q\in\bar B_{\mathsf g}(q_0,2r_{\mathrm{cm}})\text{ for some }q_0\in Q\}
\end{equation*}
let $\hat E(y,q)\in\R^d$ be the components of $\exp_y^{-1}(q)$ in the coordinate frame at $y$. This is a smooth map with $\hat E(y,y)=0$, $\partial_q\hat E(y,y)=I$ and, differentiating the first identity along the diagonal, $\partial_y\hat E(y,y)=-I$. By compactness there are $C_1\ge1$ and $\delta_1>0$ such that every pair $y,q\in O$ with $y$ within coordinate distance $\delta_1$ of $Q$ and $|y-q|\le\delta_1$ lies in $\mathcal D$, that $\|\hat E\|_{C^{m+1}(\mathcal D)}\le C_1$, and that
\begin{equation}\label{eqcm:eq:Ehat}
 d_{\mathsf g}(y,q)\le C_1|y-q|,\qquad\|\partial_q\hat E(y,q)-I\|+\|\partial_y\hat E(y,q)+I\|\le C_1|y-q|
\end{equation}
for these pairs; and $|y-q|\le C_1d_{\mathsf g}(y,q)$ on $\mathcal D$.

Let $C_{\mathrm{cm}}$ be large; we may assume $\eta>0$. The points $\xi(g,\zeta_0)$ lie within coordinate distance $\eta\le\delta_1$ of $\xi_0(\zeta_0)\in Q$, hence in $B_{\mathsf g}(\xi_0(\zeta_0),2C_1\eta)$, and $2C_1\eta\le r_{\mathrm{cm}}$. By (a) the centre of mass $\bar\xi(\zeta_0)$ exists and lies in that ball, which gives the first estimate. By continuity and compactness of $\mathcal A$ the same holds near $\zeta_0$. Near $(\bar\xi(\zeta_0),\zeta_0)$ put
\begin{equation*}
 F(y,\zeta):=\int_{\mathcal A}\varpi(g,\zeta)\,\hat E\bigl(y,\xi(g,\zeta)\bigr)\,d\mu(g),\qquad F_0(y,\zeta):=\hat E\bigl(y,\xi_0(\zeta)\bigr),\qquad F_1:=F-F_0 ;
\end{equation*}
here $F_0=\int_{\mathcal A}\varpi\,\hat E(y,\xi_0)\,d\mu$ because $\int_{\mathcal A}\varpi\,d\mu=1$. By the chain rule and the mean value theorem, the $(y,\zeta)$-derivatives of order at most $m$ of $\hat E(y,\xi(g,\zeta))-\hat E(y,\xi_0(\zeta))$ at $\zeta=\zeta_0$ are bounded by $C\eta$, with $C$ depending on $C_1$, $m$ and $A$. The Leibniz rule and the bound on the total variation of $\varpi$ give $\|D^\alpha F_1(y,\zeta_0)\|\le C\eta$ for $|\alpha|\le m$ and $y$ near $\bar\xi(\zeta_0)$.

Write $y=\xi_0(\zeta)+w$. The map $\check F_0(w,\zeta):=F_0(\xi_0(\zeta)+w,\zeta)$ vanishes at $w=0$, and its derivative in $w$ there is $-I$. Hence $\check F_0(w,\zeta)=-T(w,\zeta)w$ with $T(w,\zeta):=-\int_0^1\partial_w\check F_0(tw,\zeta)\,dt$, which is of class $C^m$ with $C^m$ norm bounded by a constant depending on $C_1$, $m$ and $A$; by \eqref{eqcm:eq:Ehat}, $\|T-I\|\le C_1|w|$. The equation $F(\bar\xi(\zeta),\zeta)=0$ becomes
\begin{equation}\label{eqcm:eq:w}
 T\bigl(w(\zeta),\zeta\bigr)\,w(\zeta)=F_1\bigl(\xi_0(\zeta)+w(\zeta),\zeta\bigr),\qquad w:=\bar\xi-\xi_0 .
\end{equation}
The derivative in $w$ of $T(w,\zeta)w-F_1(\xi_0(\zeta)+w,\zeta)$ is $I+O(C_1|w|+C\eta)$, so the implicit function theorem gives a $C^m$ solution near $\zeta_0$, which by continuity and the uniqueness in (a) is the centre of mass. We show $\|D^jw(\zeta_0)\|\le C_j\eta$ by induction on $j\le m$; the case $j=0$ is the first estimate. Differentiating \eqref{eqcm:eq:w} $j$ times at $\zeta_0$ gives
\begin{equation*}
 \bigl(T+(\partial_wT)[\,\cdot\,]\,w-\partial_yF_1\bigr)D^jw=Q_j ,
\end{equation*}
where $Q_j$ is a sum of terms of two kinds. A term of the first kind contains a derivative of $F_1$ of order at most $j$, multiplied by derivatives of $\xi_0+w$ of order less than $j$ and by $D^j\xi_0$. A term of the second kind contains a derivative of $T$, multiplied by a product of derivatives of $w$ of orders between $0$ and $j-1$ with at least one factor present. By the bound on $F_1$ and the induction hypothesis, each term is at most $C\eta$. The operator on the left is $I+O(\eta)$, hence invertible with inverse of norm at most $2$, and $\|D^jw(\zeta_0)\|\le C_j\eta$ follows.

If the components of $\mathsf g$ and $\mathsf g^{-1}$ are bounded by $\Lambda$ in $C^{m+3}(\kappa(O))$ and $\kappa(Q)$ has distance at least $\delta$ from $\partial\kappa(O)$, then the curvature bound $\Delta$, the convexity radius, and the constants $C_1$ and $\delta_1$ above are controlled by $(d,m,\Lambda,\delta)$, and so is $r_{\mathrm{cm}}$; $C_{\mathrm{cm}}$ is controlled by $(d,m,A,\Lambda,\delta)$.
\end{proof}

\subsection{Local Equivariant Comparison}\label{eqcm:sec:comparison}

We now make an almost equivariant comparison map exactly equivariant for the elements of small displacement. The hypotheses are collected first.

\begin{definition}[Comparison Data]\label{eqcm:def:data}
Let $n,m\ge1$, $L\ge1$, $\eta>0$ and $R_0\ge100$, and let $\epsilon_0$ be a constant in $(0,1/100]$, depending on $n$ and $m$, that is fixed in the proof of \cref{eqcm:prop:comparison}. \emph{Comparison data} consist of the following.
\begin{enumerate}[label=\textup{(D\arabic*)}]
\item\label{eqcm:D1} Riemannian $n$-manifolds $(M_1,\mathsf g_1)$ and $(M_2,\mathsf g_2)$ with points $u_1,u_2$, and charts $\kappa_1$ on $B_{R_0}(u_1)$ and $\kappa_2$ on $B_{2R_0}(u_2)$ with $\kappa_j(u_j)=0$, in which the components of $\mathsf g_j$ are $\epsilon_0$-close to $\delta_{ab}$ in $C^{m+3}$.
\item\label{eqcm:D2} A unimodular Lie group $\mathsf G$, with Haar measure $\mu_{\mathsf G}$, acting properly and isometrically on $M_2$, and a homomorphism $\rho$ from $\mathsf G$ to the diffeomorphism group of $M_1$, such that $\rho(g)u$ depends continuously on $(g,u)$ and smoothly on $u$, with $u$-derivatives continuous in $(g,u)$. We write $|g|:=d_2(gu_2,u_2)$ and $\mathsf G(t):=\{g\in\mathsf G:|g|\le t\}$. For $g\in\mathsf G(R_0/4)$, the map $\rho(g)$ sends $B_{R_0/4}(u_1)$ into $B_{R_0}(u_1)$, and the derivatives of orders $1$ to $m+1$ of $\kappa_1\circ\rho(g)\circ\kappa_1^{-1}$ are bounded by $L$ on $\kappa_1(B_{R_0/4}(u_1))$.
\item\label{eqcm:D3} A smooth map $\ell\colon B_{R_0}(u_1)\to M_2$ with $d_2(\ell(u_1),u_2)\le\eta$, with $|d_2(\ell(a),\ell(b))-d_1(a,b)|\le\eta$ for $a,b\in B_{R_0}(u_1)$, and with $\|\kappa_2\circ\ell\circ\kappa_1^{-1}-\mathfrak a\|_{C^{m+1}}\le\eta$ on $\kappa_1(B_{R_0}(u_1))$ for an affine isometry $\mathfrak a$ of $\R^n$.
\item\label{eqcm:D4} For every $g\in\mathsf G(R_0/4)$, the map $\xi_g(u):=g^{-1}\ell(\rho(g)u)$ is within distance $\eta$ of $\ell$ on the ball $B_{R_0/4}(u_1)$, and
\begin{equation*}
 \bigl\|\kappa_2\circ\xi_g\circ\kappa_1^{-1}-\kappa_2\circ\ell\circ\kappa_1^{-1}\bigr\|_{C^m(\kappa_1(B_{R_0/4}(u_1)))}\le\eta .
\end{equation*}
\end{enumerate}
\end{definition}

Hypothesis \ref{eqcm:D4} is almost equivariance of $\ell$ for the elements of displacement at most $R_0/4$. It says nothing about larger elements, so $\rho$ may have a kernel. The first consequence is that $\ell$ converts displacements in $M_2$ into displacements in $M_1$.

\begin{lemma}[Location]\label{eqcm:lem:location}
Under \ref{eqcm:D3} and \ref{eqcm:D4}, for $g\in\mathsf G(R_0/4)$ and $u\in B_{R_0/4}(u_1)$,
\begin{equation*}
 \bigl|d_1(\rho(g)u,u_1)-d_2\bigl(g\ell(u),\ell(u_1)\bigr)\bigr|\le2\eta,\qquad\bigl|d_1(\rho(g)u,u)-d_2\bigl(g\ell(u),\ell(u)\bigr)\bigr|\le2\eta .
\end{equation*}
Moreover $|g|\le d_2(gp,p)+2d_2(p,u_2)$ for every $g\in\mathsf G$ and $p\in M_2$.
\end{lemma}

\begin{proof}
Both $\rho(g)u$ and $u_1$ lie in $B_{R_0}(u_1)$, so \ref{eqcm:D3} gives $|d_1(\rho(g)u,u_1)-d_2(\ell(\rho(g)u),\ell(u_1))|\le\eta$. Since $g$ is an isometry, $d_2(\ell(\rho(g)u),g\ell(u))=d_2(\xi_g(u),\ell(u))\le\eta$ by \ref{eqcm:D4}. The two bounds give the first estimate, and the second is the same with $u$ in place of $u_1$. The last inequality is the triangle inequality
\begin{equation*}
 d_2(gu_2,u_2)\le d_2(gu_2,gp)+d_2(gp,p)+d_2(p,u_2).
\end{equation*}
\end{proof}

For $p\in M_2$ and $t>0$ write $\mathsf G_p(t):=\{g\in\mathsf G:d_2(gp,p)<t\}$. The next proposition corrects an almost equivariant comparison on a ball of radius $\lambda$ to an exactly equivariant one, for all elements of displacement at most $5\lambda$, without averaging over the whole group; elements of the kernel of $\rho$ of larger displacement play no role.

\begin{proposition}[Local Equivariant Comparison]\label{eqcm:prop:comparison}
Let comparison data be given, let $\lambda\ge1$ with $100\lambda\le R_0$, and assume $\eta\le\bar\eta(n,m,L)$. Then there is a smooth map $\bar\ell\colon B_\lambda(u_1)\to M_2$ with the following properties.
\begin{enumerate}[label=\textup{(\alph*)}]
\item $\bar\ell(\rho(g_0)u)=g_0\bar\ell(u)$ whenever $g_0\in\mathsf G(5\lambda)$, $u\in B_\lambda(u_1)$ and $\rho(g_0)u\in B_\lambda(u_1)$.
\item $d_2(\bar\ell(u),\ell(u))\le C\eta$, and $\|\kappa_2\circ\bar\ell\circ\kappa_1^{-1}-\kappa_2\circ\ell\circ\kappa_1^{-1}\|_{C^m}\le C\eta$ on $\kappa_1(B_\lambda(u_1))$, with $C=C(n,m,L)$.
\item The restriction of $\bar\ell$ to $B_{\lambda/2}(u_1)$ is a diffeomorphism onto an open subset of $M_2$.
\end{enumerate}
\end{proposition}

\begin{proof}
\emph{Choices.} We fix $\epsilon_0=\epsilon_0(n,m)$ here, so small that $C\epsilon_0\le1/30$ for the distortion constant $C=C(n)$ of the charts below and that the doubling constant used in Step~4 is at most $2^{n+1}$; and we take $\bar\eta(n,m,L)\le1/6$ below the finitely many further bounds on $\eta$ used in Steps~2 to~5 and in \cref{eqcm:lem:cm}(b), where $C_{\mathrm{cm}}\eta\le1$. Let $\omega\colon M_1\to[0,1]$ be a smooth function of the norm of $\kappa_1$ on $B_{R_0}(u_1)$, extended by zero, that equals $1$ where $|\kappa_1|\le2.1\lambda$ and $0$ where $|\kappa_1|\ge2.9\lambda$, with derivatives of order $j$ bounded by $C(j)$; this is possible since $\lambda\ge1$. Since $\kappa_1$ distorts distances from $u_1$ by a factor $1\pm C\epsilon_0$, with $C=C(n)$, we have $\omega=1$ on $B_{2\lambda}(u_1)$, $\operatorname{supp}\omega\subseteq B_{3\lambda}(u_1)$, and $\|\omega\circ\kappa_1^{-1}\|_{C^m}\le C(n,m)$. Let $\mathcal A:=\mathsf G(20\lambda)$, which is compact because the action on $M_2$ is proper. Let $\bar\ell(u)$, for $u\in U:=B_\lambda(u_1)$, be the map of \cref{lem:reindex} with $\mathcal X=M_1$, $\rho_M$ the action on $M_2$, this $\mathcal A$, one chart $\ell$ and the weight $\omega$; that is, the centre of mass of the image of
\begin{equation*}
 \varpi(g,u)\,d\mu_{\mathsf G}(g),\qquad \varpi(g,u):=\frac{\omega(\rho(g)u)}{Z(u)},\qquad Z(u):=\int_{\mathcal A}\omega(\rho(g)u)\,d\mu_{\mathsf G}(g),
\end{equation*}
on $\mathcal A$ under $g\mapsto\xi_g(u)=g^{-1}\ell(\rho(g)u)$. Since $20\lambda\le R_0/5$, all of this is defined by \ref{eqcm:D2}.

\emph{Step 1: The Weights Live on Small Elements.} Let $u\in U$ and $g\in\mathsf G(25\lambda)$ with $\omega(\rho(g)u)>0$. Then $\rho(g)u\in B_{3\lambda}(u_1)$, and \cref{eqcm:lem:location} and \ref{eqcm:D3} give
\begin{equation*}
 d_2\bigl(g\ell(u),\ell(u_1)\bigr)\le3\lambda+2\eta,\qquad d_2\bigl(\ell(u),\ell(u_1)\bigr)\le\lambda+\eta .
\end{equation*}
This yields $d_2(g\ell(u),\ell(u))<5\lambda$, and then, by the last part of \cref{eqcm:lem:location} with $p=\ell(u)$ and $d_2(\ell(u),u_2)\le\lambda+2\eta$,
\begin{equation}\label{eqcm:eq:small}
 g\in \mathsf G_{\ell(u)}(5\lambda)\cap\mathsf G(8\lambda).
\end{equation}
Conversely, if $g\in \mathsf G_{\ell(u)}(\lambda/2)$, then $|g|\le3\lambda$, so $g\in\mathcal A$; by \cref{eqcm:lem:location}, the distance from $\rho(g)u$ to $u_1$ is less than $\lambda/2+\lambda+3\eta$, which is less than $2\lambda$, so $\omega(\rho(g)u)=1$. Hence
\begin{equation}\label{eqcm:eq:Z}
 \mu_{\mathsf G}\bigl(\mathsf G_{\ell(u)}(\lambda/2)\bigr)\le Z(u)\le\mu_{\mathsf G}\bigl(\mathsf G_{\ell(u)}(5\lambda)\bigr),
\end{equation}
and $Z(u)>0$, since $\mathsf G_{\ell(u)}(\lambda/2)$ is a nonempty open set.

\emph{Step 2: $\bar\ell$ Is Defined.} By \ref{eqcm:D4}, every point $\xi_g(u)$ with $g\in\mathcal A$ lies within $\eta$ of $\ell(u)$, which lies in the ball $B_{\lambda+2\eta}(u_2)$, deep inside the domain of $\kappa_2$. \Cref{eqcm:lem:cm}(a), with $Q:=\bar B_{2\lambda}(u_2)$, applies with $r_{\mathrm{cm}}$ depending only on $n$ and $m$, by its last sentence; so, for $\eta$ small, $\bar\ell(u)$ is defined and lies within $\eta$ of $\ell(u)$.

\emph{Step 3: Exact Equivariance.} Let $g_0\in\mathsf G(5\lambda)$ and $u,\rho(g_0)u\in U$. Every element of $\mathcal A\cup\mathcal Ag_0$ has displacement at most $25\lambda$, since $|hg_0|\le|h|+|g_0|$. If such an element $g$ satisfies $\omega(\rho(g)u)>0$, then $g\in\mathsf G(8\lambda)\subseteq\mathcal A$ by \eqref{eqcm:eq:small}, and $|gg_0^{-1}|\le13\lambda$, so $g\in\mathcal Ag_0$ as well. Thus the weights at $u$ vanish on the symmetric difference of $\mathcal A$ and $\mathcal Ag_0$, and \cref{lem:reindex}, together with Step~2 at $u$ and at $\rho(g_0)u$, gives (a).

\emph{Step 4: $C^m$ Estimates.} We apply \cref{eqcm:lem:cm}(b) at each point of $\kappa_1(U)$, with the chart $\kappa_2$, $Q:=\bar B_{2\lambda}(u_2)$, the compact set $\mathcal A$ with $\mu_{\mathsf G}$, the weights $\varpi$ and the family $\xi_g$, $g\in\mathcal A$, written in the coordinates $\kappa_1$, and $\xi_0:=\ell$. By the last sentence of that lemma, its constant depends only on $(n,m)$ and on the bound $A$ for the total variation obtained below. The closeness hypotheses are \ref{eqcm:D4}, and the derivatives of $\kappa_2\circ\ell\circ\kappa_1^{-1}$ are bounded by $2$ by \ref{eqcm:D3}. It remains to bound the total variation of the weights. By \ref{eqcm:D2} and the chain rule, $|D^j_u\,\omega(\rho(g)u)|\le C(n,m,L)$ for $j\le m$. This derivative vanishes unless $\rho(g)u\in\operatorname{supp}\omega$, which for $g\in\mathcal A$ forces $g\in \mathsf G_{\ell(u)}(5\lambda)$ by \eqref{eqcm:eq:small}. The derivatives of $Z$ satisfy the same bound and may be computed under the integral sign, as $\mathcal A$ is compact. By the quotient rule and \eqref{eqcm:eq:Z},
\begin{equation*}
 \int_{\mathcal A}\bigl|D^j_u\varpi(g,u)\bigr|\,d\mu_{\mathsf G}(g)\le C(n,m,L)\Bigl(\frac{\mu_{\mathsf G}(\mathsf G_{\ell(u)}(5\lambda))}{\mu_{\mathsf G}(\mathsf G_{\ell(u)}(\lambda/2))}\Bigr)^{j+1}.
\end{equation*}
The group acts on $M_2$ by measure-preserving isometries, and by \ref{eqcm:D1} the Riemannian volume of $M_2$ is doubling at $\ell(u)$ for radii up to $20\lambda$, with constant $C_{\mathrm{dbl}}\le2^{n+1}$. The open set $\mathsf G_{\ell(u)}(5\lambda)$ lies in the closed displacement set of radius $5\lambda$, and $\mathsf G_{\ell(u)}(\lambda/2)$ contains the closed one of radius $\lambda/4$. So \cref{lem:orbitcount}(a), at $p=\ell(u)$ with $\bar R=20\lambda$, $R=5\lambda$ and $r=\lambda/4$, bounds the ratio by $C_{\mathrm{dbl}}^{\,9}\le2^{9(n+1)}$. \Cref{eqcm:lem:cm}(b) now gives (b), with $C=C(n,m,L)$.

\emph{Step 5: Injectivity.} By (b) with $m\ge1$ and \ref{eqcm:D3}, the coordinate expression $\bar\ell_\kappa:=\kappa_2\circ\bar\ell\circ\kappa_1^{-1}$ satisfies $\|DP-D\mathfrak a\|\le(C+1)\eta<1/2$ on $\kappa_1(U)$, which contains the convex hull of $\kappa_1(B_{\lambda/2}(u_1))$. For $v,v'$ in that set, $|\bar\ell_\kappa(v)-\bar\ell_\kappa(v')-D\mathfrak a(v-v')|\le\tfrac12|v-v'|$, so $\bar\ell_\kappa(v)\ne \bar\ell_\kappa(v')$ if $v\ne v'$. Since $DP$ is invertible, $\bar\ell_\kappa$ is a local diffeomorphism, which gives (c).
\end{proof}

\subsection{Invariant Metrics}\label{eqcm:sec:metric}

The same averaging, now over the whole group, makes a proper action isometric for a metric that is close to a given one near a point, as soon as the action is almost isometric for the elements that move that point a bounded distance. We write $d_0$ and $B^0_t(u)$ for the distance and the balls of a metric $\mathsf g_0$, and $\omega$ for the cutoff of the proof of \cref{eqcm:prop:comparison}, built with the chart $\kappa$ below.

\begin{proposition}[Invariant Metrics]\label{eqcm:prop:metric}
Let $n,m\ge1$, $L\ge1$, $\lambda\ge1$ and $\eta\in(0,1/10]$. Let a unimodular Lie group $\mathsf G$, with Haar measure $\mu_{\mathsf G}$, act smoothly and properly on an $n$-manifold $M$ by $\rho$, let $\mathsf g_0$ be a Riemannian metric on $M$ and $u_1\in M$, and assume the following.
\begin{enumerate}[label=\textup{(I\arabic*)}]
\item\label{eqcm:I1} There is a chart $\kappa$ on $B^0_{40\lambda}(u_1)$ with $\kappa(u_1)=0$, in which the components of $\mathsf g_0$ are $\epsilon_0$-close to $\delta_{ab}$ in $C^{m+3}$, with $\epsilon_0$ as in \cref{eqcm:def:data}.
\item\label{eqcm:I2} For every $u\in B^0_\lambda(u_1)$ and every $g\in\mathsf G$ with $d_0(\rho(g)u,u)\le16\lambda$, the maps $\rho(g)$ and $\rho(g)^{-1}$ send $B^0_{16\lambda}(u)$ into $B^0_{40\lambda}(u_1)$, are $(1+\eta)$-bi-Lipschitz there, and have Jacobian determinant in $[1-\eta,1+\eta]$ with respect to $\mathsf g_0$. Moreover, on $\kappa(B^0_\lambda(u_1))$ the derivatives of orders $1$ to $m+1$ of $\kappa\circ\rho(g)\circ\kappa^{-1}$ are bounded by $L$, and the components of $\rho(g)^*\mathsf g_0-\mathsf g_0$ are bounded by $\eta$ in $C^m$.
\end{enumerate}
Put $Z(u):=\int_{\mathsf G}\omega(\rho(g)u)\,d\mu_{\mathsf G}(g)$, $\Omega:=\{Z>0\}$, and, for $u\in\Omega$,
\begin{equation*}
 \tilde{\mathsf g}_u:=\frac{1}{Z(u)}\int_{\mathsf G}\omega\bigl(\rho(g)u\bigr)\,\bigl(\rho(g)^*\mathsf g_0\bigr)_u\,d\mu_{\mathsf G}(g).
\end{equation*}
Then the following hold.
\begin{enumerate}[label=\textup{(\alph*)}]
\item $\Omega=\rho(\mathsf G)\cdot\{\omega>0\}$ is open and $\rho(\mathsf G)$-invariant, and $\tilde{\mathsf g}$ is a smooth Riemannian metric on $\Omega$ for which $\rho(\mathsf G)$ acts by isometries.
\item The components of $\tilde{\mathsf g}-\mathsf g_0$ in $\kappa$ are bounded by $C(n,m,L)\,\eta$ in $C^m$ on $\kappa(B^0_\lambda(u_1))$.
\item For every $\rho(\mathsf G)$-invariant Riemannian metric $\mathsf g_*$ on $M$ there is a $\rho(\mathsf G)$-invariant Riemannian metric $\hat{\mathsf g}$ on $M$ that equals $\tilde{\mathsf g}$ on $\rho(\mathsf G)\cdot B^0_\lambda(u_1)$.
\end{enumerate}
\end{proposition}

\begin{proof}
\emph{The Metric.} Since the action is proper and $\operatorname{supp}\omega$ is compact, the set of elements $g$ that move $u$ into $\operatorname{supp}\omega$ is compact, locally uniformly in $u$. Hence $Z$ and the numerator are smooth, $\Omega$ is open, and $\Omega=\rho(\mathsf G)\cdot\{\omega>0\}$. The form $\tilde{\mathsf g}_u$ is an average of positive definite forms with nonnegative weights of total mass one, so it is positive definite. For $g_0\in\mathsf G$, right invariance of $\mu_{\mathsf G}$ gives $Z(\rho(g_0)u)=Z(u)$, and, since $\rho(g)\rho(g_0)=\rho(gg_0)$,
\begin{equation*}
 \bigl(\rho(g_0)^*\tilde{\mathsf g}\bigr)_u=\frac{1}{Z(u)}\int_{\mathsf G}\omega\bigl(\rho(gg_0)u\bigr)\,\bigl(\rho(gg_0)^*\mathsf g_0\bigr)_u\,d\mu_{\mathsf G}(g)=\tilde{\mathsf g}_u .
\end{equation*}
This proves (a).

\emph{Closeness.} Let $u\in B^0_\lambda(u_1)$ and $\mathsf G^0_u(t):=\{g\in\mathsf G:d_0(\rho(g)u,u)\le t\}$. A weight $\omega(\rho(g)u)$ is nonzero only if $\rho(g)u\in B^0_{3\lambda}(u_1)$, hence only if $g\in \mathsf G^0_u(4\lambda)$, and for these $g$ hypothesis \ref{eqcm:I2} applies. If $g\in \mathsf G^0_u(\lambda/2)$, then $\rho(g)u\in B^0_{2\lambda}(u_1)$ and the weight is $1$. Thus ${\mu_{\mathsf G}(\mathsf G^0_u(\lambda/2))\le Z(u)\le\mu_{\mathsf G}(\mathsf G^0_u(4\lambda))}$. As in Step~4 of the proof of \cref{eqcm:prop:comparison}, the total variation of the derivatives of order at most $m$ of the normalized weights is bounded by $C(n,m,L)$ times a power of the ratio of these two measures. By \ref{eqcm:I1}, the volume of $\mathsf g_0$ is doubling at $u$ for radii up to $16\lambda$, with constant $C_{\mathrm{dbl}}\le2^{n+1}$, and \ref{eqcm:I2} gives the remaining hypotheses of \cref{lem:orbitcount}(b) at $p=u$ with $\bar R=16\lambda$. With $R=4\lambda$ and $r=\lambda/2$ it bounds the ratio by $(1-\eta)^{-1}C_{\mathrm{dbl}}^{\,8}$. Since the weights have total mass one,
\begin{equation*}
 \tilde{\mathsf g}_u-(\mathsf g_0)_u=\frac{1}{Z(u)}\int_{\mathsf G}\omega\bigl(\rho(g)u\bigr)\,\bigl(\rho(g)^*\mathsf g_0-\mathsf g_0\bigr)_u\,d\mu_{\mathsf G}(g),
\end{equation*}
and the Leibniz rule, \ref{eqcm:I2} and the bound on the total variation give (b).

\emph{Extension.} The closed ball $\bar B^0_\lambda(u_1)$ is compact and $Z$ is positive and continuous on it, so
\begin{equation*}
 c_Z:=\min_{\bar B^0_\lambda(u_1)}Z>0 .
\end{equation*} Let $F\colon[0,\infty)\to[0,1]$ be smooth, with $F=0$ on $[0,1/4]$ and $F=1$ on $[1/2,\infty)$. The function $\omega_Z:=F(Z/c_Z)$ is smooth and $\rho(\mathsf G)$-invariant, it equals $1$ on $\rho(\mathsf G)\cdot B^0_\lambda(u_1)$, and it vanishes on the open set $\{Z<c/4\}$, which contains $M\setminus\Omega$. Then $\hat{\mathsf g}:=\omega_Z\tilde{\mathsf g}+(1-\omega_Z)\mathsf g_*$ has the properties in (c).
\end{proof}

\section{Notation and Constants}\label{app:notation}
This appendix lists the notation of the paper and explains in which order its parameters are fixed. \Cref{app:order} gives the order of choices; \cref{tab:objects,tab:coords} index the spaces, groups, maps and coordinates, and \cref{tab:constants} the parameters and constants, with what each depends on and where it is fixed. Symbols that occur only inside one statement or proof are not listed. Scalar constants $C,c$ depend only on $N$, and $\Psi$ is the function of the Conventions in \cref{ssec:setup}; a constant that depends on more carries that dependence in its notation. Euclidean Gromov--Hausdorff approximations in statements and proofs are written $\mathcal E$, and $\mathbb D$ always denotes a Euclidean ball in $\R^m$.

\subsection{Order of Choices}\label{app:order}
Every estimate in the paper is local and scale-invariant, so each constant either depends only on $N$ or is chosen after the data it depends on. The order is as follows.
\begin{enumerate}[label=\textup{(\arabic*)}]
\item The data $N,K,D,m,k$ and $b$ of $(\ast)$ are given. The subsequence, the limit $(Y,\hat p,H)$, the splitting $H=H_0\times\Z^b\times C$, the identification $H_0\cong\R^k$ and the approximations $(\hat f_i,\phi_i,\psi_i)$ are fixed once (\cref{ssec:setup,ssec:CQ}).
\item Thresholds depending only on $N$ come next: $\eta_0(N)$, $\bar\epsilon(N)$, $\delta_c(N)$, $\delta_0(N,1)$, $\bar\psi(N)$ and the absolute constant $\mathsf M$.
\item The accuracies $\epsilon=\epsilon(N)$ and $\eta=\eta(N)$ are fixed in \cref{ssec:setG} so that \ref{S1}--\ref{S4} hold. They never change afterwards, and every later modulus $\Psi$ in \cref{sec:local} is a number depending only on $N$.
\item For each $y\in\mathcal R(Y)$ the radii $r_0(y,\eta)$, $r_\epsilon(y)$ and $r(y)$, and for each $w\in W$ the radius $r_\Gamma(w)$, are chosen. They are positive but need not be uniform in $y$; uniformity is used only under (R), where it follows from compactness (proof of \cref{lem:Kchart}).
\item For each $x\in G$ the data $y_0$, $w'$, $r=r(y_0)$ and $r_*$ of \cref{ssec:fibrecoords} are fixed.
\item Only then is $i$ taken large: every \emph{for large $i$} and every $o(1)$ refers to $i\to\infty$ with all the above fixed.
\end{enumerate}
Two constructions run auxiliary chains of their own. In \cref{sec:norot}, an accuracy $(R,\varsigma)$ of \cref{lem:linearmodel} fixes $\epsilon'$, $\epsilon_B$, $\delta_B$ and $\delta$, depending only on $N$, $R$ and $\varsigma$; then $\eta_\delta$, $r_\delta$ and $S$ depend also on $y$, the harmonic chart $u$ and its constants on these data, and finally $r_{R,\varsigma}(y)<S/(80R)$. In \cref{ssec:Kproduct,sec:seifert}, under the hypotheses of \cref{thm:master}(v), the parameters are fixed in the order $\lambda_*$, $\mathsf M$, $r_{\mathrm f}$ and $r_{W'}$; the constants $L_Z$, $\delta_R$, $\delta$, $\delta_e$ and $\mu$ of the proof of \cref{prop:globalchart}; then $\psi$, with $\varsigma(\psi,N)$ and $r_1(\psi)$; then $\tau_0$, $\tau\le\tau_0$, $s=s(N,W',C,\tau)$, $\rho=\mu s$ and $r=\mathsf M\rho$; and finally $i$.

\subsection{Index of Notation}\label{app:index}
\Cref{tab:objects} lists the global objects, \cref{tab:coords} the coordinates and the objects of the local constructions, and \cref{tab:constants} the parameters and constants.

\begin{table}[tbp]
\centering\small
\begin{tabular}{@{}>{\raggedright}p{3.2cm}>{\raggedright}p{8.2cm}>{\raggedright\arraybackslash}p{3.1cm}@{}}
\toprule
Symbol & Meaning & Introduced\\
\midrule
$X_i$, $X$, $d_i$, $\mathfrak m_i$, $c_i$ & the spaces of $(\ast)$, their limit, metrics, measures; $\mathfrak m_i=c_i\mathcal H^N$ & $(\ast)$, \cref{thm:master}\\
$\hat X_i$, $\hat p_i$, $H_i$ & maximal abelian covers, base points, deck groups $H_1(X_i;\Z)$ & \cref{ssec:setup}\\
$(Y,\hat p,H)$ & equivariant limit of $(\hat X_i,\hat p_i,H_i)$ & \cref{ssec:setup}\\
$H_0\cong\R^k$, $\varphi_v$ & identity component of $H$; its free action, $v\in H_0$ & \cref{thm:master}, \cref{sec:norot}\\
$C$ & the finite abelian factor of $H\cong\R^k\times\Z^b\times C$ & \cref{thm:master}\\
$\Gamma$, $\Gamma_{\mathrm f}$, $\Z^b$ & $H/H_0\cong\Z^b\times C$; the free complement of $C$ in $\Gamma$; its lift to $H$ & \cref{thm:master}, \cref{ssec:CQ}\\
$W$, $P_W$, $W'$, $P'$, $P$ & $Y/H_0$; $Y\to W$; $W/\Gamma_{\mathrm f}=Y/\mathsf K$; $Y\to W'$; $Y\to X$ & \cref{ssec:setup}, \cref{cor:Ccover}\\
$\mathcal R(Z)$, $\mathcal S(Z)$, (R) & regular and singular sets; no bubbling, $\mathcal R(Y)=Y$ & \cref{sec:intro}\\
$C_y$, $\Gamma_w$, $C_{w'}$, $C_x$ & $\Stab_H(y)$, $\Stab_\Gamma(w)$, $\Stab_C(w')$, the local group at $x$; all equal & \cref{lem:localgroups}\\
$\rho_y$ & faithful transverse representation $C_y\to O(m)$ & \cref{prop:blowup}, \cref{def:reps}\\
$G$ & the good set, an open orbifold in $X$ & \cref{ssec:setG}\\
$\Pi_i$, $H_i'$ & compatible quotient maps $H_i\to\Gamma$; fibre groups $\ker\Pi_i\cong\Z^k$ & \cref{prop:CQ}\\
$\chi$, $\chi_i$, $\mathsf K_i$, $\mathsf K$ & projection $\Gamma\to C$; $\chi\circ\Pi_i$; $\ker\chi_i$; $\pr^{-1}(\Z^b)$ & \cref{cor:Ccover}\\
$X_i'$, $\pi$, $q_i$, $q$ & $\hat X_i/\mathsf K_i$; projection $\hat X_i\to X_i'$; quotients $X_i'\to X_i$, $W'\to X$ & \cref{cor:Ccover}, \cref{ssec:fibrecoords}\\
$\Lambda_i$, $\lambda_c$ & stabilizer of a tube over a point of $X$; the lifts $\psi_i(c)$ of $c\in C$ & \cref{lem:localgroups}, \cref{lem:fibre}\\
$(\hat f_i,\phi_i,\psi_i)$, $\epsilon_i$ & equivariant approximations of the covers and their accuracy & \cref{ssec:eqconv}, \cref{ssec:CQ}\\
$f_i$, $f_i'$ & the induced approximations $X_i\to X$ and $X_i'\to W'$ & \cref{ssec:eqconv}, \cref{sec:local}\\
$\sigma_i$, $E_i$, $U_x$, $U_x^0$ & Seifert fibrations $E_i\to U_x$ (on all of $X_i$ in (v)); model neighbourhoods & \cref{def:seifert}, \cref{thm:master}\\
$\widetilde W$, $\widetilde X_i$ & the covers of \cref{thm:master}(v), realized as $W'$ and $X_i'$ & \cref{thm:seifert}\\
$\T^k$, $\T_i^k$, $L_i$ & the standard torus; $\R^k/L_i$; the lattice $\ell_i(H_i')$ & \cref{prop:lattice}\\
\bottomrule
\end{tabular}
\caption{Spaces, groups and maps that are used throughout the paper, with the place where each is introduced.}
\label{tab:objects}
\end{table}

\begin{table}[tbp]
\centering\small
\begin{tabular}{@{}>{\raggedright}p{3.2cm}>{\raggedright}p{8.2cm}>{\raggedright\arraybackslash}p{3.1cm}@{}}
\toprule
Symbol & Meaning & Introduced\\
\midrule
$\mathfrak d_z(v)$, $\mathfrak d_i(g)$ & displacements $d(\varphi_vz,z)$ on $Y$ and $d(g\hat y_i,\hat y_i)$ on $\hat X_i$ & \cref{sec:norot}, \cref{ssec:fibrecoords}\\
$u$, $h_v$, $a(v)$, $A$ & harmonic chart at $y$; cocycle $u\circ\varphi_v-u$; $h_v(y)$; its linear part & \cref{sec:norot}, \cref{lem:pointlinear}\\
$U_r$, $V_r$, $L_r$ & linear model at scale $r$: a GH approximation, a $k$-plane, an isomorphism $\R^k\to V_r$ & \cref{lem:linearmodel}\\
$\mathcal F=(\mathcal F',\mathcal F'')$, $\bar L_r$ & labelled approximation at scale $r$; linear part of its group component & \cref{lem:fixedlinear}, \cref{def:reps}\\
$\zeta$, $\hat\psi$, $h_y$, $\mathfrak b$ & harmonic chart; transverse map on $Y$; induced chart on $W$; a candidate transverse map & \cref{prop:U2}, \cref{ssec:setG}\\
$y_0$, $w_0$, $w_0'$, $x_i^0$, $\hat y_i$, $C_0$, $\rho_0$ & base data of \cref{sec:local}; $C_0=C_{y_0}$ and $\rho_0=\rho_{y_0}$ & \cref{ssec:eqsplit}\\
$\Phi_i$, $\Phi$, $\mathcal L$ & approximants of the transverse map on $X_i'$; a limit; the set of limits & \cref{prop:eqsplit}\\
$E_i'(\mathbb D)$, $E_i'$, $\sigma_i'$ & tubes $\Phi_i^{-1}(\mathbb D)\cap B(x_i^0,r)$ in $X_i'$; the map $\Phi^{-1}\circ\Phi_i$ & \cref{lem:bundlecrit}, \cref{ssec:proofiv}\\
$\ell_i$ & lattice coordinate $H_i'\to\R^k$, extended to $\Pi_i^{-1}(C)$ & \cref{prop:lattice}, \cref{cor:localgroup}\\
$\tilde T_i$, $\xi_i$, $U_i$ & tube in $\hat X_i$; harmonic chart at $\hat y_i$; orbit coordinate & \cref{lem:tubecover,lem:orbitcharts}, \cref{prop:orbit}\\
$\hat\Phi_i$, $\mathcal Z_i$ & $\Phi_i\circ\pi$; the product chart $(U_i,\hat\Phi_i)$ & \cref{thm:chart}\\
$\theta$, $\theta_i$ & monodromy homomorphisms of the covers $\tilde T_i\to E_i'$ and $\hat X_i\to X_i'$ & \cref{thm:chart}, \cref{lem:fibreid}\\
$\ell_i^{\mathsf K}$, $\ell_i^H$ & extensions of $\ell_i$ to $\mathsf K_i$ and to $H_i$ & \cref{ssec:Kproduct}\\
$\hat U$, $\bar U$, $\hat h$, $\bar h$ & global orbit coordinate on $\hat X_i$ and on $X_i'$; the maps $(F\circ\pi,\hat U)$ and $(F,\bar U)$ & \cref{prop:globalorbit}, \cref{prop:globalchart}\\
$F$, $A_i$, $h_i$, $\kappa_i$ & regularized base map; base maps $X_i'\to W'$; product homeomorphisms; translations $C\to\T^k$ & \cref{ssec:regularization}, \cref{thm:virtualproduct}\\
$\operatorname{cm}$ & centre of mass (the mean in a Euclidean space) & \cref{lem:reindex}, \cref{eqcm:lem:cm}\\
\bottomrule
\end{tabular}
\caption{Coordinates and the objects of the local and global constructions, with the place where each is introduced.}
\label{tab:coords}
\end{table}

\begin{table}[tbp]
\centering\small
\begin{tabular}{@{}>{\raggedright}p{2.6cm}>{\raggedright}p{5.6cm}>{\raggedright}p{3.1cm}>{\raggedright\arraybackslash}p{2.9cm}@{}}
\toprule
Symbol & Role & Depends on & Fixed in\\
\midrule
$N$, $K$, $D$, $m$, $k$, $b$ & data of $(\ast)$; $k=N-m$, $b=b_1(X)$ & --- & $(\ast)$\\
$C$, $c$; $\Psi$ & generic constants; generic modulus & $N$; its arguments & \cref{ssec:setup}\\
$\vartheta_s(z)$, $C_{\mathrm{dbl}}$ & Bishop--Gromov ratio; a doubling constant & --- & \cref{ssec:packages}, \cref{lem:orbitcount}\\
$\varkappa$ & bi-H\"older defect of a splitting chart & $\delta$, $N$, $L$ & \cref{lem:replacement}\\
$\eta_0$; $r_0(y,\eta)$ & Reifenberg threshold; radius of Euclidean control near $H_0y$ & $N$; $y$, $\eta$ & \cref{lem:tube}\\
$c_Q$, $c_Q'$ & lower bounds for the displacement & $Q$, $t_0$, the action & \cref{lem:linlower}\\
$\delta$; $\eta_\delta$, $r_\delta$, $S$ & accuracy of the harmonic chart; its tube data and radius, $512S\le\eta_\delta r_\delta$ & $N$, $R$, $\varsigma$; also $y$ & \cref{sec:norot}, \cref{lem:linearmodel}\\
$t_1$, $t_2$, $\mathsf L$, $\alpha$ & radii in $H_0$; Lipschitz modulus; exponent $1/(1+\varkappa)$ & the chart data & \cref{prop:lip}, \cref{lem:pointlinear}\\
$R$, $\varsigma$; $r_{R,\varsigma}(y)$ & radius and accuracy of the linear models; threshold scale & ---; $N$, $R$, $\varsigma$, $y$ & \cref{lem:linearmodel}\\
$\epsilon'$, $\epsilon_B$, $\delta_B$, $C_\nabla$, $C_T$ & accuracies and constants of the transformation theorem & $N$, $R$, $\varsigma$; $C_\nabla,C_T$: $N$ & proof of \cref{lem:linearmodel}\\
$\delta_\rho$; $\bar\epsilon$ & rigidity threshold for representations; linearization threshold & $N$, $|C_y|$; $N$ & \cref{lem:isorep}, \cref{lem:fixedlinear}\\
$\epsilon$, $\eta$ & accuracy of the labelled model; parameter of \cref{lem:tube} & $N$, through \ref{S1}--\ref{S4} & \cref{ssec:setG}\\
$r_\epsilon(y)$, $r(y)$, $r_\Gamma(w)$ & model scale; chart scale; separation radius at $w$ & $\epsilon$, $\eta$, $y$; $w$ & \cref{def:reps}, \cref{ssec:setG}\\
$\delta_c$, $\delta_0(N,L_\Delta)$ & thresholds of the completion and submersion lemmas & $N$; $N$, $L_\Delta$ & \cref{lem:completion}, \cref{lem:submersion}\\
$\bar\delta$, $R_*$, $R_2$ & displacement scale with $U_{\bar\delta}\subset H_0$; displacement of lifts; generating radius & $H$, $\hat p$, the approximations & \cref{ssec:CQ}, \cref{lem:fibre}\\
$r$, $r_*$ & chart radius $r(y_0)$; admissible radius at $x$ & $y_0$; $x$, $r$ & \cref{ssec:eqsplit}, \cref{ssec:fibrecoords}\\
$\omega_{\mathrm{inj}}$, $\omega_{\mathrm{op}}$ & injectivity and openness moduli of the limit charts & $\mathcal L$ & \cref{lem:unifcharts}\\
$\lambda_*$, $\mathsf M$, $C_{W'}$ & rescaling factor; $2\cdot10^4$; constants & $K$; ---; $N$, $W'$, $C$ & \cref{ssec:Kproduct}\\
$r_{\mathrm f}$ & radius below which $\mathsf K_i$ moves points only by $H_i'$ & $W$, the $\Z^b$-action & \cref{lem:Kbasic}\\
$\psi$, $r_1(\psi)$, $\bar r(R,\varsigma)$ & accuracy and scale threshold of orbit charts; uniform threshold & $\psi$, $N$, $Y$; $N$, $R$, $\varsigma$, $Y$ & \cref{lem:Kchart}\\
$\bar\psi$, $\Psi_2$ & threshold and error for comparing orbit charts & $N$; $\psi$, $N$ & \cref{lem:Kcompare}\\
$r_{W'}$, $\Psi_3$, $L_F$ & gluing radius; gluing error; Laplacian bound of the base map & $W'$; $\psi$, $N$, $W'$; --- & \cref{prop:globalorbit}\\
$\tau$, $s$, $\epsilon_{\mathcal W}$, $C_\Delta$ & regularization accuracy, scale, threshold; constant of the bounds & $N$, $\mathcal W$, $C$, $\tau$; $C_\Delta$: $N$ & \cref{ssec:regularization}\\
$L_Z$, $\delta_R$, $\delta$, $\delta_e$, $\mu$, $\tau_0$ & constants of the product chart; $\mu=\rho/s$ & $N$; $\tau_0$ also $W'$, $C$, $r_{\mathrm f}$, $Y$ & proof of \cref{prop:globalchart}\\
$\rho$, $r=\mathsf M\rho$, $\tau_i$ & gluing scale and chart scale; accuracies $\tau_i\downarrow0$ & $\tau$, through $s$ & \cref{prop:globalchart}, \cref{thm:virtualproduct}\\
\bottomrule
\end{tabular}
\caption{Parameters and constants: the role of each, what it depends on, and the result in which it is fixed. Appendix~\ref{app:eqcm} has its own constants $\epsilon_0$, $\bar\eta$, $r_{\mathrm{cm}}$ and $C_{\mathrm{cm}}$, stated there.}
\label{tab:constants}
\end{table}

\section*{Acknowledgments and Tool Disclosure}
\textbf{AI-use disclosure.} The authors used Anthropic Claude Code and OpenAI Codex as interactive research and writing
tools. These tools assisted with exploratory discussion, testing and refinement of ideas,
literature and source organization, code development and verification, mathematical error
checking, and editorial revision. The authors take full responsibility for the content
presented.

\printbibliography[heading=bibintoc]

\end{document}